\documentclass[12pt]{amsart}
\usepackage{amssymb,amsthm,amsfonts,amsmath}
\usepackage[all]{xy}
\usepackage{longtable}
\usepackage{multirow}
\usepackage{chngpage}
\usepackage{array}
\newcolumntype{P}[1]{>{\centering\arraybackslash}p{#1}}
\newcolumntype{M}[1]{>{\centering\arraybackslash}m{#1}}

\newtheorem{theorem}[equation]{Theorem}
\newtheorem*{theorem*}{Theorem}
\newtheorem{lemma}[equation]{Lemma}
\newtheorem{corollary}[equation]{Corollary}

\theoremstyle{definition}
\newtheorem{example}[equation]{Example}

\newtheorem*{definition*}{Definition}

\theoremstyle{remark}
\newtheorem{remark}[equation]{Remark}

\renewcommand\arraystretch{1.3}

\makeatletter\@addtoreset{equation}{section}

\makeatother

\usepackage{hyperref}

\newcommand{\FF}{\mathbb{F}}

\def \Z {\mathbb{Z}}

\def \P {\mathbb{P}}
\newcommand{\PP}{\mathbb{P}}

\newcommand{\Gm}{\Bbbk^{\times}}
\newcommand{\Ga}{\Bbbk^+}

\newcommand{\cO}{\mathcal{O}}

\newcommand{\Aut}{\operatorname{Aut}}
\newcommand{\GL}{\operatorname{GL}}
\newcommand{\SL}{\operatorname{SL}}
\newcommand{\PSO}{\operatorname{PSO}}
\newcommand{\SO}{\operatorname{SO}}

\newcommand{\PGL}{\operatorname{PGL}}

\newcommand{\Sym}{\operatorname{Sym}}

\newcommand{\rk}{\operatorname{rk}}
\newcommand{\BB}{\operatorname{B}}

\newcommand{\wt}{\operatorname{wt}}
\newcommand{\Pic}{\operatorname{Pic}}

\def \ge {\geqslant}
\def \le {\leqslant}
\def \geq {\geqslant}

\title{Fano threefolds with infinite automorphism groups}

\author{Ivan Cheltsov, Victor Przyjalkowski, and Constantin Shramov}

\address{\emph{Ivan Cheltsov}
\newline
\textnormal{School of Mathematics, The University of Edinburgh,  Edinburgh EH9 3JZ, UK.}
\newline
\textnormal{National Research University Higher School of Economics, Russian Federation, AG Laboratory, HSE, 6 Usacheva str., Moscow, 117312, Russia.
}
\newline
\textnormal{\texttt{I.Cheltsov@ed.ac.uk}}}

\address{\emph{Victor Przyjalkowski}
\newline
\textnormal{Steklov Mathematical Institute of Russian Academy of Sciences, 8 Gubkina street, Moscow 119991, Russia.}
\newline
\textnormal{National Research University Higher School of Economics, Russian Federation, Laboratory of Mirror Symmetry, NRU HSE, 6 Usacheva str., Moscow, Russia, 119048.}
\newline
\textnormal{\texttt{victorprz@mi.ras.ru, victorprz@gmail.com}}}

\address{\emph{Constantin Shramov}
\newline
\textnormal{Steklov Mathematical Institute of Russian Academy of Sciences, 8 Gubkina street, Moscow 119991, Russia.}
\newline
\textnormal{National Research University Higher School of Economics, Russian Federation, AG Laboratory, HSE, 6 Usacheva str., Moscow, 117312, Russia.}
\newline
\textnormal{\texttt{costya.shramov@gmail.com}}}

\begin{document}

\begin{abstract}
We classify smooth Fano threefolds with infinite automorphism groups.
\end{abstract}

\maketitle

\begin{flushright}
\begin{minipage}[c]{3in}
To the memory of Vasily Iskovskikh
\end{minipage}
\end{flushright}

\tableofcontents

\section{Introduction}
\label{section:intro}

One of the most important results obtained by Iskovskikh
is a classification of smooth Fano threefolds of Picard rank $1$ (see~\cite{Is77}, \cite{Is78}).
In fact, he was the one who introduced the notion of Fano variety.
Using Iskovskikh's classification, Mori and Mukai classified all smooth Fano threefolds of higher Picard ranks (see~\cite{MM82}, and also~\cite{MM04} for a minor revision).
Nowadays, Fano varieties play a central role in both algebraic and complex geometry,
and provide key examples for number theory and mathematical physics.

Let $\Bbbk$ be an algebraically
closed field of characteristic zero.
Automorphism groups of (smooth) Fano varieties
are important from the point of view of birational geometry, in particular
from the point of view of birational automorphism groups of rationally connected varieties.
There is not much to study in dimension $1$; the only smooth one-dimensional Fano variety is a projective line,
whose automorphism group is $\PGL_2(\Bbbk)$.
The automorphism groups of smooth two-dimensional Fano varieties (also known as del Pezzo surfaces) are already rather tricky.
However, the structure of del Pezzo surfaces is classically known, and their automorphism groups
are described in details (see \cite{DolgachevIskovskikh}).
As for smooth Fano threefolds of Picard rank $1$,
several non-trivial examples of varieties with infinite automorphism
groups were known, see \cite[Proposition~4.4]{Mukai-CurvesK3Fano}, \cite{MukaiUmemura}, and \cite{Prokhorov-1990c}.
Recently, the following result was obtained in~\cite{KuznetsovProkhorovShramov}.

\begin{theorem}[{\cite[Theorem~1.1.2]{KuznetsovProkhorovShramov}}]
\label{theorem:Prokhorov}
Let $X$ be a smooth
Fano threefold with Picard rank equal to~$1$. Then the group~\mbox{$\Aut(X)$} is finite
unless one of the following cases occurs.
\begin{itemize}
\item
The threefold
$X$ is the projective space $\P^3$, and $\Aut(X) \cong \PGL_4(\Bbbk)$.
\item
The threefold
$X$ is the smooth quadric $Q$ in $\P^4$, and $\Aut(X) \cong \PSO_5(\Bbbk)$.
\item
The threefold $X$ is the smooth section $V_5$ of the Grassmannian $\mathrm{Gr}(2,5)\subset\P^9$
by a linear subspace of dimension~$6$, and  $\Aut(X)\cong\PGL_2(\Bbbk)$.
\item
The threefold $X$ has Fano index $1$ and anticanonical degree $22$;
moreover, the following cases are possible here:
\begin{enumerate}
\item
$X=X_{22}^{\mathrm{MU}}$ is the Mukai--Umemura threefold, and
$\Aut(X)\cong \PGL_2(\Bbbk)$;
\item
$X=X_{22}^{\mathrm a}$ is the unique threefold of this type such that the connected component
of identity in $\Aut(X)$ is isomorphic to $\Ga$;
\item
$X=X_{22}^{\mathrm m}(u)$ is a threefold from a certain one-dimensional
family such that the connected component
of identity in $\Aut(X)$ is isomorphic to $\Gm$.
\end{enumerate}
\end{itemize}
\end{theorem}

Smooth Fano threefolds of Picard rank greater than~$1$
are more numerous than those with Picard rank $1$.
Some of these threefolds having large automorphism groups have been already
studied by different authors.
Namely, Batyrev classified all smooth toric Fano threefolds
in~\cite{Batyrev} (see also~\cite{Watanabe}).
S\"u{\ss} classified in \cite{Suess14}
all smooth Fano threefolds that admit a faithful
action of a two-dimensional torus.
Smooth Fano threefolds with a faithful action of $(\Ga)^3$ were classified in \cite{HM18} (cf. \cite[Theorem~6.1]{HT99}).
Threefolds with an action of the group~\mbox{$\SL_2(\Bbbk)$}
were studied in~\cite{MukaiUmemura}, \cite{Umemura}, \cite{Nakano89},
and~\cite{Nakano98}.
For some results on higher-dimensional Fano varieties with infinite automorphism groups see~\cite{FM18} and~\cite{FOX18}.

The goal of this paper is to provide
a classification similar to that given by Theorem~\ref{theorem:Prokhorov}
in the case of higher Picard rank.
Given a smooth Fano threefold $X$,
we identify it (or rather its deformation family) by
the pair of numbers
$$
\gimel(X)=\rho.N,
$$
where $\rho$ is the Picard
rank of the threefold $X$, and $N$ is its number in the classification tables
in~\cite{MM82}, \cite{IP99}, and~\cite{MM04}.
Note that the most complete list of smooth Fano threefolds is contained in~\cite{MM04}.

The main result of this paper is the following theorem.

\begin{theorem}
\label{theorem:main}
The following assertions hold.
\begin{itemize}
\item[(i)]
The group $\Aut(X)$ is infinite for every smooth Fano threefold $X$ with
\begin{multline*}
\gimel(X)\in\big\{1.15, 1.16, 1.17, 2.26,\ldots, 2.36, 3.9, 3.13,\ldots,3.31,\\
4.2,\ldots,4.12, 5.1, 5.2, 5.3, 6.1, 7.1, 8.1, 9.1, 10.1\big\}.
\end{multline*}

\item[(ii)]
There also exist smooth Fano threefolds $X$
such that the group $\Aut(X)$ is infinite if
$$
\gimel(X)\in\big\{1.10, 2.20, 2.21, 2.22, 2.24, 3.5, 3.8, 3.10, 3.12, 4.13\big\},
$$
while for a general threefold $X$ from these deformation families the group~\mbox{$\Aut(X)$} is finite.

\item[(iii)]
The group~$\Aut(X)$ is always finite when $X$ is contained
in any of the remaining~$42$ deformation families of smooth Fano threefolds.
\end{itemize}
\end{theorem}

In fact, we describe all connected components of the identity of automorphisms groups of all smooth Fano threefolds, see Table~\ref{table:big-table}.

For a smooth Fano threefold $X$, if the group $\Aut(X)$ is infinite, then $X$ is rational.
However, unlike the case of Picard rank $1$, the threefold $X$ may have a non-trivial
Hodge number~$h^{1,2}(X)$.
The simplest example is given by a blow up of $\P^3$ along a plane cubic, that is, a smooth Fano threefold
$X$ with $\gimel(X)=2.28$; in this case one has~\mbox{$h^{1,2}(X)=1$}.
Using Theorem~\ref{theorem:main}, we obtain the following result.

\begin{corollary}[{cf. \cite[Corollary~1.1.3]{KuznetsovProkhorovShramov}}]
\label{corollary:main}
Let $X$ be a smooth Fano threefold such that $h^{1,2}(X)>0$.
Then $\Aut(X)$ is infinite if and only if
$$
\gimel(X)\in\{2.28, 3.9, 3.14, 4.2\}.
$$
If furthermore $X$ has no extremal contractions to threefolds with
non-Gorenstein singular points, then~\mbox{$\gimel(X)=4.2$}.
\end{corollary}

\begin{remark}
\label{remark:Prokhorov}
Let $X$ be a smooth Fano threefold such that the Picard rank of $X$ is at
least~$2$.
Suppose that $X$ cannot be obtained from a smooth Fano threefold
by blowing up a (smooth irreducible) curve.
In this case, the threefold $X$ is called primitive (see \cite[Definition~1.3]{MoMu83}).
By \cite[Theorem~1.6]{MoMu83}, there exists a (standard) conic bundle $\pi\colon X\to S$ such that either $S\cong\mathbb{P}^2$
and the Picard rank of $X$ is $2$,
or $S\cong\mathbb{P}^1\times\mathbb{P}^1$ and the Picard rank of $X$ is $3$.
Denote by $\Delta$ the discriminant curve of the conic bundle~$\pi$.
Suppose, in addition, that the group $\mathrm{Aut}(X)$ is infinite.
Using Theorem~\ref{theorem:main} and the classification
of primitive smooth Fano threefolds in \cite{MoMu83}, we see that
either the arithmetic genus of $\Delta$ is $1$,
or $\Delta$ is empty and $\pi$ is a $\mathbb{P}^1$-bundle.
Furthermore, in the former case it follows from the classification
that $X$
is a divisor of bi-degree $(1,2)$ in $\mathbb{P}^2\times\mathbb{P}^2$
(and in particular $X$ has a structure of a $\P^1$-bundle in this case as well).
If $\Delta$ is trivial and $S\cong\mathbb{P}^2$, then
the same classification (or \cite{Demin,SzurekWisniewski}) implies
that either $X$ is a divisor
of bi-degree $(1,1)$ in $\mathbb{P}^2\times\mathbb{P}^2$,
or $X$ is a projectivization of a decomposable
vector bundle of rank~$2$ on $\mathbb{P}^2$ (in this case $X$ is toric).
Likewise, if $\Delta$ is trivial
and $S\cong\mathbb{P}^1\times\mathbb{P}^1$, then either $X\cong\mathbb{P}^1\times\mathbb{P}^1\times\mathbb{P}^1$,
or $X$ is a blow up of the quadric cone in $\mathbb{P}^4$ in its vertex.
\end{remark}

\medskip
Information about automorphism groups of smooth complex Fano threefolds can be used to study the problem of existence of a K\"ahler--Einstein metric on them.
For instance, the Matsushima obstruction implies that a smooth Fano threefold does not admit such a metric if its automorphism group is not reductive (see \cite{Matsushima}).
Thus, inspecting our Table~\ref{table:big-table}, we obtain the following.

\begin{corollary}
\label{corollary:KE-not-reductive}
If $X$ is a smooth complex Fano threefold with
\begin{multline*}
\gimel(X)\in\big\{2.28, 2.30, 2.31, 2.33, 2.35, 2.36, 3.16, 3.18, 3.21,\ldots, 3.24, \\ 3.26, 3.28,\ldots,3.31,
 4.8,\ldots,4.12\big\},
\end{multline*}
then $X$ does not admit a  K\"ahler--Einstein metric.
In each of the families of smooth complex Fano threefolds with
$\gimel(X)\in\{1.10, 2.21, 2.26, 3.13\}$,
there exists a variety that does not admit a K\"ahler--Einstein metric.
\end{corollary}

If a smooth complex Fano variety $X$ has an infinite automorphism group,
then the vanishing of its Futaki invariant, that is a character of the Lie algebra of holomorphic vector fields,
is a necessary condition for the existence of a K\"ahler--Einstein metric on~$X$ (see~\cite{Futaki}).
This gives us a simple obstruction for the existence of a K\"ahler--Einstein metric.
If $X$ is toric, then the vanishing of its Futaki invariant is also a sufficient for $X$ to be K\"ahler--Einstein (see \cite{WangZhu}).
In this case, Futaki invariant vanishes if and only if the barycenter of the canonical weight polytope associated to $X$ is at the origin.
The Futaki invariants of smooth non-toric Fano threefolds admitting a faithful action of a two-dimensional torus have been computed in \cite[Theorem~1.1]{Suess14}.
We hope that one can use the results of this paper to compute Futaki invariant of other smooth Fano threefolds having infinite automorphism groups.

If a smooth complex Fano variety $X$ is acted on by a reductive group $G$,
one can use Tian's $\alpha$-invariant~$\mbox{$\alpha_{G}(X)$}$ to prove the existence of a K\"ahler--Einstein metric on $X$.
To be precise, if
$$
\alpha_{G}(X)>\frac{\mathrm{dim}(X)}{\mathrm{dim}(X)+1},
$$
then $X$ is K\"ahler--Einstein by \cite{Tian}.
The larger the group $G$, the larger the $\alpha$-in\-va\-ri\-ant~$\mbox{$\alpha_{G}(X)$}$ is.
This simple criterion has been used in \cite{Nadel,Donaldson,ChShV5,Suess13,CheltsovShramovV22}
to prove the existence of a K\"ahler--Einstein metric on many smooth Fano threefolds.

\begin{example}
\label{example:Donaldson}
In the notation of Theorem~\ref{theorem:Prokhorov}, one has
$$
\alpha_{\PGL_2(\Bbbk)}\Big(X_{22}^{\mathrm{MU}}\Big)=\frac{5}{6}
$$
by Donaldson's \cite[Theorem~3]{Donaldson}. Likewise, one has $\alpha_{\PGL_2(\Bbbk)}(V_5)=\frac{5}{6}$ by \cite{ChShV5}.
Thus, both Fano threefolds $X_{22}^{\mathrm{MU}}$ and $V_5$ are
K\"ahler--Einstein (when $\Bbbk=\mathbb{C}$).
\end{example}

Thanks to the proof of the Yau--Tian--Donaldson conjecture in \cite{ChenDonaldsonSun},
there is an algebraic characterization of smooth complex Fano varieties that admit K\"ahler--Einstein metrics
through the notion of $K$-stability.
In concrete cases, however, this criterion is far from being effective,
because to prove $K$-stability one has to check the positivity of the Donaldson--Futaki invariant for all possible degenerations of the variety.
In a recent paper \cite{DatarSzekelyhidi2016}, Datar and Sz\'ekelyhidi proved that given the action of a reductive
group~$G$ on the variety, it suffices to consider only $G$-equivariant degenerations.
For many smooth Fano threefolds, this equivariant version of $K$-stability has been checked effectively in \cite{IltenSuess}.
We hope that our Theorem~\ref{theorem:main} can be used to check this in some other cases.

In some applications, it is useful to know the full automorphism group of a Fano variety (cf. \cite{Prokhorov2013}).
However, a complete classification of automorphism groups is available only in dimension two (see \cite{DolgachevIskovskikh}),
and in some particular cases in dimension three
(see \cite[Proposition~4.4]{Mukai-CurvesK3Fano}, \cite[\S5]{KuznetsovProkhorovShramov}, and \cite{KuznetsovProkhorov}).
For instance, at the moment, we lack any description of the possible automorphism groups of smooth cubic threefolds.

In dimension four, several interesting examples of smooth Fano varieties with infinite automorphism groups are known (see, for instance, \cite{ProkhorovZaidenberg}).
However, the situation here is very far from classification similar to our Theorem~\ref{theorem:main}.

\medskip

The plan of the paper is as follows. We study automorphisms of smooth Fano threefolds splitting them into several
groups depending on their (sometimes non-unique) construction.
In \S\ref{section:preliminaries} we present some preliminary facts we need in the paper.
In \S\ref{section:products} we study Fano varieties that are either direct products of lower dimensional varieties
or cones. In \S\ref{section:P3}, \S\ref{section:Q},
and \S\ref{section:V5}, we study Fano threefolds
that are blow ups of $\PP^3$, the smooth quadric,
and the Fano threefold $V_5$, respectively.
In \S\ref{section:flags} and \S\ref{section:PnxPm} we study blow ups or double covers of the flag variety
$W=\mathrm{Fl}(1,2;3)$ and products of projective spaces, respectively.
In the next three sections we study three particularly remarkable families of varieties. In~\S\ref{section:2-21} we study the blow up of a smooth quadric in a twisted
quartic; these varieties are more complicated from our point of view then those in~\S\ref{section:Q}, so we separate them.
In \S\ref{section:2.24} we study divisors of bidegree $(1,2)$ in~$\PP^2\times \PP^2$. In \S\ref{section:trigonal} we study smooth
Fano threefolds
$X$ with~\mbox{$\gimel(X)=3.2$}. Note that the varieties from this family are trigonal, but the family is omitted in the Iskovskikh's list of smooth trigonal Fano threefolds in~\cite{Is78}.
Finally, in \S\ref{section:remaining-cases} we study the remaining sporadic Fano threefolds.

Summarizing \S\S\ref{section:products}--\ref{section:remaining-cases}, in Appendix~\ref{section:table} we provide a table containing
an explicit description of connected components of identity in infinite automorphism groups arising
in Theorem~\ref{theorem:main}.


\textbf{Notation and conventions.}

All varieties are assumed to be projective and defined over an algebraically
closed field~$\Bbbk$ of characteristic zero.
Given a variety $Y$ and its subvariety $Z$, we denote by
$\Aut(Y;Z)$ the stabilizer of~$Z$ in $\Aut(Y)$.
By $\Aut^0(Y)$ and $\Aut^0(Y;Z)$ we denote the connected component of identity
in $\Aut(Y)$ and $\Aut(Y;Z)$, respectively.

Throughout the paper we denote by $\mathbb{F}_n$
the Hirzebruch surface
$$
\mathbb{F}_n=\P\Big(\mathcal{O}_{\mathbb{P}^1}\oplus\mathcal{O}_{\mathbb{P}^1}\big(n\big)\Big).
$$
In particular, the surface $\mathbb{F}_1$ is the blow up of $\P^2$ at a point.
By $V_7$ we denote the blow up of~$\P^3$ at a point. We denote by $Q$ the smooth three-dimensional quadric,
and by $V_5$ the smooth section of the Grassmannian $\mathrm{Gr}(2,5)\subset\P^9$
by a linear subspace of dimension~$6$.
By $W$ we denote the flag variety
$\mathrm{Fl}(1,2;3)$ of complete flags in the three-dimensional vector space;
equivalently, this threefold can be described as the projectivization of
the tangent bundle on~$\P^2$ or a smooth divisor of bidegree $(1,1)$ on $\P^2\times\P^2$.

We denote by $\PP(a_0,\ldots,a_n)$
the weighted projective space with weights $a_0,\ldots, a_n$.
Note that $\PP(1,1,1,2)$ is the cone in $\PP^6$ over
a Veronese surface in $\PP^5$.
One has
\begin{equation}\label{eq:P1112}
\Aut\left(\PP(1,1,1,2)\right)\cong (\Ga)^6\rtimes \left(\left(\GL_3(\Bbbk)\times \Gm\right)/\Gm\right),
\end{equation}
where $\Gm$ embeds into the above product by
$$
t\mapsto \left(t\cdot \mathrm{Id}_{\GL_3(\Bbbk)}, t^2\right),
$$ cf.~$\mbox{\cite[Proposition A.2.5]{PS17}}$.

Let $n>k_1>\ldots>k_r$ be positive integers. Then we denote by $\PGL_{n;k_1,\ldots,k_r}(\Bbbk)$
the parabolic subgroup in $\PGL_n(\Bbbk)$ that consists of images of matrices in $\GL_n(\Bbbk)$ preserving
a flag of subspaces of dimensions $k_1,\ldots,k_r$.
In particular, the group~\mbox{$\PGL_{n;k}(\Bbbk)$} is isomorphic to the group of $n\times n$-matrices with a zero lower-left
rectangle of si\-ze~\mbox{$(n-k)\times k$}, and one has
\begin{equation}\label{eq:PGLnk}
\PGL_{n;k}(\Bbbk)\cong \left(\Ga\right)^{k(n-k)}\rtimes
\left(\left(\GL_{k}(\Bbbk)\times\GL_{n-k}(\Bbbk)\right)/\Gm\right).
\end{equation}
Similarly, one has
\begin{multline}\label{eq:PGLnk1k2}
\PGL_{n;k_1,k_2}(\Bbbk)\cong\\
\cong \left(\left(\Ga\right)^{k_1(n-k_1)}\rtimes\left(\Ga\right)^{k_2(k_1-k_2)}\right)\rtimes
\left(\left(\GL_{k_2}(\Bbbk)\times\GL_{k_1-k_2}(\Bbbk)\times\GL_{n-k_1}(\Bbbk)\right)/
\Gm\right).
\end{multline}
In~\eqref{eq:PGLnk} and~\eqref{eq:PGLnk1k2}, the subgroup
$\Gm$ is embedded into each factor as the group
of scalar matrices.
For brevity we write $\BB$
for the group~$\PGL_{2;1}(\Bbbk)\cong \Ga\rtimes \Gm$; this group is a Borel subgroup in $\PGL_2(\Bbbk)$.

For $n\geqslant 5$ by~$\PSO_{n;k}(\Bbbk)$ we denote the parabolic subgroup of~$\PSO_n (\Bbbk)$ preserving an isotropic linear subspace
of dimension $k$. In par\-ti\-cu\-lar,~$\PSO_{n;1}(\Bbbk)$ is a stabilizer in~\mbox{$\Aut^0(\mathcal Q)$} of a point on a smooth $(n-2)$-dimensional quadric~\mbox{$\mathcal{Q}$}. One can check that the group $\PSO_{n;1}(\Bbbk)$ is isomorphic to
the connected component of identity of the automorphism group of a cone over the smooth $(n-4)$-dimensional
quadric. Therefore, we have
$$
\PSO_{5;1}(\Bbbk)
\cong (\Ga)^3\rtimes \left(\SO_3(\Bbbk)\times \Gm\right)
\cong (\Ga)^3\rtimes \left(\PGL_2(\Bbbk)\times \Gm\right)
$$
and
\begin{equation}\label{eq:PSO61}
\PSO_{6;1}(\Bbbk)\cong (\Ga)^4\rtimes \left(\left(\SO_4(\Bbbk)\times \Gm\right)/ \{\pm 1\}\right).
\end{equation}

By $\PGL_{(2,2)}(\Bbbk)$ we denote the image in $\PGL_4(\Bbbk)$ of the group of block-diagonal matrices in $\GL_4(\Bbbk)$ with two $2\times 2$ blocks;
one has
$$
\PGL_{(2,2)}(\Bbbk)\cong \left(\GL_2(\Bbbk)\times\GL_2(\Bbbk)\right)/\Gm,
$$
where $\Gm$ is embedded into each factor $\GL_2(\Bbbk)$ as the group
of scalar matrices.
The group $\PGL_{(2,2)}(\Bbbk)$ acts on $\P^3$ preserving two skew lines.
By~$\PGL_{(2,2);1}(\Bbbk)$ we denote the parabolic subgroup in~$\PGL_{(2,2)}(\Bbbk)$ that
is the stabilizer of a point on one of these lines. It is the image in~\mbox{$\PGL_4(\Bbbk)$} of the group
of block-diagonal matrices in $\GL_4(\Bbbk)$ with two $2\times 2$ blocks, one of which is an upper-triangular matrix.
Thus, one has
$$
\PGL_{(2,2);1}(\Bbbk)\cong \left(\GL_2(\Bbbk)\times\widetilde{\BB}\right)/\Gm,
$$
where $\widetilde{\BB}$ is the subgroup of upper-triangular matrices in $\GL_2(\Bbbk)$, and
$\Gm$ is embedded into each factor as the group
of scalar matrices.

We will use without reference the explicit descriptions of Fano threefolds
provided in~\cite{MM82}, \cite{IP99}, and~\cite{MM04}. We slightly
change the descriptions in some cases for simplicity.

In some cases we compute the dimensions of families of Fano varieties
with certain properties considered \emph{up to isomorphism}. Note that in general Fano varieties do not have moduli spaces with nice properties (cf. Lemma~\ref{lemma:V5-line} below), and to appropriately approach the family parameterizing these up to isomorphism one has to  deal with moduli stacks and coarse moduli spaces of these stacks. This is not
our goal however, and we actually make only a weaker claim in such cases.
Namely, if we say that some family of Fano threefolds up to isomorphism is $d$-dimensional,
we mean that in the corresponding parameter space~$\mathcal{P}$ (which is obvious
from the description of the family of Fano varieties) there is an open subset where
the natural automorphism group of $\mathcal{P}$ acts with equidimensional orbits,
and the corresponding quotient is $d$-dimensional. Also, in such cases we do not consider the question about irreducibility of such families. We point out that in many cases the dimensions of the families of Fano threefolds
are straightforward to compute. In several non-obvious cases we provide computations for the reader's convenience.

\smallskip
\textbf{Acknowledgments.}
We are grateful to I.\,Arzhantsev, S.\,Gorchinskiy, A.\,Kuznetsov,
Yu.\,Prokhorov, L.\,Rybnikov, D.\,Timashev, and V.\,Vologodsky for useful discussions.
Special thanks go to the referees for their careful reading of the paper.
Ivan Cheltsov was partially supported by Royal Society grant No. 	IES\textbackslash R1\textbackslash 180205,
and by the Russian Academic Excellence Project~\mbox{``5-100''}.
Victor Przyjalkowski was partially supported by Laboratory of Mirror Symmetry NRU
HSE, RF government grant, ag.~\mbox{No.~14.641.31.0001}.
Constantin Shramov was partially supported by
the Russian Academic Excellence Project~\mbox{``5-100''}.
The last two authors are Young Russian Mathematics award winners and would like to thank
its sponsors and jury.
This paper was finished during the authors' visit to the Mathematisches Forschungsinstitut in Oberwolfach
in June~2018.
The authors appreciate its excellent environment and hospitality.

\section{Preliminaries}
\label{section:preliminaries}

Any Fano variety $X$ with at most Kawamata log terminal singularities
admits only a finite number of extremal contractions.
In particular, this implies that every extremal contraction
is $\Aut^0(X)$-equivariant, and for every birational extremal
contraction~\mbox{$\pi\colon X\to Y$} the action of $\Aut^0(X)$ on $Y$ is faithful. In the latter case
$\Aut^0(X)$ is naturally embedded into $\Aut^0(Y;Z)$, where $Z\subset Y$ is the image of the exceptional set of~$\pi$.
We will use these facts many times throughout the paper without reference.

The following assertion is well known to experts.

\begin{lemma}\label{lemma:non-ruled}
Let $Y$ be a Fano variety with at most Kawamata log terminal singularities,
and let $Z\subset Y$ be an irreducible subvariety.
Suppose that there is a very ample divisor~$D$ on $Y$ such that $Z$ is not contained in
any effective divisor linearly equivalent to~$D$.
Then the action of the group $\Aut^0(Y;Z)$ on $Z$ is faithful.
Furthermore, if $Z$ is non-ruled, then~$\Aut(Y;Z)$ is finite.
\end{lemma}

\begin{proof}
Since the Picard group of the variety $Y$ is finitely generated,
the linear system of~$D$ defines an $\Aut^0(Y)$-equivariant embedding $\varphi\colon Y\to \P^N$,
so that the automorphisms in $\Aut^0(Y)$ are induced by the automorphisms of $\P^N$.
Note that $Y$ is not contained in a hyperplane in $\P^N$ by construction, and
the same holds for $Z$ by assumption. Thus~$\Aut^0(Y)$ coincides with the
group $\Aut^0(\P^N;Y)$, and the group~$\Aut^0(Y;Z)$ acts faithfully on~$Z$.
Note that both $\Aut^0(Y)$ and $\Aut^0(Y;Z)$ are linear algebraic groups.
Thus, if~$\Aut^0(Y;Z)$ were non-trivial, it would contain a subgroup
isomorphic either to $\Gm$ or $\Ga$. In both cases this would imply
that~$Z$ is ruled.
We conclude that if $Z$ is non-ruled, then the group $\Aut^0(Y;Z)$ is trivial,
so that the group $\Aut(Y;Z)$ is finite.
\end{proof}

The following theorem is classical, see for instance~\cite[\S\S8--9]{Dolgachev}.

\begin{theorem}\label{theorem:dP}
Let $X$ be a smooth del Pezzo surface of degree $d=K_X^2$. Then the following assertions hold.
\begin{itemize}
\item If $d=9$, then $\Aut(X)\cong\PGL_3(\Bbbk)$.

\item If $d=8$, then
either $X\cong\P^1\times\P^1$ and $\Aut^0(X)\cong \PGL_2(\Bbbk)\times\PGL_2(\Bbbk)$,
or $X\cong\mathbb{F}_1$ and $\Aut(X)\cong\PGL_{3;1}(\Bbbk)$.

\item If $d=7$, then $\Aut^0(X)\cong \BB\times \BB$.

\item If $d=6$, then $\Aut^0(X)\cong (\Gm)^2$.

\item If $d\le 5$, then the group $\Aut(X)$ is finite.
\end{itemize}
\end{theorem}

Let $\pi\colon X\to S$ be a flat proper morphism such that $X$ is a threefold, and $S$ is a surface.
If a general fiber of $\pi$ is isomorphic to $\mathbb{P}^1$, we say that $\pi$ is a \emph{conic bundle}.
We say that $\pi$ is a \emph{standard} conic bundle if both $X$ and $S$ are smooth, and
$$
\mathrm{Pic}(X)\cong\pi^*\mathrm{Pic}(S)\oplus\mathbb{Z},
$$
see, for instance,~\cite[Definition 1.3]{Sar80},~\cite[Definition 1.12]{Sar82}, or~\cite[\S3]{Pr18}.
In this case, the morphism $\pi\colon X\to S$ is a Mori fiber space.
Let $\Delta\subset S$ be the discriminant locus of $\pi$, i.\,e. the locus that consists of points $P\in S$ such that
the scheme fiber $\pi^{-1}(P)$ is not isomorphic to $\mathbb{P}^1$.

\begin{remark}[{see~\cite[Corollary 1.11]{Sar82}}]\label{remark:CB}
If $\pi\colon X\to S$ is a standard conic bundle, then~$\Delta$ is a (possibly reducible) reduced curve that has at most nodes as singularities.
In this case, the fiber of $\pi$ over $P$ is isomorphic to a reducible reduced conic in $\mathbb{P}^2$ if $P\in \Delta$ and~$P$ is not a singular point of the curve $\Delta$.
Likewise, if $P$ is a singular point of $\Delta$, then the fiber over $P$ is isomorphic to a non-reduced conic in $\mathbb{P}^2$.
\end{remark}

The following assertion will be used in \S\ref{section:V5} and \S\ref{section:2.24}.

\begin{lemma}
\label{lemma:singular plane cubic}
Let $C\subset\P^2$ be an irreducible nodal cubic.
Then the group $\Aut(\P^2;C)$ is finite.
\end{lemma}

\begin{proof}
The action of $\Aut(\P^2;C)$ on $C$ is faithful by Lemma~\ref{lemma:non-ruled}.
Furthermore, this action lifts to the normalization of $C$, so that
$\Aut(\P^2;C)$ acts on $\P^1$ preserving a pair of points.
Therefore, we have $\Aut^0(\P^2;C)\subset\Gm$.

Suppose that $\Aut^0(\P^2;C)\cong\Gm$. Then
the action of $\Gm$ extends to the projective space
$$
\P\big(H^0(\P^1, \cO_{\P^1}(3))^\vee\big)\cong\P^3.
$$
Moreover, it preserves a twisted cubic $\widetilde{C}$ (that is the image
of $\P^1$ embedded by the latter linear system) therein,
and also preserves some point $P\in\P^3$ outside $\widetilde{C}$
(such that the projection of $\widetilde{C}$ from $P$ provides the initial
embedding $C\subset\P^2$). Since the curve~$C$ is nodal, there exists a unique
line $L$ in $\P^3$ that contains the point $P$ and intersects~$\widetilde{C}$
in two points~$P_1$ and $P_2$. The line $L$ is $\Gm$-invariant.
Furthermore, the points $P$, $P_1$, and~$P_2$ are~$\Gm$-in\-va\-ri\-ant,
so that the action of $\Gm$ on $L$ is trivial. This means
that $\Gm$ (or its appropriate central extension)
acts on $H^0(\P^1, \cO_{\P^1}(3))$ with some weights $w_1, w_2, w_3$, and~$w_4$,
where at least two of the weights coincide. This in turn means that
$\P^3$ cannot contain a~$\Gm$-invariant twisted cubic.
The obtained contradiction shows that the group~$\Aut^0(\P^2;C)$
is actually trivial.
\end{proof}

The following assertion will be used in \S\ref{section:products}, \S\ref{section:Q}, \S\ref{section:V5},
and~\S\ref{section:trigonal}.

\begin{lemma}\label{lemma:SO-stabilizer-hyperplane}
Let $H$ be a hyperplane section of a smooth $n$-dimensional quadric $Y\subset\P^{n+1}$, where $n\geqslant 2$,
and let $\Gamma\subset\Aut(Y)$ be the pointwise
stabilizer of $H$. Then $\Gamma$ is finite, and every automorphism of $H$ is induced by an automorphism of $Y$.
\end{lemma}
\begin{proof}
Denote the homogeneous coordinates on $\PP^{n+1}$ by $x_0,\ldots,x_{n+1}$.
Since the group~$\Aut(Y)$ acts transitively both on $\P^n\setminus Y$ and on~$Y$, we can assume that $H$ is given by $x_0=0$
and $Y$ is given by
$$
x_0^2+\ldots +x_{n+1}^2=0
$$
if $H$ is smooth and by
$$
x_0x_1+x_2^2+\ldots +x_{n+1}^2=0
$$
if $H$ is singular (in this case the corresponding hyperplane is tangent to $Y$ at the point $[0:1:0:\ldots:0]$).
The group $\Gamma$ in both cases acts trivially on the last $n$ coordinates,
so $\Gamma=\{\pm 1\}$ in the former case, and $\Gamma$ is trivial in the latter case.
The last assertion of the lemma is obvious.
\end{proof}

Now we prove several auxiliary assertions about two-dimensional
quadrics.

\begin{lemma}
\label{lemma:quadric-2-ramification}
Let $C$ be a smooth curve of bidegree $(1,n)$,
$n\geq 2$, on $\PP^1\times \PP^1$ ramified,
under the projection on the first factor, in two points. Then in some coordinates~\mbox{$[x_0:x_1]\times [y_0:y_1]$}
on $\PP^1\times \PP^1$
the curve $C$ is given by $x_0y_0^n+x_1y_1^n=0$.
\end{lemma}

\begin{proof}
It follows from the Riemann--Hurwitz formula that the ramification indices of the projection of $C$
to the first factor of~\mbox{$\PP^1\times\PP^1$} at both ramification points equal~$n$.

Consider homogeneous coordinates on the factors of~\mbox{$\PP^1\times\PP^1$}  such that the branch points are~\mbox{$[0:1]$}
and~\mbox{$[1:0]$},
and the ramification points are~\mbox{$[0:1]\times [0:1]$} and~\mbox{$[1:0]\times [1:0]$}.
In the local coordinates $x$, $y$ at~\mbox{$[0:1]\times[0:1]$} the equation of $C$ considered as a polynomial in the $y$-coordinate is of
degree $n$ with the only root at~$y=0$, so it is proportional to~$y^n$. The same applies to the
other ramification point.
\end{proof}

\begin{corollary}
\label{corollary:quadric-two-ramification}
Let $C\subset \PP^1\times\PP^1$ be a smooth
curve of bidegree $(1,n)$, $n\ge 2$, such that~$\mbox{$\Aut(\PP^1\times \PP^1;C)$}$ is infinite.
Then $C$ is unique up to the action of $\Aut(\P^1\times\P^1)$, and one has
$\Aut^0(\PP^1\times \PP^1;C)\cong\Gm$.
\end{corollary}

\begin{proof}
The action of $\Aut^0(\PP^1\times \PP^1;C)$ on $C$ is faithful by Lemma~\ref{lemma:non-ruled}.
The action of~$\Aut(\PP^1\times \PP^1;C)$ preserves the set of
ramification points of the projection of $C$ to the
first factor in~\mbox{$\P^1\times\P^1$}. The cardinality of this set is at
least~$2$, and hence it is exactly~$2$.
The rest is done by Lemma~\ref{lemma:quadric-2-ramification}.
\end{proof}

\begin{lemma}
\label{lemma:quadric-curve-dimension}
Up to the action of $\Aut(\PP^1\times \PP^1)$, there is a unique smooth curve of bidegree~$(1,1)$ or $(1,2)$ on $\PP^1\times \PP^1$,
and a $(2n-5)$-dimensional family of smooth curves of bidegree $(1,n)$ for $n\geqslant 3$.
\end{lemma}

\begin{proof}
The uniqueness in the cases of bidegrees~$(1,1)$ and~$(1,2)$ is obvious.

Suppose that $n\geqslant 3$.
The dimension of the linear system of curves
of bidegree~$(1,n)$ on $\PP^1\times \PP^1$ is
$$
2\cdot (n+1)-1=2n+1.
$$
Let $C$ be a general smooth curve from this linear system.
Let $\pi_1$ be the projection of~$C$ on the first factor of~$\PP^1\times \PP^1$.
Then the ramification points of $\pi_1$ are $\Aut^0(\PP^1\times \PP^1;C)$-invariant.
Since for a general $C$ there are at least $4$ such ramification points,
we conclude that the group $\Aut^0(\PP^1\times \PP^1;C)$ acts trivially on $C\cong\PP^1$.
On the other hand, the action of this group on $C$ is faithful by Lemma~\ref{lemma:non-ruled}.
Thus, we conclude that the group $\Aut(\PP^1\times \PP^1;C)$ is finite.
Since the group $\Aut(\PP^1\times \PP^1)$ has dimension~$6$, the assertion immediately follows.
\end{proof}

\begin{remark}
\label{remark:quadric-two-ramification-small}
Let $C$ be a curve of bidegree $(1,n)$ on $\PP^1\times \PP^1$.
Then
$$
\Aut^0(\PP^1\times \PP^1;C)\cong\BB\times \PGL_2(\Bbbk)
$$
for $n=0$ and $\Aut^0(\PP^1\times \PP^1;C)\cong\PGL_2(\Bbbk)$ for $n=1$.
\end{remark}

We conclude this section with an elementary (but useful) observation concerning the natural projection
from $\GL_n(\Bbbk)$ to $\PGL_n(\Bbbk)$.

\begin{remark}
\label{remark:GL2PGL}
Let $\Gamma$ be a subgroup of $\GL_{n-1}(\Bbbk)$ that contains
all scalar matrices. Consider a subgroup $\Gamma\times\Gm\subset\GL_n(\Bbbk)$
embedded into the group of block-diagonal matrices with blocks of sizes
$n-1$ and $1$. Then the image of $\Gamma\times\Gm$ in $\PGL_n(\Bbbk)$
is isomorphic to $\Gamma$.
\end{remark}

\section{Direct products and cones}
\label{section:products}

In this section we consider smooth Fano threefolds $X$ with
$$
\gimel(X)\in\big\{ 2.34, 2.36, 3.9, 3.27, 3.28, 3.31, 4.2, 4.10, 5.3, 6.1, 7.1, 8.1, 9.1, 10.1 \big\}.
$$

\begin{lemma}\label{lemma:direct-product}
Let $X_1$ and $X_2$ be normal projective varieties. Then
$$
\Aut^0\Big(X_1\times X_2\Big)\cong\Aut^0\big(X_1\big)\times\Aut^0\big(X_2\big).
$$
Furthermore, let $Z\subset X_1$ be a subvariety, and $P\in X_2$ be a point. Consider
$Z$ as a subvariety of the fiber of the projection~\mbox{$X_1\times X_2\to X_2$} over the point~$P$.
Then
$$
\Aut^0\Big(X_1\times X_2; Z\Big)\cong\Aut^0\big(X_1; Z\big)\times\Aut^0\big(X_2; P\big).
$$
\end{lemma}

\begin{proof}
The group $\Aut^0(X_1\times X_2)$ acts trivially on the Neron--Severi group of $X_1\times X_2$.
In particular, it preserves the numerical class of a pull back of some very ample divisor from $X_1$.
This implies that the projection $X_1\times X_2\to X_1$ is $\Aut^0(X_1\times X_2)$-equivariant.
Similarly, we see that the projection $X_1\times X_2\to X_2$ is also $\Aut^0(X_1\times X_2)$-equivariant,
and the first assertion follows. The second assertion easily follows from the first one.
\end{proof}

\begin{corollary}\label{corollary:direct-product}
Let $X$ be a smooth Fano threefold. The following assertions hold.
\begin{itemize}
\item If $\gimel(X)=2.34$, then $\Aut^0(X)\cong \PGL_2(\Bbbk)\times\PGL_3(\Bbbk)$.

\item If $\gimel(X)=3.27$, then $\Aut^0(X)\cong \PGL_2(\Bbbk)\times\PGL_2(\Bbbk)\times\PGL_2(\Bbbk)$.

\item If $\gimel(X)=3.28$, then $\Aut^0(X)\cong \PGL_2(\Bbbk)\times \PGL_{3;1}(\Bbbk)$.

\item If $\gimel(X)=4.10$, then $\Aut^0(X)\cong \PGL_2(\Bbbk)\times \BB\times \BB$.

\item If $\gimel(X)=5.3$, then $\Aut^0(X)\cong \PGL_2(\Bbbk)\times(\Gm)^2$.

\item If $\gimel(X)\in\{6.1,7.1,8.1,9.1,10.1\} $, then $\Aut^0(X)\cong \PGL_2(\Bbbk)$.
\end{itemize}
\end{corollary}

\begin{proof}
In all these cases $X$ is a product of $\P^1$ and a del Pezzo surface.
Thus, the required assertions follow from Lemma~\ref{lemma:direct-product} and
Theorem~\ref{theorem:dP}.
\end{proof}

\begin{lemma}\label{lemma:2-36}
Let $X$ be a smooth Fano threefold
with $\gimel(X)=2.36$.
Then
$$
\Aut^0(X)\cong\Aut\big(\P(1,1,1,2)\big).
$$
\end{lemma}
\begin{proof}
The threefold $X$ is a blow up of $\P(1,1,1,2)$ at its (unique) singular point.
\end{proof}

We refer the reader to~\eqref{eq:P1112} for a detailed description
of the group~\mbox{$\Aut\big(\P(1,1,1,2)\big)$}.

\begin{lemma}\label{lemma:3-9}
Let $Y$ be a smooth Fano variety embedded in $\P^N$ by a complete linear system~$|D|$,
where $D$ is a very ample divisor on $Y$ such that
$D\sim_{\mathbb{Q}} -\lambda K_Y$ for some positive rational number~$\lambda$.
Let $Z\subset Y$ be an irreducible non-ruled subvariety such that $Z$ is not contained in any divisor from the linear system  $|D|$.
Let $\widehat{Y}$ be a cone in $\P^{N+1}$ with vertex $P$ over the variety $Y$.
Then
$$
\Aut^0\big(\widehat{Y};{Z\cup P}\big)\cong\Gm.
$$
\end{lemma}

\begin{proof}
Note that $\widehat{Y}$ a Fano variety, and it
has a Kawamata log terminal singularity at the vertex $P$
because $D$ is proportional to the anticanonical class of~$Y$.
Furthermore, one has
$$
\Aut^0\big(\widehat{Y}\big)\cong\Aut^0\big(\P^{N+1};{\widehat{Y}}\big).
$$
In particular, we can identify $\Aut^0(\widehat{Y};{Z\cup P})$ with a subgroup of $\Aut^0(\P^{N+1};{\widehat{Y}})$.
The group~$\Aut^0(\widehat{Y};{Z\cup P})$ preserves the linear span $\P^N$ of $Z$,
and by Lemma~\ref{lemma:non-ruled} it acts trivially on $\P^N$.
This implies that $\Aut^0(\widehat{Y};{Z\cup P})$ is contained in the pointwise
stabilizer of $\P^N\cup P$ in $\Aut(\P^{N+1})$. The latter stabilizer
is isomorphic to $\Gm$. On the other hand,~$\Aut^0(\widehat{Y};{Z\cup P})$
contains an obvious subgroup isomorphic to $\Gm$, and the assertion follows.
\end{proof}

\begin{corollary}\label{corollary:3-9-and-4-2}
Let $X$ be a smooth Fano threefold
with $\gimel(X)=3.9$ or $\gimel(X)=4.2$.
Then~$\Aut^0(X)\cong\Gm$.
\end{corollary}

\begin{lemma}\label{lemma:3-31}
Let $X$ be a smooth Fano threefold
with $\gimel(X)=3.31$.
Then
$$
\Aut^0(X)\cong \PSO_{6;1}(\Bbbk).
$$
\end{lemma}
\begin{proof}
The threefold $X$ is a blow up of a cone $Y$ over a smooth quadric surface at its (unique) singular point.
Therefore, we have $\Aut^0(X)\cong \Aut^0(Y)$. On the other hand,~$Y$ is isomorphic to the intersection of a smooth four-dimensional
quadric $\mathcal Q\subset \P^5$ with a tangent space at some point.
Using Lemma~\ref{lemma:SO-stabilizer-hyperplane}, we see that
$\Aut^0(Y)$ is isomorphic to a stabilizer of a point on $\mathcal{Q}$
in $\Aut^0(\mathcal{Q})\cong \PSO_6(\Bbbk)$.
\end{proof}

We refer the reader to~\eqref{eq:PSO61} for a detailed description of the group~\mbox{$\PSO_{6;1}(\Bbbk)$}.

\section{Blow ups of the projective space}
\label{section:P3}

In this section we consider smooth Fano threefolds $X$ with
\begin{multline*}
\gimel(X)\in\big\{ 2.4, 2.9, 2.12, 2.15, 2.25, 2.27, 2.28, 2.33, 2.35, 3.6, 3.11, 3.12, 3.14, 3.16, \\
 3.23, 3.25, 3.26, 3.29, 3.30, 4.6, 4.9, 4.12, 5.2 \big\}.
\end{multline*}

Lemma~\ref{lemma:non-ruled} immediately implies the following.

\begin{corollary}\label{corollary:P3-easy-blow-up}
Let $X$ be a smooth Fano threefold
with
$$
\gimel(X)\in\big\{2.4, 2.9, 2.12, 2.15, 2.25\big\}.
$$
Then the group $\Aut(X)$ is finite.
\end{corollary}
\begin{proof}
These varieties are blow ups of $\P^3$ along smooth curves of positive genus that are not contained in a plane.
\end{proof}

\begin{corollary}\label{corollary:3-6-and-3-11}
Let $X$ be a smooth Fano threefold
with $\gimel(X)\in\{3.6,3.11\}$.
Then the group $\Aut(X)$ is finite.
\end{corollary}
\begin{proof}
The variety $X$
is a blow up of a smooth Fano variety $Y$ with $\gimel(Y)=2.25$.
Thus the assertion follows from Corollary~\ref{corollary:P3-easy-blow-up}.
\end{proof}

\begin{lemma}\label{lemma:2-27}
Let $X$ be a smooth Fano threefold
with $\gimel(X)=2.27$.
Then
$$
\Aut^0(X)\cong\PGL_2(\Bbbk).
$$
\end{lemma}
\begin{proof}
The threefold $X$ is a blow up of $\PP^3$ along a twisted cubic curve~$C$.
Applying Lemma~\ref{lemma:non-ruled}, we see that the group~\mbox{$\Aut^0(X)\cong\Aut^0(\PP^3; C)$}
is a subgroup of~\mbox{$\Aut(C)\cong\PGL_2(\Bbbk)$}. On the other hand, since~\mbox{$C\cong\PP^1$}
is embedded into $\P^3$ by a complete linear system, one has~\mbox{$\Aut(C)\subset\Aut^0(\PP^3; C)$}.
\end{proof}

\begin{lemma}\label{lemma:2-28}
Let $X$ be a smooth Fano threefold
with $\gimel(X)=2.28$.
Then
$$
\Aut^0(X)\cong (\Ga)^3\rtimes \Gm.
$$
\end{lemma}
\begin{proof}
The threefold $X$ is a blow up of $\PP^3$ along a plane cubic curve.
Applying Lemma~\ref{lemma:non-ruled}, we see that $\Aut^0(X)$ is isomorphic to a pointwise stabilizer of a plane in $\Aut(\PP^3)\cong\PGL_4(\Bbbk)$.
\end{proof}

\begin{lemma}\label{lemma:2-33-and-2-35}
Let $X$ be a smooth Fano threefold
with $\gimel(X)=2.33$ or $\gimel(X)=2.35$.
Then~$\Aut^0(X)$ is isomorphic to $\PGL_{4;2}(\Bbbk)$ or
$\PGL_{4;1}(\Bbbk)$, respectively.
\end{lemma}
\begin{proof}
The threefold $X$ with $\gimel(X)=2.33$ is a blow up of $\PP^3$ along a line,
while the three\-fold~$X$ with~\mbox{$\gimel(X)=2.35$} is a blow up of $\PP^3$ at a point.
Thus, the assertions of the lemma follow from the definitions of the corresponding
parabolic subgroups in~\mbox{$\Aut(\PP^3)\cong\PGL_4(\Bbbk)$}.
\end{proof}

\begin{lemma}\label{lemma:3-12}
There exists a unique smooth Fano threefold $X$
with $\gimel(X)=3.12$ such that $\Aut^0(X)\cong\Gm$.
For all other smooth Fano threefolds $X$ with $\gimel(X)=3.12$,
the group~$\Aut(X)$ is finite.
\end{lemma}

\begin{proof}
The threefold $X$ is a blow up of $\P^3$ along a disjoint union of a line $\ell$
and a twisted cubic $Z$. We have
$\Aut^0(X)\cong\Aut^0(\P^3;{Z\cup\ell})$.
Consider the pencil $\mathcal{P}$ of planes in $\P^3$ passing through $\ell$.
This pencil is $\Aut^0(\P^3;{Z\cup\ell})$-invariant. Thus there
is an exact sequence of groups
$$
1\to\Aut_{\mathcal{P}}\to\Aut^0(\P^3; Z\cup\ell)\to\Gamma,
$$
where $\Aut_{\mathcal{P}}$ preserves every member of $\mathcal{P}$, and
$\Gamma$ is a subgroup of $\Aut(\P^1)$.
Since a general surface $\Pi\cong\P^2$ in $\mathcal{P}$ intersects
$Z\cup\ell$ by a union of the line $\ell$ and three non-collinear points
outside $\ell$, we conclude that the (connected) group $\Aut_{\mathcal{P}}$ is trivial.
On the other hand, $\Gamma$ is a connected group that preserves the planes in $\mathcal{P}$ that are
tangent to~$Z$. Since there are at least two such planes in $\mathcal{P}$, we conclude that $\Gamma$ can be
infinite only if there are exactly two of them. The latter means that $\ell$ is the intersection line
of two osculating planes of $Z$. Conversely, if $\ell$ is constructed in this way, then $\Aut^0(X)\cong\Aut^0(\P^3;{Z\cup\ell})$ is isomorphic to the stabilizer of the two corresponding tangency points on $Z$ in $\Aut(\P^3;Z)\cong\PGL_2(\Bbbk)$, that is,
to $\Gm$. It remains to notice that the latter configuration is unique up to the action of $\Aut(\P^3;Z)$.
\end{proof}

\begin{lemma}\label{lemma:3-14}
Let $X$ be a smooth Fano threefold
with $\gimel(X)=3.14$.
Then $\Aut^0(X)\cong\Gm$.
\end{lemma}
\begin{proof}
The threefold $X$ is a blow up of $\P^3$ along a union of a point $P$
and a smooth cubic curve $Z$ contained in a plane $\Pi$ disjoint from $P$. Thus, we have~$\mbox{$\Aut^0(X)\cong\Aut^0(\P^3;{Z\cup P})$}$.
The plane $\Pi$ is $\Aut^0(X)$-invariant.
Furthermore, by Lemma~\ref{lemma:non-ruled}, the action of $\Aut^0(\P^3;{Z\cup P})$ on $\Pi$ is trivial,
and the assertion follows.
\end{proof}

\begin{lemma}\label{lemma:3-16}
Let $X$ be a smooth Fano threefold
with $\gimel(X)=3.16$.
Then $\Aut^0(X)\cong\BB$.
\end{lemma}
\begin{proof}
The threefold $X$
is a blow up of the threefold $V_7$ along a proper transform
of a twisted cubic $Z$ passing through the center $P$ of the blow up $V_7\to\P^3$.
Therefore, $\Aut^0(X)$ is isomorphic to
the subgroup of $\Aut(\P^3)$ that preserves both $Z$ and $P$.
The stabilizer of~$Z$ in $\Aut(\P^3)$ is isomorphic to $\PGL_2(\Bbbk)$,
and the assertion follows.
\end{proof}

\begin{lemma}\label{lemma:3-23}
Let $X$ be a smooth Fano threefold
with $\gimel(X)=3.23$.
Then
$$
\Aut^0(X)\cong (\Ga)^3\rtimes (\BB\times \Gm).
$$
\end{lemma}
\begin{proof}
The threefold $X$
is a blow up of the threefold $V_7$ along a proper transform
of a conic~$Z$ passing through the center $P$ of the blow up $V_7\to\P^3$.
Therefore, $\Aut^0(X)$ is isomorphic to the subgroup $\Theta$ of
$\Aut(\P^3)$ that preserves both $Z$ and $P$.

Choose a point $P'$ not contained in the linear span of $Z$, and let $\Gamma$
be the subgroup of $\Theta$ that fixes $P'$.
Then $\Theta\cong (\Ga)^3\rtimes \Gamma$.
On the other hand, $\Gamma$ is the image in $\PGL_4(\Bbbk)$ of the group
$\Gamma'$, such that the image of $\Gamma'$ in $\PGL_3(\Bbbk)\cong\Aut(\P^2)$ is the group that preserves a conic in $\P^2$ and
a point on it.
Now the assertion follows from Remark~\ref{remark:GL2PGL}, cf. Lemma~\ref{lemma:2-30-and-2-31}
below.
\end{proof}

\begin{lemma}\label{lemma:3-25}
Let $X$ be a smooth Fano threefold
with $\gimel(X)=3.25$.
Then
$$
\Aut^0(X)\cong \PGL_{(2,2)}(\Bbbk).
$$
\end{lemma}
\begin{proof}
The threefold $X$
is a blow up of
$\P^3$ along a disjoint union of two lines $\ell_1$ and~$\ell_2$.
Therefore, $\Aut^0(X)$ is isomorphic to the
subgroup of
$\Aut(\P^3)$ that preserves both $\ell_1$ and~$\ell_2$.
\end{proof}

\begin{lemma}\label{lemma:3-26}
Let $X$ be a smooth Fano threefold
with $\gimel(X)=3.26$.
Then
$$
\Aut^0(X)\cong (\Ga)^3\rtimes \left(\GL_2(\Bbbk)\times \Gm\right)
$$
\end{lemma}
\begin{proof}
The threefold $X$ is a blow up of
$\P^3$ along a disjoint union of a line $\ell$ and a point~$P$.
Therefore, $\Aut^0(X)$ is isomorphic to the
subgroup of
$\Aut(\P^3)$ that preserves both $\ell$ and~$P$.
The quotient of the latter group by its unipotent radical is isomorphic to the image in~$\PGL_4(\Bbbk)$ of a subgroup of $\GL_4(\Bbbk)$
that consists of block-diagonal matrices with blocks of sizes $2$, $1$, and $1$.
Now the assertion follows from Remark~\ref{remark:GL2PGL}.
\end{proof}

\begin{lemma}\label{lemma:3-29}
Let $X$ be a smooth Fano threefold
with $\gimel(X)=3.29$.
Then
$$
\Aut^0(X)\cong \PGL_{4;3,1}(\Bbbk).
$$
\end{lemma}
\begin{proof}
The threefold $X$
is a blow up of the threefold $V_7$ along a
line in the exceptional divisor $E\cong\P^2$ of the blow up
$V_7\to\P^3$ of a point $P$ on $\P^3$.
Therefore, $\Aut^0(X)$ is isomorphic to the subgroup of
$\Aut(\P^3)$ that preserves both $P$ and some plane $\Pi$
passing through $P$.
\end{proof}

\begin{lemma}\label{lemma:3-30}
Let $X$ be a smooth Fano threefold
with $\gimel(X)=3.30$.
Then
$$
\Aut^0(X)\cong \PGL_{4;2,1}(\Bbbk).
$$
\end{lemma}
\begin{proof}
The threefold $X$
is a blow up of the threefold $V_7$ along a proper transform
of a line~$\ell$ passing through the center $P$ of the blow up $V_7\to\P^3$.
Therefore, $\Aut^0(X)$ is isomorphic to the subgroup of
$\Aut(\P^3)$ that preserves both $\ell$ and $P$.
\end{proof}

\begin{lemma}
\label{lemma:4-6}
There is a unique smooth Fano threefold $X$
with $\gimel(X)=4.6$.
Moreover, one has
$$
\Aut^0(X)\cong\PGL_2(\Bbbk).
$$
\end{lemma}

\begin{proof}
The variety $X$ can be described as a blow up of $\PP^3$ along three disjoint lines $\ell_1$, $\ell_2$, and $\ell_3$.
Thus $\Aut^0(X)\cong\Aut^0(\PP^3;{\ell_1\cup\ell_2\cup\ell_3})$.
Note that there is a unique quadric $Q'$ passing through $\ell_1$, $\ell_2$, $\ell_3$,
see for instance~\cite[Exercise~7.2]{Re88}.
Hence~$Q'$ is preserved by~\mbox{$\Aut^0(\PP^3;{\ell_1\cup\ell_2\cup\ell_3})$}. Furthermore, the quadric~$Q'$
is smooth. Since the elements of $\Aut(Q')$ are linear, one has
$$\Aut^0(\PP^3;{\ell_1\cup\ell_2\cup\ell_3})\cong\Aut^0(Q';{\ell_1\cup\ell_2\cup\ell_3}).
$$
Since $\ell_i$ are disjoint, they are rulings of the same family of lines on $Q'\cong \PP^1\times\PP^1$.
Now the assertions of the lemma follow from
the facts that the subgroup in $\Aut(\PP^1)\cong\PGL_2(\Bbbk)$ preserving three points on $\PP^1$ is finite, and that
$\PGL_2(\Bbbk)$
acts transitively on triples of distinct points on $\PP^1$.
\end{proof}

\begin{lemma}\label{lemma:4-9}
Let $X$ be a smooth Fano threefold
with $\gimel(X)=4.9$.
Then
$$
\Aut^0(X)\cong \PGL_{(2,2);1}(\Bbbk). 
$$
\end{lemma}
\begin{proof}
The threefold $X$
is a blow up of a one-dimensional fiber of the morphism~$\mbox{$\pi\colon Y\to\P^3$}$,
where $\pi$ is a blow up of $\P^3$
along a disjoint union of two lines $\ell_1$ and~$\ell_2$.
Therefore, $\Aut^0(X)$ is isomorphic to the
subgroup of
$\Aut(\P^3)$ that preserves~$\ell_1$,~$\ell_2$, and a point on one of
these lines.
\end{proof}

\begin{lemma}\label{lemma:4-12}
Let $X$ be a smooth Fano threefold
with $\gimel(X)=4.12$.
Then
$$
\Aut^0(X)\cong
(\Ga)^4\rtimes \left(\GL_2(\Bbbk)\times \Gm\right).
$$
\end{lemma}

\begin{proof}
The threefold $X$
is a blow up of two one-dimensional fibers of the morphism~$\mbox{$\pi\colon Y\to\P^3$}$,
where $\pi$ is a blow up of $\P^3$
along a line $\ell$.
Therefore, $\Aut^0(X)$ is isomorphic to
the subgroup of
$\Aut(\P^3)$ that preserves $\ell$ and two points on $\ell$.
The quotient of the latter group by its unipotent radical is isomorphic to the image in $\PGL_4(\Bbbk)$ of a subgroup of $\GL_4(\Bbbk)$
that consists of block-diagonal matrices with blocks of sizes $2$, $1$, and $1$.
Now the assertion follows from Remark~\ref{remark:GL2PGL}.
\end{proof}

\begin{lemma}\label{lemma:5-2}
Let $X$ be a smooth Fano threefold
with $\gimel(X)=5.2$.
Then
$$
\Aut^0(X)\cong \Gm\times \GL_2(\Bbbk).
$$
\end{lemma}
\begin{proof}
The threefold $X$
is a blow up of two one-dimensional fibers
contained in the same irreducible component
of the exceptional divisor
of the morphism $\pi\colon Y\to\P^3$,
where~$\pi$ is a blow up of $\P^3$
along a disjoint union of two lines $\ell_1$ and $\ell_2$.
Therefore, $\Aut^0(X)$ is isomorphic to the
subgroup of
$\Aut(\P^3)$ that preserves $\ell_1$, $\ell_2$, and two points on one of
these lines.
The latter group is isomorphic to the image in $\PGL_4(\Bbbk)$ of a subgroup of~$\GL_4(\Bbbk)$
that consists of block-diagonal matrices with blocks of sizes $2$, $1$, and $1$.
Now the assertion follows from Remark~\ref{remark:GL2PGL}.
\end{proof}

\section{Blow ups of the quadric threefold}
\label{section:Q}

In this section we consider smooth Fano threefolds $X$ with
$$
\gimel(X)\in\big\{2.7, 2.13, 2.17, 2.23, 2.29, 2.30, 2.31, 3.10, 3.15, 3.18, 3.19, 3.20, 4.4, 5.1\big\}.
$$
Let $Q\subset \PP^4=\PP(V)$ be a smooth quadric and $F\colon V\to \Bbbk$ be the corresponding
(rank~$5$) quadratic form. We say that a quadratic form
(defined on some linear space $U$)
has rank~$k$ (or vanishes for $k=0$) on $\PP(U)$ if it has rank $k$ on~$U$.

Since the ample generator of $\Pic(Q)$ defines an embedding $Q\hookrightarrow\P^4$,
Lemma~\ref{lemma:non-ruled} immediately implies the following.

\begin{corollary}\label{corollary:Q-easy-blow-up}
Let $X$ be a smooth Fano threefold
with
$$
\gimel(X)\in\big\{2.7, 2.13, 2.17\big\}.
$$
Then the group $\Aut(X)$ is finite.
\end{corollary}
\begin{proof}
These varieties are blow ups of $Q$ along smooth curves of positive genus that are not contained in a hyperplane
section.
\end{proof}

\begin{lemma}\label{lemma:2-23}
Let $X$ be a smooth Fano threefold with $\gimel(X)=2.23$.
Then the group $\Aut(X)$ is finite.
\end{lemma}

\begin{proof}
The threefold $X$ is a blow up of $Q$ along a curve $Z$ that is an intersection
of a hyperplane section $H$ of $Q$ with another quadric.
One has $\Aut^0(X)\cong\Aut^0(Q;Z)$.
There is an exact sequence of groups
$$
1\to \Gamma_H\to\Aut(Q;Z)\to\Aut(H;Z),
$$
where $\Gamma_H$ is the pointwise stabilizer of $H$ in $\Aut(Q;Z)$.
Since $Z$ is an elliptic curve,
by Lemma~\ref{lemma:non-ruled}
the group~$\Aut(H;Z)$ is finite. Furthermore,
by Lemma~\ref{lemma:SO-stabilizer-hyperplane}
the group $\Gamma_H$ is finite. Thus, the group~$\Aut(X)$ is finite as well.
\end{proof}

\begin{lemma}\label{lemma:2-30-and-2-31}
Let $X$ be a smooth Fano threefold with $\gimel(X)=2.30$ or $\gimel(X)=2.31$.
Then the group $\Aut^0(X)$ is isomorphic to $\PSO_{5;1}(\Bbbk)$ or $\PSO_{5;2}(\Bbbk)$, respectively.
\end{lemma}

\begin{proof}
Note that the threefold $X$ with $\gimel(X)=2.30$
can be described as a blow up of a point on the smooth three-dimensional
quadric. The rest is straightforward.
\end{proof}

To understand automorphism groups of more complicated blow ups of $Q$ along conics and lines,
we will need some elementary auxiliary facts.

\begin{lemma}
\label{lemma:orthogonal-degenerations}
Let $C=\Pi\cap Q$ be the conic on $Q$ cut out by a plane $\Pi$ and let
$\ell_\Pi$ be the
line orthogonal to $\Pi$ with respect to $F$. Let $F_\Pi$ and $F_{\ell_\Pi}$ be restrictions
of $F$ to the cones over~$\Pi$ and $\ell_\Pi$ respectively.
One has $3-\rk(F_\Pi)=2-\rk(F_{\ell_\Pi})$.
In particular, $\ell_{\Pi}\subset Q$ if and only if $C$ is a double line,
$\ell_{\Pi}$ is tangent to $Q$
if and only if $C$ is reducible and reduced, and
$\ell_{\Pi}$ intersects $Q$ transversally if and only if $C$ is smooth.
\end{lemma}

\begin{proof}
The numbers on both sides of the equality are dimensions of kernels of $F_\Pi$ and~$F_{\ell_\Pi}$ respectively.
Both of them are equal to $\dim (\Pi\cap \ell_\Pi)+1$.
\end{proof}

\begin{lemma}
\label{lemma:two-lines-on-quadric}
Let $C=\Pi\cap Q$ be the conic on $Q$ cut out by a plane $\Pi$ and let
$\ell_\Pi$ be the line orthogonal to $\Pi$ with respect to $F$.
Let $\ell\subset Q$ be a line disjoint from $C$. Then
\begin{itemize}
\item[(i)] the lines $\ell$ and $\ell_\Pi$ are disjoint;

\item[(ii)] if $L\cong\PP^3$ is such that $\ell, \ell_\Pi\subset L$, then $L\cap Q$ is smooth.
\end{itemize}
\end{lemma}

\begin{proof}
Suppose that $\ell$ and $\ell_\Pi$ are not disjoint.
Let $\ell\cap \ell_\Pi=P$. Consider any
point~\mbox{$P'\in C$} and the line $\ell_{P'}$ passing through $P$ and~$P'$.
The quadratic form $F$ vanishes on~$P$ and~$P'$, and the corresponding
vectors are orthogonal to each other with respect to~$F$.
This implies that $F$ vanishes on $\ell_{P'}$, so that
one has $\ell_{P'}\subset Q$. Thus, the cone $T$
over~$C$ with the ver\-tex~$P$ lies on $Q$. In particular, $T$ lies in the
tangent space to $Q$ at $P$ and,
since the intersection of the tangent space with $Q$ is two-dimensional, $T$ is exactly the intersection.
This means that $\ell$ lies on the cone and, thus, intersects $C$. The contradiction proves assertion~(i).

Thus, the linear span $L$ of $\ell$ and $\ell_\Pi$ is three-dimensional. Suppose that $L\cap Q$ is singular.
Then the restriction of $F$ to $L$ is degenerate, so its kernel is nontrivial.
This means that there exist a point $P$ lying on $Q$, $\Pi$, and, thus, on $C$. It also lies on $\ell$
by Lemma~\ref{lemma:orthogonal-degenerations}: since $\ell$ lies on $Q$, its orthogonal
plane (containing $P$) intersected with $Q$ coincides with $\ell$.
Thus, $C$ intersects $\ell$, which gives a contradiction required for assertion~(ii).
\end{proof}

\begin{lemma}
\label{lemma:two-conics}
Let $C_1$ and $C_2$ be two disjoint smooth conics on $Q$. Let $\ell_1$ and $\ell_2$ be their orthogonal lines.
Let $L$ be the linear span of $\ell_1$ and $\ell_2$. Then $L\cong \PP^3$ and $Q\cap L$ is smooth.
\end{lemma}

\begin{proof}
Let $\Pi_1$ and $\Pi_2$ be planes containing $C_1$ and $C_2$. Then the orthogonal linear space
to $L$ is $\Pi_1\cap \Pi_2$. However $\Pi_1$ and $\Pi_2$ intersect by a point, since otherwise the curves
$C_1$ and $C_2$ intersect by points $\Pi_1\cap \Pi_2\cap Q$.

Suppose that $Q\cap L$ is singular. Then the rank of $F$ restricted to $L$ is not maximal. This means that there exist a point
$P$ lying in the kernel of the restricted form. One has~$P\in Q$, $P\in \Pi_1$, $P\in \Pi_2$, so $C_1$ intersects $C_2$.
\end{proof}

\begin{lemma}
\label{lemma:conic-on-quadric-stabilizer}
Let $C\subset Q$ be a smooth conic. Then
$$
\Aut(Q; C)\cong \PGL_2(\Bbbk)\times\Gm,
$$
so that the factor $\PGL_2(\Bbbk)$ acts faithfully on $C$, while the factor~$\Gm$
is the pointwise stabilizer of $C$ in~\mbox{$\Aut(Q)$}.
\end{lemma}

\begin{proof}
Let $\Pi\cong \PP^2$ be the linear span of $C$.
By Lemma~\ref{lemma:orthogonal-degenerations},
the line $\ell$ orthogonal to~$C$ intersects $Q$ by two points $P_1$ and $P_2$.
Since the automorphisms of $Q$ are linear, they preserve $\Pi$ and $\ell$.
Choose coordinates $x_0,\ldots,x_4$ in $\PP^4$ such that the plane $\Pi$ is given by~$x_0=x_1=0$,
the line
$\ell$ is given by
$$
x_2=x_3=x_4=0,
$$
the points $P_i$ are~\mbox{$P_1=[1:0:0:0:0]$} and $P_2=[0:1:0:0:0]$,
and $C$ is given by
$$
x_0=x_1=x_2^2+x_3^2+x_4^2=0.
$$
Then $Q$ is given by
$$
x_0x_1+x_2^2+x_3^2+x_4^2=0.
$$
One has
$$
\Aut^0(Q;C)=\Aut^0(Q;{C\cup P_1\cup P_2}).
$$
The subgroup $\Aut^0(Q; C\cup P_1\cup P_2)\subset\PGL_2(\Bbbk)$
is the image of the subgroup $\Gamma\cong\mathrm{O}_3(\Bbbk)$ in $\GL_5(\Bbbk)$ that consists of block-diagonal matrices with blocks of sizes $3$, $1$, and $1$, where the $3\times 3$-block is an orthogonal matrix, and the entries in the $1\times 1$-blocks are inverses of each other. Since the only scalar matrices contained in $\Gamma$ are just $\pm\mathrm{Id}_{\GL_5(\Bbbk)}$, we see that a subgroup
$\SO_3(\Bbbk)\times\Gm\subset\Gamma$ of index $2$ maps isomorphically to the image of $\Gamma$ in~$\PGL_5(\Bbbk)$. Therefore, one has
$$
\Aut^0(Q; C)\cong \PSO_3(\Bbbk)\times\Gm\cong\PGL_2(\Bbbk)\times\Gm.
$$
The remaining assertions of the lemma are obvious.
\end{proof}

\begin{lemma}
\label{lemma:2-29}
There is a unique smooth Fano threefold $X$
with $\gimel(X)=2.29$.
Moreover,
$$
\Aut^0(X)\cong\PGL_2(\Bbbk)\times\Gm.
$$
\end{lemma}

\begin{proof}
The variety $X$ is a blow up of $Q$ along a smooth conic $C$.
This means that~$\Aut^0(X)\cong \Aut^0(Q;C)$, and the assertion follows from Lemma~\ref{lemma:conic-on-quadric-stabilizer}.
\end{proof}

\begin{lemma}
\label{lemma:3-10}
There is a unique variety $X$ with $\gimel(X)=3.10$ and~$\mbox{$\Aut^0(X)\cong(\Gm)^2$}$.
There is a one-dimensional family of
varieties such that for any its element $X$
one has~$\mbox{$\gimel(X)=3.10$}$ and $\Aut^0(X)\cong\Gm$.
For any other smooth Fano threefold $X$
with~$\mbox{$\gimel(X)=3.10$}$
the group $\Aut(X)$ is finite.
\end{lemma}

\begin{proof}
The threefold $X$ is a blow up of $Q$ along two disjoint conics
$C_1$ and $C_2$.
Thus one has~$\Aut^0(X)\cong\Aut^0(Q;{C_1\cup C_2})$.
By Lemma~\ref{lemma:orthogonal-degenerations},
the line orthogonal to $C_i$ intersects $Q$ by two points $P_1^{(i)}$ and $P_2^{(i)}$.
Thus, $\Aut^0(Q;{C_1\cup C_2})=\Aut^0(Q;{\cup P^{(i)}_j})$. By Lemma~\ref{lemma:two-conics},
the linear span $L$ of $\{P_j^{(i)}\}$ is isomorphic to $\PP^3$,
and the quadric $Q'=Q\cap L$ is smooth.
Moreover, the points $P^{(i)}_1$ and $P^{(i)}_2$ cannot lie on the same ruling of $Q'\cong \PP^1\times \PP^1$ by Lemma~\ref{lemma:orthogonal-degenerations}.
Let $\pi_1$ and $\pi_2$ be two projections of $Q'\cong \PP^1\times \PP^1$ on the factors.
If $|\{\pi_1(P^{(i)}_j)\}|\geq 3$ and~$\mbox{$|\{\pi_2(P^{(i)}_j)\}|\geq 3$}$, then $\Aut(Q';{\cup P^{(i)}_j})$ is finite
since stabilizer of $3$ or more points on $\PP^1$ is finite.
Thus $\Aut^0(Q;{C_1\cup C_2})\cong\Aut^0(X)$ is finite by Lemma~\ref{lemma:SO-stabilizer-hyperplane}.
Thus we can assume that $|\{\pi_1(P^{(i)}_j)\}|=2$. If $|\{\pi_2(P^{(i)}_j)\}|=2$, then $\Aut^0(Q';{\{P_j^{(i)}\}})\cong(\Gm)^2$
and, since the automorphisms of $Q'$ acts on $Q'$ by elements of $\PGL_4(\Bbbk)$, one has $\Aut^0(X)\cong(\Gm)^2$.
Since automorphisms of a quadric surface act transitively on the fourtuples of points of the type as above
and since any two smooth hyperplane sections of a quadric threefold can be
identified by an automorphism of the quadric, all varieties $X$ with $\Aut^0(X)\cong(\Gm)^2$
are isomorphic. Similarly, in the case $|\{\pi_2(P^{(i)}_j)\}|\geq 3$ we get a one-dimensional family
of varieties~$X$ with $\Aut^0(X)\cong\Gm$. (For a general point of the family one has~$\mbox{$|\{\pi_2(P^{(i)}_j)\}|=4$}$.)
\end{proof}

\begin{lemma}
\label{lemma:3-15}
Let $X$ be a smooth Fano threefold with $\gimel(X)=3.15$. Then $X$ is unique up to isomorphism and $\Aut^0(X)\cong\Gm$.
\end{lemma}

\begin{proof}
The threefold $X$ is a blow up of $Q$ along
a disjoint union of a conic $C$ and a line $\ell$.
One has~$\Aut^0(X)\cong\Aut^0(Q; C\cup\ell)$.
By Lemma~\ref{lemma:orthogonal-degenerations}, the line orthogonal to $C$ intersects $Q$ by two points $P_1$ and $P_2$.
Thus, $\Aut^0(Q;{C\cup\ell})\cong\Aut^0(Q;{\ell\cup P_1\cup P_2})$. By Lemma~\ref{lemma:two-lines-on-quadric},
the linear span~$L$ of $\ell$, $P_1$, $P_2$ is isomorphic to $\PP^3$ and a quadric $Q'=Q\cap L$ is smooth.
By Lemma~\ref{lemma:SO-stabilizer-hyperplane}, we have $\Aut^0(Q;{\ell\cup P_1\cup P_2})\cong\Aut^0(Q';{\ell\cup P_1\cup P_2})$.
Let us notice that~$P_1$ and $P_2$ lie on different rulings on $Q'\cong \PP^1\times \PP^1$, because otherwise the line passing
through~$P_1$ and $P_2$ lies on $Q$. The images of $\ell$, $P_1$, and $P_2$ under the projection on the base of the family of lines on~$Q'$ containing
$\ell$ gives three points on~$\PP^1$, so their stabilizer is finite. The projection of $P_1$ and~$P_2$
on the base of the other family of lines gives two points on~$\PP^1$ and the stabilizer of the two points on~$\PP^1$ is $\Gm$. The automorphisms of $\PP^1$
preserving the two points are induced from automorphisms of~$Q'$ and $Q$.
Thus~$\Aut^0(X)\cong\Gm$.
Moreover, any two smooth hyperplane sections of a quadric threefold can be
identified by an automorphism of the quadric. Finally,
there is an automorphism of a smooth two-dimensional quadric sending any line and two points (which do not lie on the line) on different
rulings to another line and two points (which do not lie on the line) on different
rulings.
This gives the remaining assertion of the lemma.
\end{proof}

\begin{lemma}
\label{lemma:3-18}
There is a unique smooth Fano threefold $X$
with $\gimel(X)=3.18$.
Moreover, one has
$$
\Aut^0(X)\cong\BB\times \Gm.
$$
\end{lemma}

\begin{proof}
The variety $X$ is a blow up of a point $P$ on a quadric $Q$ and a proper transform of a conic $C$ passing through it.
Thus $\Aut^0(X)$ is a subgroup of automorphisms of $Q$ preserving $C$ and $P$, and the assertion follows from Lemma~\ref{lemma:conic-on-quadric-stabilizer}.
\end{proof}

\begin{remark}
In~\cite[\S12]{IP99} another description of the smooth Fano threefold
$X$ with~$\gimel(X)=3.18$ is given. Namely,
$X$ is described as a blow up of $\P^3$ along a
disjoint union of a line and a conic.
However these descriptions are equivalent. Indeed, after the
blow up of the conic on~$\PP^3$ the proper transform $\widetilde{\Pi}\cong\P^2$
of the plane $\Pi$ containing the conic has normal bundle $\cO_{\P^2}(-1)$,
so the contraction of $\widetilde{\Pi}$
gives a smooth quadric. The line on~$\PP^3$ becomes the conic passing through
the point which is the image of~$\widetilde{\Pi}$.
\end{remark}

\begin{lemma}
\label{lemma:3-19}
There is a unique smooth Fano threefold $X$
with $\gimel(X)=3.19$.
Moreover, one has
$$
\Aut^0(X)\cong\Gm\times \PGL_2(\Bbbk).
$$
\end{lemma}

\begin{proof}
The variety $X$ is a blow up of two different points $P_1$, $P_2$ on $Q$
not contained in a line in~$Q$.
Thus one gets~$\Aut^0(X)\cong\Aut^0(Q;{P_1\cup P_2})$. Since automorphisms of the quadric are linear,
the line $\ell$ passing through $P_1$ and $P_2$ is preserved by the automorphisms, as well as its orthogonal
plane $\Pi$ and the conic $\Pi\cap Q$, which is smooth by Lemma~\ref{lemma:orthogonal-degenerations}.
The assertion of the lemma follows from Lemma~\ref{lemma:conic-on-quadric-stabilizer}.
\end{proof}

\begin{lemma}
\label{lemma:3-20}
There is a unique variety $X$ with $\gimel(X)=3.20$.
One has
$$
\Aut^0(X)\cong\Gm\times \PGL_2(\Bbbk).
$$
\end{lemma}

\begin{proof}
The threefold $X$ is a blow up of $Q$ along two disjoint lines
$\ell_1$ and $\ell_2$.
One has~$\mbox{$\Aut^0(X)\cong\Aut^0(Q;{\ell_1\cup \ell_2})$}$.
Let $L$ be a linear span of $\ell_1\cup\ell_2$. Then $L\cong \PP^3$,
and the quadric surface $Q'=L\cap Q$ is smooth. The lines $\ell_1$ and~$\ell_2$ contain in the same family of lines
on $Q'$. Thus by the fact that stabilizer of two points on $\PP^1$ is $\Gm$,
surjectivity of a restriction from $\PGL_4(\Bbbk)$ to $\Aut(Q')$ and Lemma~\ref{lemma:SO-stabilizer-hyperplane}
gives~$\Aut^0(X)\cong\Gm\times \PGL_2(\Bbbk)$.
Uniqueness of $X$ follows from the fact that any two smooth hyperplane sections of a quadric threefold can be
identified by an automorphism of the quadric.
\end{proof}

\begin{lemma}
\label{lemma:4-4}
There is a unique smooth Fano threefold $X$ with $\gimel(X)=4.4$.
Moreover, one has
$$
\Aut^0(X)\cong(\Gm)^2.
$$
\end{lemma}

\begin{proof}
The variety $X$ is a blow up of two
points $P_1$, $P_2$ on a quadric $Q$ followed by the blow up of the proper transform of a conic $C$ passing through them.
Thus $\Aut^0(X)$ is a subgroup of automorphisms of $Q$ preserving $C$, $P_1$, and $P_2$.
The assertion of the lemma follows from Lemma~\ref{lemma:conic-on-quadric-stabilizer}.
\end{proof}

\begin{lemma}\label{lemma:5-1}
Let $X$ be a smooth Fano threefold
with $\gimel(X)=5.1$.
Then
$$
\Aut^0(X)\cong \Gm.
$$
\end{lemma}

\begin{proof}
The threefold $X$
is a blow up of three one-dimensional fibers of the
morphism~$\pi\colon Y\to Q$,
where $\pi$ is a blow up of $Q$
along a conic $C$.
Therefore, $\Aut^0(X)$ is isomorphic to
the connected component of identity in the
pointwise stabilizer of $C$ in the group $\Aut(Q)$.
The rest is straightforward, cf. the proof of Lemma~\ref{lemma:conic-on-quadric-stabilizer}.
\end{proof}

\section{Blow ups of the quintic del Pezzo threefold}
\label{section:V5}

In this section we consider smooth Fano threefolds $X$ with
$$
\gimel(X)\in\big\{2.14, 2.20, 2.22, 2.26\big\}.
$$

Let $V_5$ be the smooth section of the Grassmannian $\mathrm{Gr}(2,5)\subset\P^9$
by a linear subspace of dimension~$6$, that is, the smooth Fano threefold of Picard rank $1$ and anticanonical degree~$40$.
Then~$\mathrm{Pic}(V_5)$ is generated by an ample divisor $H$ such that $-K_{V_5}\sim 2H$ and~\mbox{$H^3=5$}.
The linear system $|H|$ is base point free and gives an embedding $V_5\hookrightarrow\mathbb{P}^6$.

The automorphism group $\mathrm{Aut}(V_5)$ is known to be isomorphic to $\mathrm{PGL}_2(\Bbbk)$,
see, e.~g.,~\cite[Proposition~4.4]{Mukai-CurvesK3Fano} or \cite[Proposition~7.1.10]{CheltsovShramov}.
Moreover, the threefold $V_5$ is a union of three $\mathrm{PGL}_2(\Bbbk)$-orbits that can be described as follows (see
\cite[Lemma~1.5]{MukaiUmemura}, \cite[Remark~3.4.9]{IP99}, \cite[Proposition~2.13]{Sanna}).
The unique one-dimensional orbit is a rational normal curve $\mathcal{C}\subset V_5$ of degree~$6$.
The unique two-dimensional orbit is of the form $\mathcal{S}\setminus\mathcal{C}$, where
$\mathcal{S}$ is an irreducible surface in the linear system~$\vert 2H\vert$ whose singular locus consists of the curve $\mathcal{C}$.

\begin{corollary}
\label{corollary:stabilizer}
Let $P$ be a point in $V_5\setminus\mathcal{S}$.
Then the stabilizer of $P$ in $\mathrm{Aut}(V_5)$ is finite.
\end{corollary}

Actually, one can show that the stabilizer of a point
in $V_5\setminus\mathcal{S}$ is isomorphic to the octahedral
group, but we will not use this fact.

\begin{lemma}
\label{lemma:2.14}
Let $X$ be a smooth Fano threefold with $\gimel(X)=2.14$.
Then $\mathrm{Aut}(X)$ is finite.
\end{lemma}

\begin{proof}
The threefold $X$ is a blow up of $V_5$ along
a complete intersection $C$ of two surfaces in the linear system $|H|$.
Then
$$
\mathrm{Aut}^0\big(X\big)\cong\mathrm{Aut}^0\big(V_5; C\big).
$$
Since $C$ is a smooth elliptic curve, the group $\mathrm{Aut}^0(V_5; C)$
must act trivially on it.
On the other hand, it follows from \cite[Lemma~7.2.3]{CheltsovShramov} that $C$ is not contained in the surface~$\mathcal{S}$,
because $\mathrm{deg}(C)=5$.
Therefore, there exists a point $P\in C$ such that $P\not\in\mathcal{S}$ and $P$ is fixed by $\mathrm{Aut}^0(X)$.
Now, applying Corollary~\ref{corollary:stabilizer},
we see that the group $\mathrm{Aut}^0(V_5; C)$ is trivial,
so that $\mathrm{Aut}(X)$ is finite.
\end{proof}

\begin{remark}
\label{remark:V5-lines}
Let $\mathfrak{H}_{\ell}$ be the Hilbert scheme of lines on $V_5$.
There is a $\mathrm{PGL}_2(\Bbbk)$-equivariant identification of~$\mathfrak{H}_{\ell}$
with the plane~$\mathbb{P}^2$, see \cite[Proposition~2.20]{Sanna} (cf.~$\mbox{\cite[Theorem~I]{FuNa})}$.
This plane contains a unique $\mathrm{PGL}_2(\Bbbk)$-invariant conic,
which we denote by $\mathfrak{C}$.
By~$\mbox{\cite[1.2.1]{Iliev}}$ and \cite[Remark~3.4.9]{IP99},
the lines on $V_5$ that are contained in the surface~$\mathcal{S}$
are those that correspond to the points of the conic~$\mathfrak{C}$, and they are exactly the tangent lines to the curve~$\mathfrak{C}$.
Moreover, if $C$ is a line in $V_5$ that is contained in the surface~$\mathcal{S}$,
then for its normal bundle one has
$$
\mathcal{N}_{C/V_5}\cong\mathcal{O}_{\mathbb{P}^1}(1)\oplus\mathcal{O}_{\mathbb{P}^1}(-1)
$$
by \cite[Proposition~2.27]{Sanna}. Likewise, if
$C\not\subset\mathcal{S}$, then $\mathcal{N}_{C/V_5}\cong\mathcal{O}_{\mathbb{P}^1}\oplus\mathcal{O}_{\mathbb{P}^1}$.
\end{remark}

\begin{remark}
\label{remark:V5-blow-up-Fano}
Let $C$ be either a line or an irreducible conic in $V_5$,
let $\pi\colon X\to V_5$ be a blow up of the curve $C$, and let $E$ be the exceptional surface of the blow up $\pi$.
Then the linear system $|\pi^*(H)-E|$ is base point free,
because $V_5$ is a scheme-theoretic intersection of quadrics in $\mathbb{P}^6$,
and $V_5$ does not contain planes.
Thus, the divisor $-K_{X}\sim \pi^*(2H)-E$ is ample.
\end{remark}

\begin{lemma}
\label{lemma:V5-line}
Up to isomorphism, there are exactly two smooth Fano threefolds $X$ with~$\gimel(X)=2.26$.
For one of them, we have
$$
\Aut^0(X)\cong\Gm.
$$
For another one, we have $\mathrm{Aut}^0(X)\cong\BB$.
\end{lemma}

\begin{proof}
In this case, the threefold $X$ is a blow up of $V_5$ along a line $C$;
moreover, by Remark~\ref{remark:V5-blow-up-Fano} a blow up of an arbitrary
line on $V_5$ is a smooth Fano variety.
By Remark~\ref{remark:V5-lines}, the Hilbert scheme
$\mathfrak{H}_{\ell}$ of lines on $V_5$ is
isomorphic to $\P^2$, and by~\cite[Lemma~4.2.1]{KuznetsovProkhorovShramov}
the action of the group
$\Aut(V_5)\cong\PGL_2(\Bbbk)$ on $\mathfrak{H}_{\ell}$ is faithful.
Therefore, we have
$$
\Aut^0(X)\cong\Aut^0(V_5;C)\cong\Gamma,
$$
where $\Gamma$ is the stabilizer in $\PGL_2(\Bbbk)$ of the point
$[C]\in\mathfrak{H}_{\ell}$.
Furthermore, there are two~$\PGL_2(\Bbbk)$-or\-bits
in $\mathfrak{H}_{\ell}$: one is the conic $\mathfrak{C}$, and the other is
$\mathfrak{H}_{\ell}\setminus\mathfrak{C}$.
If~\mbox{$[C]\in\mathfrak{C}$}, then~$\mbox{$\Gamma\cong\BB$}$;
if~\mbox{$[C]\in \mathfrak{H}_{\ell}\setminus\mathfrak{C}$}, then~\mbox{$\Gamma\cong\Gm$}.
Thus, up to isomorphism we get two Fano threefolds $X$
with $\gimel(X)=2.26$, with $\Aut^0(X)$ isomorphic to~$\BB$ and~$\Gm$,
respectively.
\end{proof}

\begin{remark}
\label{remark:V5-line}
Let $C$ be a line in the threefold $V_5$, and let $\pi\colon X\to V_5$ be a blow of this line.
By Remark~\ref{remark:V5-blow-up-Fano}, the threefold $X$ is a smooth Fano threefold $X$ with $\gimel(X)=2.26$.
By Lemma~\ref{lemma:V5-line}, either $\Aut^0(X)\cong\Gm$ or $\mathrm{Aut}^0(X)\cong\BB$.
One can show this without using the description of the Hilbert scheme of lines on $V_5$.
Indeed, it follows from \cite[p.~117]{MoMu83} or from \cite[Proposition~3.4.1]{IP99} that there exists a commutative diagram
$$
\xymatrix{
&&X\ar@{->}[dr]^{\eta}\ar@{->}[ld]_{\pi}&&\\%
&V_5\ar@{-->}[rr]^{\phi}&&Q&}
$$
where $Q$ is a smooth quadric threefold in $\mathbb{P}^4$,
the rational map $\phi$ is given by the projection from the line $C$,
and the morphism $\eta$ is a blow-up of twisted cubic curve in~$Q$,
which we denote by $C_3$.
If $C$ is not contained in the surface $\mathcal{S}$, then
$$
E\cong\mathbb{P}^1\times\mathbb{P}^1.
$$
Likewise, if $C$ is contained in the surface $\mathcal{S}$, then $E\cong\mathbb{F}_2$.
This follows from Remark~\ref{remark:V5-lines}.
In both cases $\eta(E)$ is a hyperplane section of the quadric $Q$ that passes through the curve~$C_3$ (see \cite[Proposition~3.4.1(iii)]{IP99}).
If  $C$ is not contained in the surface $\mathcal{S}$, then this hyperplane section is smooth.
Otherwise, the surface $\eta(E)$ is a quadric cone,
so that the induced morphism $E\to\eta(E)$ contracts the $(-2)$-curve of the surface $E\cong\mathbb{F}_2$ in this case.
One has
$$
\mathrm{Aut}(X)\cong \mathrm{Aut}(V_5;C)\cong\mathrm{Aut}(Q;C_3),
$$
where the group $\mathrm{Aut}(Q;C_3)$ is easy to describe explicitly,
since the pair $(Q,C_3)$ is unique up to projective equivalence (in each of our cases).
Indeed, fix homogeneous coordinates
$$[x:y:z:t:w]$$ on $\mathbb{P}^4$.
We may assume that $\eta(E)$ is cut out on $Q$ by $w=0$.
Then we can identify $C_3$ with the image of the map
$$
[\lambda:\mu]\mapsto\big[\lambda^3:\lambda^2\mu:\lambda\mu^2:\mu^3:0\big].
$$
If $\eta(E)$ is smooth, then we may assume that $Q$ is given by
\begin{equation}\label{eq:quadric-for-V5}
xt-yz+w^2=0.
\end{equation}
In this case, it follows from Lemma~\ref{lemma:SO-stabilizer-hyperplane} and Corollary~\ref{corollary:quadric-two-ramification} that
$\mathrm{Aut}^0(Q;C_3)\cong\Gm$.
Here, the action of the group $\Gm$ is given by
$$
\zeta\colon [x:y:z:t:w]\mapsto[x:\zeta^2 y:\zeta^4z:\zeta^6t:\zeta^3w].
$$
Similarly, if $\eta(E)$ is singular, one can show that $\mathrm{Aut}^0(Q;C_3)\cong\BB$.
\end{remark}

For every line in $V_5$, there exists a unique surface in $|H|$ that is singular along this line.
This surface is spanned by the lines in $V_5$ that intersect this given line.
More precisely, we have the following result.

\begin{lemma}
\label{lemma:V5-line-cubic}
Let $S$ be a surface in $|H|$ that has non-isolated singularities.
Then $S$ is singular along some line $C$, and it is smooth away from $C$.
If $C\subset\mathcal{S}$, then~$S$~does not contain irreducible curves of degree $3$.
Likewise, if $C\not\subset\mathcal{S}$, then $S$ does not contain irreducible curves of degree $3$ that intersect $C$.
Moreover, in this case, the surface~$S$ contains a unique $\mathrm{Aut}^0(V_5;C)$-invariant irreducible cubic curve that is disjoint from $C$.
Furthermore, this curve is a twisted cubic curve.
\end{lemma}

\begin{proof}
If $H$ is a general surface in $|H|$, then $H$ is a smooth del Pezzo surface of degree $5$, and
$$
S\vert_{H}\in |-K_{H}|,
$$
so that $S\vert_{H}$ is an irreducible singular curve of arithmetic genus $1$,
which implies that $S\vert_{H}$ has a unique singular point (an ordinary isolated double point or an ordinary cusp).
This shows that $S$ is singular along some line $C$, and $S$ has isolated singularities away from this line.

Let us use the notation of Remark~\ref{remark:V5-line}.
Denote by $\widetilde{S}$ the proper transform of the surface~$S$ on the threefold $X$.
Then $\widetilde{S}$ is the exceptional surface of the birational morphism~$\eta$.
In particular, the surface $S$ is smooth away from the line $C$.

We have $\widetilde{S}\cong\mathbb{F}_n$ for some non-negative integer $n$.
If $C$ is not contained in the surface~$\mathcal{S}$, then $n=1$.
This follows from the proof of \cite[Lemma~13.2.1]{CheltsovShramov}.
Indeed, let $\mathbf{s}$ be a curve in~$\mathbb{F}_n$ such that $\mathbf{s}^2=-n$,
and let $\mathbf{f}$ be a general fiber of the natural projection $\xi\colon\mathbb{F}_n\to\mathbb{P}^1$.
Then
$$
-\widetilde{S}\vert_{\widetilde{S}}\sim \mathbf{s}+k\mathbf{f}
$$
for some integer $k$. Then
$$
-7=2+K_{Q}\cdot C_3=\widetilde{S}^3=\Big(\mathbf{s}+k\mathbf{f}\Big)^2=-n+2k,
$$
which shows that $k=\frac{n-7}{2}$.
Thus, if $C$ is not contained in the surface $\mathcal{S}$,
then $\eta(E)$ is a smooth surface, which implies that
$$
E\vert_{\widetilde{S}}\sim\mathbf{s}+\frac{n-1}{2}\mathbf{f}
$$
is a section of the natural projection $\widetilde{S}\to\mathbb{P}^1$,
which immediately shows that $n=1$, since~$0\leqslant E\vert_{\widetilde{S}}\cdot\mathbf{s}=-\frac{n+1}{2}$ otherwise.
Likewise, if $C$ is contained in the surface $\mathcal{S}$, then~$\eta(E)$ is a quadric cone, which implies that
$$
E\vert_{\widetilde{S}}=Z+F,
$$
where $Z$ and $F$ are irreducible curves in $\widetilde{S}$ such that $Z$ is a section of the projection $\xi$,
and $F$ is a fiber of $\xi$ over the singular point of the quadric cone $\eta(E)$.
In this case, we have
$$
Z\sim\mathbf{s}+\frac{n-3}{2}\mathbf{f},
$$
so that $0\leqslant Z\cdot\mathbf{s}=-\frac{n+3}{2}$ if $n\ne 3$.
Hence, if $C\subset\mathcal{S}$, then $n=3$.

Let $M$ be an irreducible curve in $S$ such that $M\ne C$,
and let $\widetilde{M}$ be the proper transform of the curve $M$ on the threefold $X$.
Then
$$
\widetilde{M}\sim a\mathbf{s}+b\mathbf{f}
$$
for some non-negative integers $a$ and $b$.
Moreover, we have $M\ne\mathbf{s}$, because $\mathbf{s}\subset E\cap\widetilde{S}$.
In particular, we have
$$
0\leqslant \mathbf{s}\cdot\widetilde{M}=\mathbf{s}\cdot\Big(a\mathbf{s}+b\mathbf{f}\big)=b-na,
$$
so that $b\geqslant na$. Thus, if $C$ is contained in the surface $\mathcal{S}$, then, since  $n=3$, we have
$$
\mathrm{deg}\big(M\big)=\pi^*(H)\cdot\widetilde{M}=\big(\mathbf{s}+4\mathbf{f}\big)\cdot\widetilde{M}=b+a\geqslant 4a,
$$
which implies that $\mathrm{deg}(M)\ne 3$.

To complete the proof of the lemma, we may assume that
$C$ is not contained in $\mathcal{S}$ and~$M$ is an irreducible cubic curve.
We have to show that such a curve $M$ exists and it is unique.
Almost as above, we have
$$
3=\mathrm{deg}\big(M\big)=\pi^*(H)\cdot\widetilde{M}=\Big(\mathbf{s}+3\mathbf{f}\big)\cdot\widetilde{M}=b+2a\geqslant 3a,
$$
so that $\widetilde{M}\sim \mathbf{s}+\mathbf{f}$.
In particular, the curve $M$ is disjoint from the curve $C$, since $E\cap\widetilde{S}=\mathbf{s}$.

Recall from Remark~\ref{remark:V5-line} that $\mathrm{Aut}^0(X)\cong\mathrm{Aut}^0(V_5;C)\cong\mathrm{Aut}^0(Q;C_3)\cong\Gm$ and
$$
\mathrm{Aut}(X)\cong\mathrm{Aut}(V_5;C)\cong\mathrm{Aut}(Q;C_3).
$$
Thus, to complete the proof of the lemma, it is enough to show that
the linear system~\mbox{$|\mathbf{s}+\mathbf{f}|$} contains unique irreducible $\mathrm{Aut}^0(X)$-invariant curve.
To do this, observe that the group $\mathrm{Aut}(Q;C_3)$ contains a slightly larger subgroup $\Gamma\cong\Gm\rtimes\Z/2\Z$.
On the quadric~$Q$ given by equation~\eqref{eq:quadric-for-V5},
the additional involution acts as
$$
[x:y:z:t:w]\mapsto[t:z:y:x:w].
$$
This group $\Gamma$ also faithfully acts on the surface $\widetilde{S}\cong\mathbb{F}_1$.

We claim that the linear system $|\mathbf{s}+\mathbf{f}|$ contains a unique $\Gamma$-invariant curve
(cf. the proof of~\mbox{\cite[Lemma~13.2.1]{CheltsovShramov}}, where this is shown for the subgroup of the group $\Gamma$
that is isomorphic to the group $\mathrm{D}_{10}$).
Indeed, let $\theta\colon\widetilde{S}\to\mathbb{P}^2$ be the contraction of the curve~$\mathbf{s}$.
Then $\theta$ defines a faithful action of $\Gamma$ on $\mathbb{P}^2$.
It is easy to check that $\mathbb{P}^2$ has a unique $\Gamma$-invariant line.
Denote this line by $\ell$, and denote its proper transform on $\widetilde{S}$ by $\widetilde{\ell}$.
Then~$\ell$ does not contain $\theta(\mathbf{s})$, so that
$$
\widetilde{\ell}\sim\mathbf{s}+\mathbf{f}.
$$
Thus, the curve $\widetilde{\ell}$ is the unique $\Gamma$-invariant curve in $|\mathbf{s}+\mathbf{f}|$.

By construction, the curve $\widetilde{\ell}$ is $\mathrm{Aut}^0(X)$-invariant curve in $|\mathbf{s}+\mathbf{f}|$.
In fact, it is unique irreducible $\mathrm{Aut}^0(X)$-invariant curve in $|\mathbf{s}+\mathbf{f}|$.
This completes the proof of the lemma.
\end{proof}

\begin{corollary}
\label{corollary:V5-blow-up-Fano}
Let $C$ be a twisted cubic curve in $V_5$,
let $\pi\colon X\to V_5$ be a blow up of the curve $C$,
and let $E$ be the exceptional surface of the blow up $\pi$.
Then the linear system~$|\pi^*(H)-E|$ is free from base points,
and the divisor $-K_{X}$ is ample.
\end{corollary}

\begin{proof}
It is enough to show that $|\pi^*(H)-E|$ is free from base points.
Suppose that this is not the case. Then $V_5$ contains a line $L$ such that either $L$ is a secant of the curve $C$ or the line $L$ is tangent to $C$.
This follows from the facts that $V_5$ is
a scheme-theoretic intersection of quadrics and that $V_5$ does not contain
quadric surfaces.
Let $S$ be the surface in $|H|$ that is singular along $L$.
Then $C$ is contained in $S$, which contradicts Lemma~\ref{lemma:V5-line-cubic}.
\end{proof}

\begin{lemma}
\label{lemma:V5-cubic}
Up to isomorphism, there exists a unique smooth Fano threefold $X$ with~$\gimel(X)=2.20$ such that the group $\Aut(X)$ is infinite.
Moreover, in this case, one has~$\Aut^0(X)\cong\Gm$.
\end{lemma}

\begin{proof}
In this case, the threefold $X$ is a blow up of $V_5$ along a twisted cubic curve.
Denote this cubic by $C$. Then
$$
\mathrm{Aut}^0(X)\cong\mathrm{Aut}^0(V_5;C).
$$
Note that $C$ is not contained in the surface $\mathcal{S}$ by \cite[Lemma~7.2.3]{CheltsovShramov}.

It follows from Corollary~\ref{corollary:V5-blow-up-Fano} that there exists a commutative diagram
\begin{equation}
\label{equation:V5-cubic-projection}
\xymatrix{
&&X\ar@{->}[dr]^{\eta}\ar@{->}[ld]_{\pi}&&\\%
&V_5\ar@{-->}[rr]^{\phi}&&\mathbb{P}^2&}
\end{equation}
such that $\phi$ is a linear projection from the twisted cubic $C$.
The morphism $\eta$
is a standard conic bundle that is given by the linear system $|\pi^*(H)-E|$.
Simple computations imply that the discriminant of the conic bundle $\eta$ is a curve of degree $3$.
Denote it by $\Delta$.
Then~$\Delta$ has at most isolated ordinary double points  by  Remark~\ref{remark:CB}.

Note that the diagram \eqref{equation:V5-cubic-projection} is $\mathrm{Aut}(X)$-equivariant.
Moreover, if the group $\mathrm{Aut}^0(X)$ is not trivial,
then it acts non-trivially on $\mathbb{P}^2$ in \eqref{equation:V5-cubic-projection}
because $\mathrm{Aut}^0(X)$ acts non-trivially on the $\pi$-exceptional surface $E$,
which follows from Corollary~\ref{corollary:stabilizer}, since $C$ is not contained in the surface $\mathcal{S}$.

Using Lemmas~\ref{lemma:non-ruled} and \ref{lemma:singular plane cubic}, we see that $\Delta$ must be reducible.
Thus, we write $\Delta=\ell+M$, where $M$ is a possibly reducible conic.

Let $\widetilde{S}$ be the surface in $|\pi^*(H)-E|$ such that $\eta(\widetilde{S})=\ell$, and let $S=\pi(\widetilde{S})$.
Then both $\widetilde{S}$ and $S$ are non-normal by construction.
Thus, it follows from Lemma~\ref{lemma:V5-line-cubic} that $S$ is singular along some line in $V_5$.
Denote this line by $L$.
Then $L$ is not contained in the surface $\mathcal{S}$ by Lemma~\ref{lemma:V5-line-cubic}.
Then $\mathrm{Aut}^0(V_5;L)\cong\Gm$ by Remark~\ref{remark:V5-line}.

Since $L$ must be $\mathrm{Aut}^0(V_5;C)$-invariant, we see that either $\mathrm{Aut}^0(V_5;C)$ is trivial, or
$$
\mathrm{Aut}^0\big(V_5;C\big)\cong\Gm.
$$
In the later case, Lemma~\ref{lemma:V5-line-cubic} also implies that $C$ is the unique $\mathrm{Aut}^0(V_5;C)$-invariant twisted cubic curve
contained in the surface $S$.
This implies that, up to the action of~\mbox{$\Aut(V_5)$},
there exists a unique choice for $C$ such that the group $\mathrm{Aut}^0(V_5;C)$ is not trivial,
and in this case, one has $\mathrm{Aut}^0(V_5;C)\cong\Gm$.
In fact, Lemma~\ref{lemma:V5-line-cubic} also implies that this case indeed exists.
This completes the proof of the lemma.
\end{proof}

We need the following fact about rational quartic curves in $\PP^3$.

\begin{lemma}
\label{lemma:twisted-quartic-curve}
Let $C_4$ be a smooth rational quartic curve in $\mathbb{P}^3$.
Then $C_4$ is contained in a unique quadric surface.
Moreover, this quadric surface is smooth.
\end{lemma}

\begin{proof}
Dimension count shows that $C_4$ is contained in a quadric surface, which we denote by $S$.
Then this quadric surface is unique, since otherwise $C_4$ would be a complete intersection of two quadric surfaces in $\mathbb{P}^3$,
which is not the case.

Suppose that $S$ is singular. Then $S$ is an irreducible quadric cone.
Let $\alpha\colon\mathbb{F}_2\to S$ be the blow up of the vertex of the cone $S$,
and let~$\widetilde{C}_4$ be the proper transform of the curve~$C_4$ on the surface $\mathbb{F}_2$.
Denote by $\mathbf{s}$ the $(-2)$-curve in $\mathbb{F}_2$, and denote by $\mathbf{f}$
the general fiber of the natural projection $\mathbb{F}_2\to\mathbb{P}^1$.
Then
$$
\widetilde{C}_4\sim a\mathbf{s}+b\mathbf{f}
$$
for some non-negative integers $a$ and $b$. Then
$$
b=\Big(\mathbf{s}+2\mathbf{f}\Big)\cdot\Big(a\mathbf{s}+b\mathbf{f}\Big)=\mathrm{deg}\big(C_4\big)=4,
$$
since $\alpha$ is given by the linear system $|\mathbf{s}+2\mathbf{f}|$.
Keeping in mind that $\widetilde{C}_4$ is a smooth rational curve, we deduce that $a=1$. Then
$$
\mathbf{s}\cdot\widetilde{C}_4=\mathbf{s}\cdot\Big(\mathbf{s}+4\mathbf{f}\Big)=2,
$$
which implies that $C_4=\alpha(\widetilde{C}_4)$ is singular.
This shows that $S$ is a smooth quadric, since~$C_4$ is smooth.
\end{proof}

Let us conclude this section by proving the following result.

\begin{lemma}
\label{lemma:V5-conic}
Up to isomorphism,
there is a unique smooth Fano threefold $X$ with~\mbox{$\gimel(X)=2.22$} such that $\Aut(X)$ is infinite.
Moreover, in this case, one has~\mbox{$\Aut^0(X)\cong\Gm$}.
\end{lemma}

\begin{proof}
In this case, the threefold $X$ is a blow up of $V_5$ along a conic.
Denote this conic by $C$. It easily follows from Remark~\ref{remark:V5-blow-up-Fano} that there exists a commutative diagram
\begin{equation}
\label{equation:V5-conic-projection}
\xymatrix{
&&X\ar@{->}[dr]^{\eta}\ar@{->}[ld]_{\pi}&&\\%
&V_5\ar@{-->}[rr]^{\phi}&&\mathbb{P}^3&}
\end{equation}
where $\pi$ is the blow up of the conic $C$, the morphism $\phi$ is a linear projection from the conic $C$,
and the morphism $\eta$ is a blow-up of a smooth rational quartic curve, which we denote by $C_4$.
Since \eqref{equation:V5-conic-projection} is $\mathrm{Aut}(X)$-equivariant,
we see that
$$
\mathrm{Aut}(X)\cong\mathrm{Aut}\big(V_5;C\big)\cong
\mathrm{Aut}\big(\mathbb{P}^3;C_4\big).
$$

By Lemma~\ref{lemma:twisted-quartic-curve},
the curve $C_4$ is contained in a unique quadric surface,
which we denote by $S$.
This quadric surface is smooth again by Lemma~\ref{lemma:twisted-quartic-curve}.
Thus, we have
$$
S\cong\mathbb{P}^1\times\mathbb{P}^1,
$$
and $C_4$ is a curve of bidegree $(1,3)$.
Note that $\Aut^0(\PP^3; C_4)\cong \Aut^0(S; C_4)$ by Lemma~\ref{lemma:non-ruled}.
Thus, by Corollary~\ref{corollary:quadric-two-ramification},
there exists a unique (up to the projective equivalence) choice for~$C_4$ such that the group $\Aut^0(\mathbb{P}^3;C_4)$ is not trivial.
In this case, Corollary~\ref{corollary:quadric-two-ramification} also implies that $\Aut^0(\mathbb{P}^3;C_4)\cong\Gm$.
This case indeed exists.
For instance, one such smooth rational quartic curve $C_4$ is given by the parameterization
$$
\big[u^4:u^3v:uv^3:v^4\big]
$$
where $[u:v]\in\mathbb{P}^1$. In this case, the quadric $S$ is given by $xt=yz$,
where $[x:y:z:t]$ are homogeneous coordinates on~$\mathbb{P}^3$.
Since $C_4$ is a scheme-theoretic intersection of cubic surfaces,
the blow up of $\mathbb{P}^3$ at this quartic curve is indeed a Fano threefold,
which can be obtained by blowing up $V_5$ at a conic.
\end{proof}

\section{Blow ups of the flag variety}
\label{section:flags}

In this section we consider smooth Fano threefolds $X$ with
$$
\gimel(X)\in\big\{2.32, 3.7, 3.13, 3.24, 4.7\big\}.
$$
Recall that we denote the (unique) smooth Fano threefold with $\gimel(X)=2.32$ by $W$.
This threefold is isomorphic to the flag variety $\mathrm{Fl}(1,2;3)$ of complete flags in the three-dimensional vector space,
and also to the projectivization of the tangent bundle on~$\P^2$ or a smooth divisor of bi-degree $(1,1)$ on $\P^2\times\P^2$.

We start with a well-known observation which we will later use several times without reference.

\begin{lemma}
\label{lemma:2-32}
One has $\Aut^0(W)\cong\PGL_3(\Bbbk)$, and each of the two
projections~\mbox{$W\to\P^2$} induces an isomorphism
$$
\Aut^0(W)\cong\Aut(\P^2)\cong\PGL_3(\Bbbk).
$$
\end{lemma}

\begin{proof}
Left to the reader.
\end{proof}

\begin{lemma}\label{lemma:3-7}
Let $X$ be a smooth Fano threefold
with $\gimel(X)=3.7$.
Then the group $\Aut(X)$ is finite.
\end{lemma}
\begin{proof}
The threefold $X$
is a blow up of the flag variety $W$ along a smooth curve $C$ which
is an intersection of two divisors from $|-\frac{1}{2}K_W|$. By adjunction formula,
$C$ is an elliptic curve.
We have $\Aut^0(X)\subset\Aut^0(W;C)$.

Let $\pi_1\colon W\to\P^2$ and $\pi_2\colon W\to\P^2$ be natural projections.
Then both of them are~$\Aut^0(W)$-equivariant.
Let $C_1=\pi_1(C)$ and $C_2=\pi_2(C)$.
Since the intersection of the fibers of each of the projections $\pi_i$
with divisors from the linear system $|-\frac{1}{2}K_W|$ equals~$1$, we
see that $C_1$ and $C_2$ are isomorphic to $C$. One has
$$
\Aut^0(W;C)\subset \Aut(\P^2;{C_1})\times \Aut(\P^2;{C_2}).
$$
On the other hand, both groups $\Aut(\P^2;{C_1})$ and $\Aut(\P^2;{C_2})$ are finite by Lemma~\ref{lemma:non-ruled}.
\end{proof}

We will need the following simple auxiliary facts.

\begin{lemma}[{\cite[Lemma~6.2(a)]{ProkhorovZaidenberg}}]
\label{lemma:two-tangent-conics}
Let $C_1$ and $C_2$ be two irreducible conics in $\P^2$.
The following assertions hold.
\begin{itemize}
\item[(i)] If $|C_1\cap C_2|=1$, then $\Aut^0(\P^2; C_1\cup C_2)\cong\Ga$.

\item[(ii)] If $C_1$ and $C_2$ are tangent to each other at two distinct
points, then~\mbox{$\Aut^0(\P^2; C_1\cup C_2)\cong\Gm$}.
\end{itemize}
\end{lemma}

Now we proceed to varieties $X$ with $\gimel(X)=3.13$.

\begin{lemma}\label{lemma:3-13}
The following assertions hold.
\begin{itemize}
\item There is a unique smooth Fano threefold $X$
with $\gimel(X)=3.13$ and
$$
\Aut^0\big(X\big)\cong\PGL_2\big(\Bbbk\big).
$$

\item There is a unique smooth Fano threefold $X$
with $\gimel(X)=3.13$ and $\Aut^0(X)\cong\Ga$.

\item For all other smooth Fano threefolds $X$ with $\gimel(X)=3.13$, one has $\Aut^0(X)\cong\Gm$.
\end{itemize}
\end{lemma}

\begin{proof}
A smooth Fano threefold $X$ with $\gimel(X)=3.13$ is a blow up
of the flag variety $W$ along a curve $C$ such that both natural projections
$\pi_1$ and $\pi_2$ map $C$ isomorphically to smooth conics $C_i$ in
the two copies of $\P^2$.
Let $S_1=\pi_1^{-1}(C_1)$ and $S_2=\pi_2^{-1}(C_2)$.
One can check that $S_i\cong\P^1\times\P^1$,
the intersection $S_1\cap S_2$ is a curve of bidegree
$(2,2)$ on~$S_1\cong\P^1\times\P^1$, and $C$ is its irreducible component of
bidegree $(1,1)$.
One has
$$
\Aut^0(X)\cong\Aut^0(W;C).
$$
The curve $C$ and the surfaces $S_i$
are $\Aut^0(W;C)$-invariant. Moreover, the projections~$\pi_i\colon W\to\P^2$ are $\Aut^0(W;C)$-equivariant, and the conics
$C_i$ are invariant with respect to the arising action of $\Aut^0(W;C)$ on
$\P^2$.

Note that the threefold $W$ allows to identify one copy of $\P^2$
with the dual of the other. On the other hand, the conic $C_1$ provides
an identification of $\P^2\supset C_1$ with its dual. Under this
identification, we may consider $C_2$ as a conic contained in the same
projective plane~$\P^2$ as the conic $C_1$, so that both $C_1$ and $C_2$ are
$\Aut^0(W;C)$-invariant with respect to the action of
$\Aut^0(W;C)$ on $\P^2$.
Moreover, we conclude that
$$
\Aut^0(W;C)\cong\Aut(\P^2;C_1\cup C_2).
$$

Keeping in mind that $W$ is the flag variety $\mathrm{Fl}(1,2;3)$,
we can describe
the curve~$\mbox{$S_1\cap S_2\subset W$}$ as
$$
S_1\cap S_2=\left\{(P,\ell)\in W\mid P\in C_1\ \text{and}\ \text{$\ell$\ is tangent to\ $C_2$}\right\}.
$$
The double covers $\pi_i\colon S_1\cap S_2\to C_i$ are branched
exactly over the points of $C_1\cap C_2$.
If~$S_1\cap S_2$ is a reduced curve, then its arithmetic genus
is equal to $1$, and we conclude that for it to have an irreducible
component of bidegree $(1,1)$ on $S_1\cong\P^1\times\P^1$, one of the following
cases must occur:
either $|C_1\cap C_2|=2$ and $C_1$ is tangent to $C_2$ at both points of their
intersection, or $|C_1\cap C_2|=1$.
If the intersection $S_1\cap S_2$ is not reduced, then it is just
the curve $C$ taken with multiplicity $2$, so that the conics $C_1$ and
$C_2$ coincide. Recall that
the conics $C_i$ are irreducible in all of these cases.

Suppose that the conics $C_1$ and $C_2$ are tangent at two distinct points.
(Note that up to isomorphism there is a
one-dimensional family of such pairs of conics.)
Then one gets~$\Aut^0(\P^2;{C_1\cup C_2})\cong\Gm$ by
Lemma~\ref{lemma:two-tangent-conics}(ii).

Suppose that the conics $C_1$ and $C_2$ are tangent with multiplicity $4$ at a single point.
(Note that up to isomorphism there is a unique pair of conics like this.)
Then one gets~$\Aut^0(\P^2;{C_1\cup C_2})\cong\Ga$ by
Lemma~\ref{lemma:two-tangent-conics}(i).

Finally, suppose that the conics $C_1$ and $C_2$ coincide.
Then
\begin{equation*}
\Aut^0(\P^2;{C_1\cup C_2})\cong\PGL_2(\Bbbk).
\qedhere
\end{equation*}
\end{proof}

\begin{remark}
The Fano threefold $X$ with $\gimel(X)=3.13$ and $\Aut^0(X)\cong\PGL_2(\Bbbk)$
appeared in~\mbox{\cite[Example~2.4]{Prokhorov2013}}, cf. also~\cite{Nakano89}.
\end{remark}

\begin{lemma}\label{lemma:3-24}
Let $X$ be a smooth Fano threefold
with $\gimel(X)=3.24$.
Then one has~$\Aut^0(X)\cong \PGL_{3;1}(\Bbbk)$.
\end{lemma}
\begin{proof}
The threefold $X$
is a blow up of the flag variety $W$ along a
fiber of a projection~$\mbox{$W\to\P^2$}$.
The morphisms $X\to W$ and $X\to\P^2$ are
$\Aut^0(X)$-equivariant, which easily implies the assertion.
\end{proof}

\begin{lemma}\label{lemma:4-7}
Let $X$ be a smooth Fano threefold
with $\gimel(X)=4.7$.
Then
$$
\Aut^0(X)\cong  \GL_2(\Bbbk).
$$
\end{lemma}
\begin{proof}
The threefold $X$
is a blow up of the flag variety $W$ along a disjoint union of two fibers
$C_1$ and $C_2$ of the projections $\pi_1,\pi_2\colon W\to\P^2$,
respectively.
Therefore, we have~\mbox{$\Aut^0(X)\cong\Aut^0(W;C_1\cup C_2)$}.
Let~$\mbox{$P=\pi_1(C_1)$}$ and~$\mbox{$\ell=\pi_1(C_2)$}$,
so that $\ell$ is a line on $\P^2$.
Note that $P\not\in\ell$, since otherwise~\mbox{$C_1\cap C_2\neq\varnothing$}.
The $\Aut^0(X)$-equivariant map~$\pi_1$ provides an isomorphism $\Aut^0(W)\cong\Aut(\P^2)$.
Under this isomorphism,
the subgroup~$\Aut^0(W;C_1\cup C_2)$ is mapped to a subgroup
of $\Aut(\P^2;{\ell\cup P}).$

We claim that $\Aut^0(X)$ is actually isomorphic to $\Aut^0(\P^2;{\ell\cup P})$.
Indeed, let $\sigma$ be an element of the latter group, and let $\hat{\sigma}$
be its (unique) preimage in $\Aut^0(W)$. Then $\hat{\sigma}$
preserves the curve $C_1=\pi_1^{-1}(P)$ and the surface $S=\pi_1^{-1}(\ell)$.
Moreover,
$C_2$ is the unique fiber of $\pi_2$ contained in $S$, and thus $\hat{\sigma}$
preserves $C_2$ as well.

We conclude that $\Aut^0(X)$ is isomorphic to a stabilizer of a disjoint union of a point and a line on $\PP^2$.
Now the assertion follows from Remark~\ref{remark:GL2PGL}.
\end{proof}

\section{Blow ups and double covers of direct products}
\label{section:PnxPm}

In this section we consider smooth Fano threefolds $X$ with
$$
\gimel(X)\in\big\{ 2.2, 2.18, 3.1, 3.3, 3.4, 3.5, 3.8, 3.17, 3.21, 3.22,
 4.1, 4.3, 4.5, 4.8, 4.11, 4.13  \big\}.
$$

\begin{lemma}\label{lemma:double-cover-of-a-product}
Let $Y=\P^{n_1}\times\ldots\times\P^{n_k}$, and let $\theta\colon X\to Y$ be a smooth double cover of~$Y$
branched over a divisor
of multidegree $(d_1,\ldots,d_k)$.
Suppose that $n_i+1\le d_i\le 2n_i$ for every~$i$.
Then the group $\Aut(X)$ is finite.
\end{lemma}
\begin{proof}
Let $\pi_i$ be the projection of $Y$ to the $i$-th factor, and let
$H_i$ be a hyperplane therein.
Then the divisor class
$$
H=\sum\limits_{i=1}^k\pi_i^*(H_i)
$$
defines the Segre embedding.
On the other hand, the branch divisor $Z$ is divisible by~$2$ in $\Pic(Y)$, and thus is not contained
in any effective divisor linearly equivalent to $H$.
Moreover, by adjunction formula the canonical class of $Z$ is numerically
effective, so that~$Z$ is not (uni)ruled.
Thus, Lemma~\ref{lemma:non-ruled} implies that the group $\Aut(Y;Z)$ is finite.
On the other hand, $X$ is a Fano variety. The morphisms
$\pi_i\circ\theta\colon X\to \P^{n_i}$ are extremal contractions, and so they
are $\Aut^0(X)$-equivariant. This implies that $\theta$ is $\Aut^0(X)$-equivariant as well,
so that $\Aut^0(X)$ is a subgroup of $\Aut(Y;Z)$.
\end{proof}

\begin{corollary}\label{corollary:double-cover-of-a-product}
Let $X$ be a smooth Fano threefold
with $\gimel(X)\in\big\{2.2, 3.1\big\}$.
Then the group~$\Aut(X)$ is finite.
\end{corollary}
\begin{proof}
A smooth Fano threefold with $\gimel(X)=2.2$ is a double cover of $\P^1\times\P^2$ with branch divisor
of bidegree $(2,4)$.
A smooth Fano threefold with $\gimel(X)=3.1$ is a double cover of~\mbox{$\P^1\times\P^1\times\P^1$} with branch divisor
of tridegree $(2,2,2)$.
Therefore, the assertion follows from  Lemma~\ref{lemma:double-cover-of-a-product}.
\end{proof}

\begin{lemma}\label{lemma:2-18}
Let $X$ be a smooth Fano threefold
with $\gimel(X)=2.18$.
Then the group $\Aut(X)$ is finite.
\end{lemma}

\begin{proof}
The variety $X$ is a double cover of $\P^1\times\P^2$ branched over a divisor $Z$
of bidegree~$(2,2)$. The natural morphisms from $X$ to $\P^1$ and $\P^2$ are extremal contractions,
which implies that the double cover $\theta\colon X\to\P^1\times\P^2$ is $\Aut^0(X)$-equivariant.
Thus $\Aut^0(X)$ is a subgroup of $\Aut(\P^1\times\P^2;Z)$. By Lemma~\ref{lemma:non-ruled}
the action of $\Aut(\P^1\times\P^2;Z)$ on $Z$ is faithful. Considering the projection of $Z$ to $\P^2$,
we see that $Z$ is a double cover of $\P^2$ branched over a quartic, that is, a (smooth) del Pezzo
surface of degree $2$. Therefore, by Theorem~\ref{theorem:dP} the automorphism group
of $Z$ is finite, and the assertion follows.
\end{proof}

\begin{corollary}\label{corollary:3-4}
Let $X$ be a smooth Fano threefold
with $\gimel(X)=3.4$.
Then the group~$\Aut(X)$ is finite.
\end{corollary}
\begin{proof}
The variety $X$
is a blow up of a smooth Fano variety $Y$ with $\gimel(Y)=2.18$.
Thus the assertion follows from Lemma~\ref{lemma:2-18}.
\end{proof}

\begin{lemma}\label{lemma:3-3}
Let $X$ be a smooth Fano threefold with $\gimel(X)=3.3$.
Then the group $\Aut(X)$ is finite.
\end{lemma}

\begin{proof}
The variety $X$ is a divisor of tridegree $(1,1,2)$ in $\P^1\times\P^1\times\P^2$.
A natural pro\-jec\-ti\-on~$X\to\P^1\times\P^2$ is $\Aut^0(X)$-equivariant. One can check that
$\pi$ is a blow up of~$\P^1\times\P^2$ along a smooth curve $C$ that is a complete intersection of two
divisors of bidegree~$(1,2)$.
This implies that $\Aut^0(X)$ is a subgroup of $\Aut(\P^1\times\P^2;C)$.
Since $C$ is not contained in any effective divisor of bidegree $(1,1)$,
the action of $\Aut(\P^1\times\P^2;C)$ on $C$ is faithful by Lemma~\ref{lemma:non-ruled}.
On the other hand, we see from adjunction formula that
$C$ has genus $3$. Therefore, the automorphism group of $C$ is finite, and the assertion follows.
\end{proof}

The following fact was explained to us by A.\,Kuznetsov.

\begin{lemma}
\label{lemma:3.8-description}
Let $X$ be a Fano threefold with $\gimel(X)=3.8$.
Then $X$ is a blow up of an intersection of divisors of bidegrees $(0,2)$ and $(1,2)$ on $\PP^1\times \PP^2$.
\end{lemma}

\begin{proof}
The variety $X$ can be described as follows. Let $g\colon \mathbb F_1\to \PP^2$ be a blow up
of a point and let $p_1$ and $p_2$ be two natural projections of
$\mathbb{F}_1\times \PP^2$ to $\mathbb F_1$ and $\PP^2$ respectively.
Then~$X$ is a divisor from the linear section $|p_1^*g^*\cO_{\mathbb{P}^2}(1)\otimes p_2^*\cO_{\PP^2}(2)|$.

Let us reformulate this description.
Let $Y=\PP^1\times \PP^2$ and let
$a$ and $b$ be pull backs of~$\cO_{\PP^1}(1)$ and $\cO_{\PP^2}(1)$, respectively.
One has
$$
\PP=\mathbb F_1\times \PP^2\cong\PP_Y\Big(\cO_{Y}\oplus \cO_{Y}(-a)\Big).
$$
Let $h\in |\cO_\PP(1)|$. Then $X$ is a divisor from the linear system $|h+2\pi^*b|$, where $\pi\colon\PP\to Y$ is a projection.

For any rank $2$ vector bundle $E$ over a smooth variety $M$ if  $\phi\colon \PP_M(E)\to M$ is a natural
projection and if $V\in |\cO_{\PP_M(E)}(1)|$, then the natural birational map $V\to M$
is a blow up of $Z=\{s=0\}$, where
$$
s\in H^0\Big(M, E^*\Big)\cong H^0\Big(\PP_M(E),\cO_{\PP_M(E)}(1)\Big).
$$
This means that $\pi_*\cO_{\PP}(h)$ is dual to $\cO_{Y}\oplus \cO_{Y}(-a)$.
The variety $X$ is a blow up of a section of
$$
\pi_*\cO_{\PP}\big(h+2\pi^*b)\cong\Big(\cO_Y\oplus \cO_Y(a)\Big)\otimes \cO_Y\big(2b\big)\cong\cO_Y\big(2b\big)\oplus\cO_Y\big(a+2b\big).
$$
In other words, the threefold $X$ is a blow up of an intersection of divisors of bidegrees~$(0,2)$ and $(1,2)$ on $\PP^1\times \PP^2$.
\end{proof}

\begin{lemma}\label{lemma:blow-up-P1xP2}
Let $1\le n$ and $1\le m\le 2$ be integers.
Let $\mathcal{C}$ be the family of smooth
curves~$C$ of bidegree $(n,m)$ on $\P^1\times\P^2$
which project isomorphically to $\P^2$
(so that the projection of~$C$ to $\P^1$ is an $n$-to-$1$ cover,
and the image of $C$ under the projection to $\P^2$ is
a curve of degree~$m$).
Then up to the action of $\Aut(\P^1\times\P^2)$ the
family of curves in $\mathcal{C}$ has  di\-men\-si\-on~$0$ if~$1\le n\le 2$,
and dimension $2n-5$ if $n\ge 3$.
Furthermore, up to the action of $\Aut(\P^1\times\P^2)$
there is a unique
curve $C_0$ in this family such that $\Aut(\P^1\times\P^2;{C_0})$
is infinite. One has
\begin{itemize}
\item
$\Aut^0(\P^1\times\P^2;{C_0})\cong (\Ga)^2\rtimes(\Gm)^2$
if $m=1$ and $n\ge 2$;

\item $\Aut^0(\P^1\times\P^2;{C_0})\cong \PGL_2(\Bbbk)$ if $m=2$ and $n=1$;

\item $\Aut^0(\P^1\times\P^2;{C_0})\cong \Gm$ if $m=2$ and $n\ge 2$.
\end{itemize}
\end{lemma}

\begin{proof}
Choose a curve $C$ from $\mathcal{C}$.
Let $\pi\colon\P^1\times\P^2\to\P^2$ be the natural projection,
so that $\pi(C)$ is a line if $m=1$ and a smooth conic if $m=2$.
Let $S=\pi^{-1}(\pi(C))$.
One has $S\cong\P^1\times\P^1$, and $C$ is a curve of bidegree $(n,1)$ on~$S$.
Furthermore, the surface $S$ is $\Aut(\P^1\times\P^2;C)$-invariant.
Obviously, the action of the group~$\Aut^0(S;C)$ on $S$ comes from
the restriction of the action of $\Aut(\P^1\times\P^2;C)$.
Therefore, the assertions concerning the number of parameters follow from
Corollary~\ref{corollary:quadric-two-ramification} and
Lemma~\ref{lemma:quadric-curve-dimension}.

If $m=2$, then the action of $\Aut(\P^1\times\P^2;C)$ on $S$ is faithful
by Lemma~\ref{lemma:non-ruled}, so that~$\Aut(\P^1\times\P^2;C)\cong\Aut(S;C)$. Hence
the assertions of the lemma follow from
Corollary~\ref{corollary:quadric-two-ramification} and
Remark~\ref{remark:quadric-two-ramification-small}
in this case.

Now we assume that $m=1$ and $n\ge 2$. Then one has
$$
\Aut(\P^1\times\P^2;C)\cong \Gamma\rtimes\Aut(S;C),
$$
where $\Gamma$ is the pointwise stabilizer of $S$ in $\Aut(\P^1\times\P^2)$.
On the other hand,
$\Gamma$ is isomorphic to the pointwise stabilizer of the line $\pi(S)$ on $\PP^2$, so that
$$
\Gamma\cong(\Ga)^2\rtimes\Gm.
$$
Therefore, the assertion of the lemma follows from
Corollary~\ref{corollary:quadric-two-ramification} and
Remark~\ref{remark:quadric-two-ramification-small}
in this case as well.
\end{proof}

\begin{corollary}\label{corollary:3-5-3-8-3-17-3-21}
Smooth Fano threefolds $X$ with $\gimel(X)=3.5$, $3.8$, $3.17$, and $3.21$
up to isomorphism form a family of dimension $5$, $3$, $0$, and $0$,
respectively. In each of these families, there is a
unique variety $X_0$ with infinite automorphism group.
For~\mbox{$\gimel(X_0)=3.17$}, one has $\Aut^0(X_0)\cong\PGL_2(\Bbbk)$.
For $\gimel(X_0)=3.5$ and $3.8$, one has $\Aut^0(X_0)\cong \Gm$.
For~\mbox{$\gimel(X_0)=3.21$}, one has
$$
\Aut^0(X_0)\cong (\Ga)^2\rtimes(\Gm)^2.
$$
\end{corollary}

\begin{proof}
A variety $X$ with $\gimel(X)=3.5$ is a blow up of a curve $C$ of
bidegree $(5,2)$ on $\P^1\times\P^2$.
A variety $X$ with $\gimel(X)=3.8$ is a blow up of a curve $C$ of bidegree $(4,2)$
on $\P^1\times\P^2$ by Lemma~\ref{lemma:3.8-description}.
A variety $X$ with $\gimel(X)=3.17$ is a blow up of a curve $C$ of
bidegree $(1,2)$ on $\P^1\times\P^2$.
A variety $X$ with $\gimel(X)=3.21$ is a blow up of a curve $C$ of bidegree $(2,1)$
on~$\P^1\times\P^2$.
We conclude that $\Aut^0(X)\cong\Aut^0(\P^1\times\P^2;C)$.
Therefore, everything follows from Lemma~\ref{lemma:blow-up-P1xP2}.
\end{proof}

\begin{remark}[{cf. Lemma~\ref{lemma:3-24}}]
One can use an argument similar to the proof
of Lemma~\ref{lemma:blow-up-P1xP2} to show that there is a unique
smooth Fano threefold $X$ with $\gimel(X)=3.24$, and
one has~$\Aut^0(X)\cong \PGL_{3;1}(\Bbbk)$. Indeed, such a variety
can be obtained as a blow up of~\mbox{$\P^1\times\P^2$} along a curve
of bidegree~\mbox{$(1,1)$}.
\end{remark}

\begin{remark}
In \cite[Theorem~1.1]{Suess14}, it is claimed that there exists a smooth Fano threefold $X$ with $\gimel(X)=3.8$ that admits a faithful action of $(\Gm)^2$.
Actually, this is not the case: the two-dimensional torus cannot faithfully act on this Fano threefold by Corollary~\ref{corollary:3-5-3-8-3-17-3-21}.
This also follows from the fact that every smooth Fano threefold in this family admits a fibration into del Pezzo surfaces of degree $5$,
which is given by the projection $\P^1 \times\P^2\to \P^1$ in Lemma~\ref{lemma:3.8-description}.
These threefolds can be obtained by a blow up of a divisor of bi-degree $(1,2)$ in $\P^2\times\P^2$ along a smooth conic.
By Lemma~\ref{lemma:3.8-description}, this conic is mapped to a conic in $\P^2$ by a projection to the second factor.
In \cite{Suess14}, the description of smooth Fano threefolds $X$ with $\gimel(X)=3.8$ uses different conic: that is mapped to point in $\P^2$ by this projection.
The blow up of such a {\it wrong} conic results in a weak Fano threefold that is not a Fano threefold.
Therefore, we still do not know whether there exists a smooth Fano threefolds $X$ with $\gimel(X)=3.8$
that admits a nontrivial K\"ahler--Ricci soliton or not as stated by \cite[Theorem~6.2]{IltenSuess}.
\end{remark}

Similarly to Lemma~\ref{lemma:blow-up-P1xP2}, one can prove the following.

\begin{lemma}\label{lemma:blow-up-P1xP1xP1}
Let $n$ be a positive integer. Let $\mathcal{C}$ be the family of smooth
curves 
of tridegree $(1,1,n)$ on $\P^1\times\P^1\times\P^1$.
Then up to the action of $\Aut(\P^1\times\P^1\times\P^1)$ the
family $\mathcal{C}$ has dimension $0$ if $1\le n\le 2$
and dimension $2n-5$ if $n\ge 3$.
Furthermore, up to the action of $\Aut(\P^1\times\P^1\times\P^1)$
there is a unique
curve $C_0$ in this family such that~$\Aut(\P^1\times\P^1\times\P^1;{C_0})$
is infinite. One has
\begin{itemize}
\item
$\Aut^0(\P^1\times\P^1\times\P^1;{C_0})\cong\PGL_2(\Bbbk)$
if $n=1$;

\item
$\Aut^0(\P^1\times\P^1\times\P^1;{C_0})\cong\Gm$
if $n\ge 2$.
\end{itemize}
\end{lemma}

\begin{proof}
Left to the reader.
\end{proof}

\begin{corollary}
\label{corollary:4-3-4-13}
Smooth Fano threefolds $X$ with $\gimel(X)=4.3$ or $\gimel(X)=4.13$
up to isomorphism form a family of dimension $0$ or $1$, respectively.
In both cases, there is a
unique variety~$X_0$ with infinite automorphism group.
In both cases one has $\Aut^0(X_0)\cong\Gm$.
\end{corollary}

\begin{proof}
A variety $X$ with $\gimel(X)=4.3$ or $\gimel(X)=4.13$ is a blow up of
$\P^1\times\P^1\times\P^1$ along a curve~$C$ of
tridegree $(1,1,2)$ or $(1,1,3)$, respectively.
We conclude that
$$
\Aut^0(X)\cong\Aut^0(\P^1\times\P^1\times\P^1;C).
$$
Therefore, everything follows from Lemma~\ref{lemma:blow-up-P1xP1xP1}.
\end{proof}

\begin{remark}[{cf. Lemma~\ref{lemma:4-6}}]
One can use Lemma~\ref{lemma:blow-up-P1xP1xP1}
to prove that there is a unique
smooth Fano threefold $X$ with $\gimel(X)=4.6$,
and one has $\Aut^0(X)\cong\PGL_2(\Bbbk)$.
Indeed, such a variety can be obtained as a blow up
of~\mbox{$\P^1\times\P^1\times\P^1$} along a curve of
tridegree~\mbox{$(1,1,1)$}.
\end{remark}

\begin{lemma}\label{lemma:3-22}
Let $X$ be a smooth Fano threefold
with $\gimel(X)=3.22$.
Then
$$
\Aut^0(X)\cong\BB\times\PGL_2(\Bbbk).
$$
\end{lemma}
\begin{proof}
The threefold $X$
is a blow up $\P^1\times\P^2$ along a
conic $Z$ in a fiber of a pro\-jec\-ti\-on~$\P^1\times\P^2\to\P^1$.
The morphisms $X\to\P^1$ and $X\to\P^2$ are $\Aut^0(X)$-equivariant.
Thus the assertion can be deduced from Lemma~\ref{lemma:direct-product}.
\end{proof}

\begin{lemma}\label{lemma:4-1}
Let $X$ be a smooth Fano threefold
with $\gimel(X)=4.1$.
Then the group~$\Aut(X)$ is finite.
\end{lemma}

\begin{proof}
The threefold $X$ is
a divisor of multidegree $(1,1,1,1)$ in
$\P^1\times\P^1\times\P^1\times\P^1$.
Taking a projection $X\to\P^1\times\P^1\times\P^1$,
we see that $X$ is a blow up of $\P^1\times\P^1\times\P^1$
along a smooth curve $C$ that is an intersection of two divisors of
tridegree $(1,1,1)$. Thus, one has
$$
\Aut^0(X)\cong\Aut^0(\P^1\times\P^1\times\P^1;C).
$$
By adjunction formula,
$C$ is an elliptic curve. Consider the projections
$$
\pi_i\colon\P^1\times\P^1\times\P^1\to\P^1,\quad i=1,2,3.
$$
Then each $\pi_i$ is $\Aut^0(\P^1\times\P^1\times\P^1;C)$-equivariant,
and the restriction of $\pi_i$ to $C$ is a double cover
$C\to\P^1$. The group
$\Aut^0(\P^1\times\P^1\times\P^1;C)$ is non-trivial if and only if its action
on one of the $\P^1$'s is non-trivial.
However, the latter must preserve the set of four branch points of the double
cover, which implies that the group~\mbox{$\Aut^0(\P^1\times\P^1\times\P^1;C)$}
is trivial.
\end{proof}

\begin{lemma}\label{lemma:4-5}
Let $X$ be a smooth Fano threefold
with $\gimel(X)=4.5$.
Then
$$
\Aut^0(X)\cong(\Gm)^2.
$$
\end{lemma}
\begin{proof}
The threefold $X$
is a blow up $\P^1\times\P^2$ along a
disjoint union of smooth curves~$Z_1$ and~$Z_2$ of bidegrees $(2,1)$
and $(1,0)$, respectively.
One has
$$
\Aut^0(X)\cong\Aut^0(\P^1\times\P^2;{Z_1\cup Z_2}).
$$
By Lemma~\ref{lemma:blow-up-P1xP2}
one has
$$
\Aut^0(\P^1\times\P^2;{Z_1})\cong (\Ga)^2\rtimes(\Gm)^2,
$$
where the subgroup $\Aut^0(\P^1\times\P^2;{Z_1})\subset\Aut(\P^2)$
acts as a stabilizer
of two points in $\PP^2$, namely, the images $P_1$ and $P_2$
under the projection $\pi_2\colon\P^1\times\P^2\to\P^2$
of the ramification points of the double cover
$Z_1\to\P^1$ given by the
projection~\mbox{$\pi_1\colon\P^1\times\P^2\to\P^1$}.

Consider the action of $\Aut^0(\P^1\times\P^2)$ on $\P^2$
defined via the $\Aut^0(\P^1\times\P^2)$-equivariant
projection~$\pi_2$. It is easy to see that
$\Aut^0(\P^1\times\P^2;{Z_1\cup Z_2})$
is the subgroup in~$\mbox{$\Aut^0(\P^1\times\P^2;{Z_1})$}$
that consists of all elements preserving the point $\pi_2(Z_2)$ on $\P^2$.
Note that $\pi_2(Z_1)$ is the line passing through the points $P_1$ and
$P_2$; the point $\pi_2(Z_2)$ is not contained in this line, since
otherwise one would have $Z_1\cap Z_2\neq\varnothing$.
Therefore, $\Aut^0(\P^1\times\P^2;{Z_1\cup Z_2})$ acts on~$\PP^2$ preserving three points
$P_1$, $P_2$, and $\pi_2(Z_2)$ in general position, so that
\begin{equation*}
\Aut^0(\P^1\times\P^2;{Z_1\cup Z_2})\cong(\Gm)^2.
\qedhere
\end{equation*}
\end{proof}

\begin{lemma}\label{lemma:4-8}
Let $X$ be a smooth Fano threefold
with $\gimel(X)=4.8$.
Then
$$
\Aut^0(X)\cong\BB\times\PGL_2(\Bbbk).
$$
\end{lemma}
\begin{proof}
The threefold $X$
is a blow up $\P^1\times\P^1\times\P^1$ along a
curve $Z$ of bidegree $(1,1)$ in a fiber of a projection $\P^1\times\P^1\times\P^1\to\P^1$.
The morphisms $X\to\P^1$ and $X\to\P^1\times\P^1$ are $\Aut^0(X)$-equivariant.
Thus the assertion can be deduced from Remark~\ref{remark:quadric-two-ramification-small} and Lemma~\ref{lemma:direct-product}.
\end{proof}

\begin{lemma}\label{lemma:4-11}
Let $X$ be a smooth Fano threefold
with $\gimel(X)=4.11$.
Then
$$
\Aut^0(X)\cong\BB\times \PGL_{3;1}(\Bbbk).
$$
\end{lemma}
\begin{proof}
The threefold $X$
is a blow up $\P^1\times\mathbb{F}_1$ along a
$(-1)$-curve $Z$ in a fiber of a projection~$\P^1\times\mathbb{F}_1\to\P^1$.
The morphisms $X\to\P^1$ and $X\to\mathbb{F}_1$ are $\Aut^0(X)$-equivariant.
Moreover, the $(-1)$-curve on $\mathbb{F}_1$ is unique, and hence is invariant with respect to the whole group
$\Aut(\mathbb{F}_1)$.
Thus the assertion can be deduced from Theorem~\ref{theorem:dP}
and Lemma~\ref{lemma:direct-product}.
\end{proof}

\section{Blow up of a quadric along a twisted quartic}
\label{section:2-21}

To deal with the case $\gimel(X)=2.21$, we need some auxiliary information about representations of the group $\SL_2(\Bbbk)$.

\begin{lemma}\label{lemma:SL-6-plus-1-stabilizers}
Let $U_4$ be the (unique) irreducible five-dimensional representation of
the group $\SL_2(\Bbbk)$ (or $\PGL_2(\Bbbk)$), and let $U_0$ be its (one-dimensional) trivial
representation. Consider the projective space
$\P=\P(U_0\oplus U_4)\cong\P^5$, and for any point $R\in\P$ denote by~$\Gamma^0_R$ the connected component of identity in the stabilizer of $R$
in $\PGL_2(\Bbbk)$. The following assertions hold.
\begin{itemize}
\item There is a unique point $Q_0\in\P$ such that
$\Gamma^0_{Q_0}=\PGL_2(\Bbbk)$.

\item Up to the action of $\PGL_2(\Bbbk)$, there is
a unique point $Q_a\in\P$ such that
$\Gamma^0_{Q_a}\cong\Ga$.

\item Up to the action of $\PGL_2(\Bbbk)$, there is
a unique point $Q_{\BB}\in\P$ such that
$\Gamma^0_{Q_{\BB}}\cong\BB$.

\item Up to the action of $\PGL_2(\Bbbk)$, there is
a one-dimensional family of points $Q_m^\xi\in\P$
parameterized by an open subset of the affine line, and
an isolated point $Q_m^{3,1}\in\P$,
such that
$$
\Gamma^0_{Q_m^\xi}\cong\Gamma^0_{Q_m^{3,1}}\cong\Gm.
$$

\item The point $Q_0$ is contained in the
closure of the $\PGL_2(\Bbbk)$-orbit
of the point~$Q_a$, and also in the closure of the family~$Q_m^\xi$.
\end{itemize}
\end{lemma}
\begin{proof}
There
is a $\PGL_2(\Bbbk)$-equivariant (set-theoretical) identification of
$\P$ with a disjoint union $U_4\sqcup\P(U_4)$.
Thus, we have to find the points with infinite stabilizers in $U_4$
and~$\P(U_4)$.

The representation $U_4$ can be identified with the space of
homogeneous polynomials of degree $4$ in two variables $u$ and $v$,
where the action of $\PGL_2(\Bbbk)$ comes from the natural
action of $\SL_2(\Bbbk)$. Obviously, the point $Q_0=0$ is the only one
with stabilizer $\PGL_2(\Bbbk)$. The point $Q_a$ can be chosen from the
$\PGL_2(\Bbbk)$-orbit of the polynomial $u^4$, and the points~$Q_m^\xi$ can be chosen as $\xi^{-1}u^2v^2$.

Now consider the projectivization $\P(U_4)$.
The point $Q_{\BB}$ can be chosen as the equivalence class of the polynomial~$u^4$.
Furthermore, up to the action of $\PGL_2(\Bbbk)$ there are exactly
two points~$Q_m^{3,1}$ and~$Q_m^{2,2}$ in $\P(U_4)$
such that the connected component of identity of their stabilizer
is isomorphic to $\Gm$. These points
can be chosen as classes of the polynomials~$u^3v$ and $u^2v^2$, respectively.
Obviously, the point $Q_m^{2,2}$ is the limit of the points
$Q_m^\xi$ for $\xi\to 0$ (while $Q_0$ is the limit
for $\xi\to \infty$).

Finally, we note that $Q_m^{3,1}$ is not contained in the closure of
the family $Q_m^\xi$ (and is not contained in the union of $\PGL_2(\Bbbk)$-orbits
of the corresponding points), because a polynomial in $u$ and $v$ with a simple
root cannot be a limit of polynomials having only multiple roots.
\end{proof}

\begin{lemma}\label{lemma:2-21}
The following assertions hold.
\begin{itemize}
\item There exists a unique smooth Fano threefold $X$ with $\gimel(X)=2.21$ and
$$
\Aut^0(X)\cong\PGL_2(\Bbbk).
$$

\item There is a unique smooth Fano threefold $X$ with $\gimel(X)=2.21$ and $\Aut^0(X)\cong\Ga$.

\item There is a one-parameter family
of smooth Fano threefolds $X$ with $\gimel(X)=2.21$ and $\Aut^0(X)\cong\Gm$.

\item For all other smooth Fano threefolds $X$ with $\gimel(X)=2.21$, the group $\Aut(X)$ is finite.
\end{itemize}
\end{lemma}

\begin{proof}
The threefold $X$ is a blow up of $Q$ along a twisted quartic $Z$.
Similarly to Lemma~\ref{lemma:2-27},  we conclude that $\Aut(X)$
is the stabilizer of the quadric $Q$ in the subgroup~$\Gamma\cong\PGL_2(\Bbbk)$
of $\Aut(\P^4)$ that acts naturally on $Z$.

Let $U_1$ be a two-dimensional
vector space such that $Z\cong\P^1$ is identified with $\P(U_1)$. Then
$U_1$ has a natural structure of an $\SL_2(\Bbbk)$-representation which
induces the action of~$\Gamma$ on $Z$. The projective space
$\P^4$ is identified with the projectivization of the
$\SL_2(\Bbbk)$-representation $\Sym^4(U_1)$, and the linear system $\mathcal{Q}$
of quadrics in $\P^4$ passing through~$Z$ is identified with the
projectivization of some $\SL_2(\Bbbk)$-invariant six-dimensional vector
subspace $U$ in $\Sym^2(\Sym^4(U_1))$. By~\cite[Exercise 11.31]{FH91},
the latter $\SL_2(\Bbbk)$-representation
splits into irreducible summands
$$
\Sym^2(\Sym^4(U_1))\cong U_0\oplus U_4\oplus U_8,
$$
where $U_i$ is the (unique) irreducible $\SL_2(\Bbbk)$-representation
of dimension $i+1$.
Therefore, one has $U\cong U_0\oplus U_4$.

Let $Q_0$ be the quadric that corresponds to the trivial~$\SL_2(\Bbbk)$-sub\-rep\-re\-sen\-ta\-ti\-on~$U_0\subset U$. Then $Q_0$
is $\PGL_2(\Bbbk)$-invariant, and $\Aut(Q_0;Z)\cong\PGL_2(\Bbbk)$.
We observe that the quadric~$Q_0$ is smooth.
Indeed, suppose that it is singular.
If it is a cone whose vertex is either a point or a line,
then its vertex gives an $\SL_2(\Bbbk)$-subrepresentation
in~\mbox{$\Sym^4(U_1)\cong U_4$}. Since $U_4$ is an irreducible
$\SL_2(\Bbbk)$-representation, we obtain a contradiction.
Similarly, we see that $Q_0$ cannot be reducible or non-reduced.
We conclude that there exists a unique smooth Fano threefold
$X_0$ with $\gimel(X_0)=2.21$ and $\Aut^0(X_0)\cong\PGL_2(\Bbbk)$.

Now we use the results of Lemma~\ref{lemma:SL-6-plus-1-stabilizers}.
They imply all the required assertions provided that we check smoothness
(or non-smoothness)
of the corresponding varieties. For the threefold $X$ with
$\Aut^0(X)\cong\Ga$, and for a general threefold $X$ with
$\Aut^0(X)\cong \Gm$, smoothness follows from the presence of a smooth variety
$X_0$ in the closure of the corresponding family.

It remains to notice that the quadrics $Q_{\BB}$ and $Q_m^{3,1}$ are singular.
Indeed, one can choose homogeneous coordinates $[x:y:z:t:w]$ on $\mathbb{P}^4$
such that the group $\Gm$ acts on $\mathbb{P}^4$ by
\begin{equation}\label{eq:Gm-action-twisted-quartic}
\zeta\colon[x:y:z:t:w]\mapsto[x:\zeta y:\zeta^2 z:\zeta^3 t:\zeta^4 w],
\end{equation}
so that the quadrics $Q_{\BB}$ and $Q_m^{3,1}$ are defined
by equations~\mbox{$y^2=xz$} and~\mbox{$xt=yz$}, respectively.
\end{proof}

\begin{remark}
There is an easy geometric way
to construct the singular $\BB$-invariant quadric~$Q_{\BB}$ that contains the twisted quartic~$Z$.
Indeed, the group $\BB$ has a fixed point~$P$ on
$Z$. A projection from $P$ maps
$\P^4$ to a projective space~$\P^3$ with an action of $\BB$, and maps~$Z$ to a $\BB$-invariant twisted cubic $Z'$ in $\P^3$. Furthermore, there is a $\BB$-fixed point $P'$
on $Z'$. A projection from $P'$ maps
$\P^3$ to a projective plane $\P^2$ with an action of $\BB$, and maps
$Z'$ to a $\BB$-invariant conic in $\P^2$. Taking a cone over this conic with vertex at $P'$,
we obtain a $\BB$-invariant quadric surface in $\P^3$ passing through~$Z'$.
Taking a cone over the latter quadric with vertex at $P$,
we obtain a $\BB$-invariant quadric~$Q_{\BB}$ in $\P^4$ passing through $Z$.
Note that this quadric is singular, and thus it is different from $Q_0$.
By Lemma~\ref{lemma:SL-6-plus-1-stabilizers} every $\BB$-invariant quadric
passing through $Z$ coincides either with $Q_0$ or with~$Q_{\BB}$. This implies that there
does not exist a smooth Fano threefold~$X$ with $\gimel(X)=2.21$ and $\Aut^0(X)\cong\BB$.
\end{remark}

\begin{remark}
The Fano threefold $X$ with $\gimel(X)=2.21$ and $\Aut^0(X)\cong\PGL_2(\Bbbk)$
appeared in~\cite[Example~2.3]{Prokhorov2013}.
\end{remark}

Smooth Fano threefolds with $\gimel(X)=2.21$ such that $\Aut^0(X)\cong\Gm$ can be described very explicitly.
Namely, each such threefold $X$ is a blow up of the smooth quadric threefold~$Q_{\lambda}$ in $\mathbb{P}^4$ that is given by
\begin{equation}\label{eq:quadric-1-parameter}
z^2=\lambda xw+(1-\lambda)yt
\end{equation}
along the twisted quartic curve $Z$ that is given by the parameterization
$$
\big[u^4:u^3v:u^2v^2:uv^3:v^4\big]
$$
where $[u:v]\in\mathbb{P}^1$.
Here $[x:y:z:t:w]$ are homogeneous coordinates on $\mathbb{P}^4$,
the group~$\Gm$ acts on $\mathbb{P}^4$ as in~\eqref{eq:Gm-action-twisted-quartic},
and $\lambda\in\Bbbk$ such that $\lambda\ne 0$ and $\lambda\ne 1$.
Note that~\mbox{$\mathrm{Aut}(Q_{\lambda};Z)$} also contains an additional involution
$$
\iota\colon [x:y:z:t:w]\mapsto[w:t:z:y:x].
$$
Together with $\Gm$, they generate the subgroup $\Gm\rtimes\mathbb{Z}/2\mathbb{Z}$.
The action of this group lifts to $X$.
Observe that there exists an $\mathrm{Aut}(Q_{\lambda};Z)$-commutative diagram
$$
\xymatrix{
&&X\ar@{->}[dr]^{\pi}\ar@{->}[ld]_{\pi^\prime}&&\\%
&Q_{\lambda}\ar@{-->}[rr]^{\phi}&&Q_{\lambda^\prime}&}
$$
such that $\pi$ is a blow up of $Q_{\lambda}$ along the curve $Z$,
the morphism $\pi^\prime$ is a blow up of some smooth quadric $Q_{\lambda^\prime}$ along the curve $Z$,
and $\phi$ is a birational map given by the linear system of quadrics passing through $Z$.
In fact, it follows from~$\mbox{\cite[Remark~2.13]{CheltsovShramovV22}}$ that~$\lambda=\lambda^\prime$ and~$\phi$ can be chosen to be an involution.
In the case when we have~$\mbox{$\mathrm{Aut}^0(Q_{\lambda};Z)\cong\Aut^0(X)\cong\PGL_2(\Bbbk)$}$, this follows from \cite[Example~2.3]{Prokhorov2013}.

\begin{remark}
It was pointed out to us by A.\,Kuznetsov that in the above notation, for
the threefold~$X_{-1/3}$ corresponding to $\lambda=-1/3$ one has~\mbox{$\Aut^0(X_{-1/3})\cong\PGL_2(\Bbbk)$}.
To check this it is enough to write down the condition that the quadric~\eqref{eq:quadric-1-parameter}
is invariant with respect to generators of the Lie algebra of the group~\mbox{$\SL_2(\Bbbk)$}.
\end{remark}

\section{Divisor of bidegree $(1,2)$ on $\mathbb{P}^2\times\mathbb{P}^2$}
\label{section:2.24}

In this section, we consider smooth Fano threefolds $X$ with $\gimel(X)=2.24$.
All of them are divisors of bidegree $(1,2)$ on $\mathbb{P}^2\times\mathbb{P}^2$.

\begin{lemma}
\label{lemma:line and conic}
Let $C$ and $\ell$ be a conic and a line on $\P^2$, respectively.
Suppose that $C$ and~$\ell$ intersect transversally (at two distinct points).
Then $\Aut^0(\P^2;{C\cup\ell})\cong\Gm$.
\end{lemma}

\begin{proof}
Let $P_1$ and $P_2$ be the two points of intersection $C\cap\ell$.
Let $\ell'$ be a tangent line to $C$ at $P_1$.
Choose coordinates $[x:y:z]$ on $\PP^2$ such that
the lines $\ell$ and $\ell'$ are given by $x=0$ and $y=0$,
so that $P_1=[0:0:1]$. We can also assume that $P_2=[0:1:0]$.
In these coordinates the conic $C$ is given by $x^2=yz$.
An automorphism of $\PP^2$ preserving
$\ell$ and $C$ acts on the tangent space
$T_{P_1}(C\cup\ell)\cong\Bbbk^2$ by scaling $x$ and $y$
(considered as coordinates
on $T_{P_1}(C\cup\ell)$), so
it acts in the same way on the initial $\PP^2$. Keeping in mind that the automorphism
should preserve $C=\{x^2=yz\}$, we get the assertion of the lemma.
\end{proof}

\begin{lemma}[{cf. \cite[Theorem~1.1]{Suess14}}]
\label{lemma:2-24}
Any smooth divisor of bidegree $(1,2)$ on $\PP^2\times \PP^2$ has finite automorphic group with two exceptions.
The connected component of identity of the automorphism group for one exception is isomorphic
to~$\Gm$, and for another exception it is isomorphic to~$(\Gm)^2$.
\end{lemma}

\begin{proof}
Let $X$ be a smooth divisor of bidegree $(1,2)$ on $\PP^2\times \PP^2$.
The projection $\phi$ on the first factor provides $X$ a structure of conic
bundle. Its discriminant curve $\Delta$ is a curve of degree $3$
given by vanishing of the discriminant of the quadratic form (whose coordinates
are linear functions on the base of the conic bundle).
The curve $\Delta$ is at worst nodal by Remark~\ref{remark:CB}.

Let us denote coordinates on $\PP^2\times \PP^2=\PP^2_x\times \PP^2_y$ by $[x_0:x_1:x_2]\times [y_0:y_1:y_2]$.
Let the group $\Theta$ be defined
as the maximal subgroup of $\Aut^0(X)$ acting by fiberwise transformations with respect to $\phi$.
There is an exact sequence of groups
$$
1\to \Theta\to\Aut^0(X)\to \Gamma,
$$
where $\Gamma$ acts faithfully on $\PP^2_x$.

We claim that the group $\Theta$ is finite.
Indeed, suppose that it is not.
Let $\ell$ be a general line in $\PP_x^2$,
and let $S$ be the surface $\phi^{-1}(\ell)$.
Then $S$ is $\Theta$-invariant, and the image of $\Theta$ in $\Aut(S)$ is infinite.
On the other hand, the surface $S$ is a smooth del Pezzo surface of degree $5$, so that
$\Aut(S)$ is finite by Theorem~\ref{theorem:dP}.
The obtained contradiction shows that the kernel of the action of the group $\Aut^0(X)$ on $\PP^2_x$ is finite.

The variety $X$ is given by
\begin{equation}
\label{equation:2-24}
x_0Q_0+x_1Q_1+x_2Q_2=0,
\end{equation}
where $Q_i$ are quadratic forms in $y_j$.
Note that they are linearly independent because $X$ is smooth.

The curve $\Delta$ is $\Gamma$-invariant. If $\Delta$ is a smooth cubic, then by Lemma~\ref{lemma:non-ruled}
the group $\Gamma$ is finite.
If $\Delta$ is singular, but irreducible, then by Remark~\ref{remark:CB}
and Lemma~\ref{lemma:singular plane cubic}
the group $\Gamma$
is finite.

Suppose that $\Delta$ is a union of a line and a smooth conic. Then this line intersects the conic transversally
since $\Delta$ is nodal. In particular, $\Gamma$, and thus also $\Aut^0(X)$, is a subgroup of $\Gm$.
Let us get an equation of $X$ in appropriate coordinates.

First, we can assume that the line is given by $x_0=0$ and the intersection points with the conic
have coordinates $[0:1:0]$ and $[0:0:1]$. This means that if we put $x_0=x_1=0$ or $x_0=x_2=0$ in~\eqref{equation:2-24},
we get squares of linear forms, because fibers over nodes of~$\Delta$
are double
lines by Remark~\ref{remark:CB}. Taking these (linearly independent!) linear forms as coordinates
on $\PP^2_y$ one gets $Q_1=y_1^2$ and $Q_2=y_2^2$.

Now let
$$
Q_0=a_0y_0^2+a_1y_0y_1+a_2y_0y_2+a_3y_1^2+a_4y_1y_2+a_5y_2^2.
$$
Let us notice that $a_0\neq 0$ since otherwise the point of $X$ given by $x_0=y_1=y_2=0$ is singular.
So we can assume that $a_0=1$.
Making a linear change of coordinates
$$
y_0=y_0'-\frac{a_1}{2}y_1'-\frac{a_2}{2}y_2',\ \ \  y_1=y_1',\ \ \  y_2=y_2'
$$
and
dropping primes for simplicity we may assume that $a_1=a_2=0$. Making, as above,
a linear change of coordinates
$$
x_0=x_0'-a_3x_1'-a_5x_2',\ \ \  x_1=x_1',\ \ \  x_2=x_2'
$$
and dropping primes again
we may assume that $a_3=a_5=0$. Finally, using scaling we can assume that $a_4=-1$,
since $a_4\neq 0$ because otherwise $\Delta$ is a union of three lines.
Summarizing, in some coordinates $X$ is given by
$$
x_0(y_0^2-y_1y_2)+x_1y_1^2+x_2y_2^2=0.
$$
The action of $\Gm$ from Lemma~\ref{lemma:line and conic} is given by weights
$$
\wt(x_0)=0,\ \ \ \wt(x_1)=2,\ \ \ \wt(x_2)=-2,\ \ \ \wt(y_0)=0,\ \ \ \wt(y_1)=-1,\ \ \ \wt(y_2)=1,
$$
so in this case $\Aut^0(X)\cong\Gm$.

Similarly, if $\Delta$ is a union of three lines in general position, then
$\Gamma$, and thus $\Aut^0(X)$, is a subgroup of $(\Gm)^2$.
Taking the intersection points of the lines by
$[1:0:0]$, $[0:1:0]$, and $[0:0:1]$, one can easily see that $X$ can be given by
$$
x_0y_0^2+x_1y_1^2+x_2y_2^2=0.
$$
The toric structure on $\PP^2_x$ given by the three lines induces the action of $(\Gm)^2$ on $X$,
so in this case $\Aut^0(X)\cong(\Gm)^2$.
\end{proof}

\section{Threefold missing in the Iskovskikh's trigonal list}
\label{section:trigonal}

Let $X$ be a smooth Fano threefold with $\gimel(X)=3.2$.
The threefold $X$ can be described as follows. Let
$$
U=\mathbb{P}\Big(\mathcal{O}_{\mathbb{P}^1\times\mathbb{P}^1}\oplus\mathcal{O}_{\mathbb{P}^1\times\mathbb{P}^1}\big(-1,-1\big)\oplus\mathcal{O}_{\mathbb{P}^1\times\mathbb{P}^1}\big(-1,-1\big)\Big),%
$$
let $\pi\colon U\to\mathbb{P}^1\times\mathbb{P}^1$ be a natural projection, and
let $L$ be a tautological line bundle on $U$.
Then $X$ is a smooth threefold in the linear system $|2L+\pi^*(\mathcal{O}_{\mathbb{P}^1\times\mathbb{P}^1}(2,3))|$.

According to \cite{Is77}, the threefold $X$ is not hyperelliptic, see also \cite{ChPrSh}.
Thus, the linear system $|-K_X|$ gives an embedding $X\hookrightarrow\mathbb{P}^9$.
Note that $X$ is not an intersection of quadrics in $\mathbb{P}^9$.
Indeed, let $\omega\colon X\to\mathbb{P}^1\times\mathbb{P}^1$ be the restriction of the projection $\pi$ to the threefold $X$,
let $\pi_1\colon \mathbb{P}^1\times\mathbb{P}^1\to\mathbb{P}^1$ and $\pi_2\colon \mathbb{P}^1\times\mathbb{P}^1\to\mathbb{P}^1$
be projections to the first and the second factors, respectively.
Let $\phi_1=\pi_1\circ\omega$ and $\phi_2=\pi_2\circ\omega$.
Then a general fiber of the morphism $\phi_1$ is a smooth cubic surface.
This immediately implies that $X$ is not an intersection of quadrics in $\mathbb{P}^9$.

\begin{remark}
\label{remark:Iskovskikh-scroll}
In the notation of \cite[\S2]{Is78}, the threefold $X$ is trigonal.
However, it is missing in the classification of smooth trigonal Fano threefolds obtained in~\mbox{\cite[Theorem~2.5]{Is78}}.
Implicitly, in the proof of this theorem, Iskovskikh showed that $X$ can be obtained as follows.
The scheme intersection of all quadrics  in $\mathbb{P}^9$ containing $X$ is a scroll
$$
R=\mathbb{P}\Big(\mathcal{O}_{\mathbb{P}^1}(2)\oplus\mathcal{O}_{\mathbb{P}^1}(2)\oplus\mathcal{O}_{\mathbb{P}^1}(1)\oplus\mathcal{O}_{\mathbb{P}^1}(1)\Big).%
$$
It is embedded to $\mathbb{P}^9$ by the tautological linear system, which we denote by $M$.
Denote by $F$ a fiber of a general projection $R\to\mathbb{P}^1$.
Then $X$ is contained in the linear sys\-tem~$|3M-4F|$.
In the notation of \cite[\S2]{Reid}, we have $R=\mathbb{F}(2,2,1,1)$,
and $X$ is given by
\begin{multline*}
\alpha^1_2(t_1,t_2)x_1^3+\alpha^2_2(t_1,t_2)x_1^2x_2+\alpha^1_1(t_1,t_2)x_1^2x_3+\alpha^2_1(t_1,t_2)x_1^2x_4+\\
+\alpha^3_2(t_1,t_2)x_1x_2^2+\alpha^3_1(t_1,t_2)x_1x_2x_3+\alpha^4_1(t_1,t_2)x_1x_2x_4+\alpha^1_0(t_1,t_2)x_1x_3^2+\\
+\alpha^2_0(t_1,t_2)x_1x_3x_4+\alpha^3_0(t_1,t_2)x_1x_4^2+\alpha^4_2(t_1,t_2)x_2^3+\alpha^5_1(t_1,t_2)x_2^2x_3+\\
+\alpha^6_1(t_1,t_2)x_2^2x_4+\alpha^4_0(t_1,t_2)x_2x_3^2+\alpha^5_0(t_1,t_2)x_2x_3x_4+\alpha^6_0(t_1,t_2)x_2x_4^2=0,
\end{multline*}
where each $\alpha^i_d(t_1,t_2)$ is a homogeneous polynomial of degree $d$.
Thus, the threefold $X$ is the threefold $T_{11}$ in \cite{ChPrSh}.
Note that the natural projection $R\to\mathbb{P}^1$ restricted to~$X$ gives us the morphism $\phi_1$.
In the proof of \cite[Theorem~2.5]{Is78}, Iskovskikh applied Lefschetz theorem to $X$ to deduce that
its Picard group is cut out by divisors in the scroll~$R$ to exclude this case (this is case $4$ in his proof).
However, the threefold $X$ is not an ample divisor on~$R$, since its restriction to
the subscroll $x_3=x_4=0$ is negative, so that Lefschetz theorem is not applicable here.
\end{remark}

\begin{lemma}
\label{lemma:r-3-n-2}
Let $X$ be a smooth Fano threefold with $\gimel(X)=3.2$.
Then the group~$\mathrm{Aut}(X)$ is finite.
\end{lemma}

\begin{proof}
In the notation of Remark~\ref{remark:Iskovskikh-scroll}, let $S$ be the subscroll given by
$x_3=x_4=0$. Then one has $S\cong\mathbb{P}^1\times\mathbb{P}^1$, and $S$ is contained in $X$.
Furthermore, the normal bundle of $S$ in~$X$ is $\mathcal{O}_{\mathbb{P}^1\times\mathbb{P}^1}(-1,-1)$.
This implies the existence of the following commutative diagram:
\begin{equation}
\label{equation:diagram-scroll}
\xymatrix{
 &&V&& \\
U_{1}\ar@{->}[dd]_{\psi_{1}}\ar@{->}[rru]^{\gamma_{1}}&&&&U_{2}\ar@{->}[dd]^{\psi_{2}}\ar@{->}[llu]_{\gamma_{2}} \\
&&X\ar@{->}[dd]_{\omega}\ar@{->}[lld]_{\phi_1}\ar@{->}[rrd]^{\phi_2}\ar@{->}[uu]_{\alpha}\ar@{->}[ull]^{\beta_{1}}\ar@{->}[urr]_{\beta_{2}}&& \\
\mathbb{P}^1 &&&& \mathbb{P}^1\\
&&\mathbb{P}^1\times\mathbb{P}^1\ar@{->}[llu]^{\pi_1}\ar@{->}[rru]_{\pi_2}&&}%
\end{equation}
Here $U_1$ and $U_2$ are smooth threefolds, the morphisms $\beta_1$ and $\beta_2$ are contractions of the surface $S$ to curves in
these threefolds, the morphism $\alpha$ is a contraction of the surface $S$ to an isolated ordinary double point of the threefold $V$,
the morphism $\phi_2$ is a fibration into del Pezzo surfaces of degree $6$,
the morphism $\psi_1$ is a fibration into del Pezzo surfaces of degree $4$,
and $\psi_2$ is a fibration into quadric surfaces.
By construction, $V$ is a Fano threefold that has one isolated ordinary double point,
and the morphisms $\gamma_1$ and $\gamma_2$ are small resolution of this singular point.
Note also that $-K_{V}^{3}=16$.

Observe that the diagram \eqref{equation:diagram-scroll} is  $\mathrm{Aut}(X)$-equivariant.
In particular, there exists an exact sequence of groups
$$
1\longrightarrow G_{\phi_1}\longrightarrow\mathrm{Aut}(X)\longrightarrow G_{\mathbb{P}^1}\longrightarrow 1,
$$
where $G_{\phi_1}$ is a subgroup in $\mathrm{Aut}(X)$ that leaves a general fiber of $\phi_1$ invariant,
and $G_{\mathbb{P}^1}$ is a subgroup in $\mathrm{Aut}(\mathbb{P}^1)$.
Since a general fiber of $\phi_1$ is a smooth cubic surface,
we see from Theorem~\ref{theorem:dP} that $G_{\phi_1}$ is finite.
Let us show that $G_{\mathbb{P}^1}$ is also finite.

There exists an exact sequence of groups
$$
1\longrightarrow G_{\omega}\longrightarrow\mathrm{Aut}(X)\longrightarrow G_{\mathbb{P}^1\times\mathbb{P}^1}\longrightarrow 1,
$$
where $G_{\omega}$ is subgroup in $\mathrm{Aut}(X)$ that leaves a general fiber of $\omega$ invariant,
and $G_{\mathbb{P}^1\times\mathbb{P}^1}$ is a subgroup in $\mathrm{Aut}(\mathbb{P}^1\times\mathbb{P}^1)$.
If the group $G_{\mathbb{P}^1\times\mathbb{P}^1}$ is finite,
then the group $G_{\mathbb{P}^1}$ is also finite,
because there is a natural surjective homomorphism $G_{\mathbb{P}^1\times\mathbb{P}^1}\to G_{\mathbb{P}^1}$.

To prove the lemma, it is enough to show that $G_{\mathbb{P}^1\times\mathbb{P}^1}$ is finite.
Note that this group preserves the projections $\pi_1$ and $\pi_2$,
because $\phi_1$ is a fibration into cubic surfaces,
while~$\phi_2$ is a fibration into del Pezzo surfaces of degree $6$.
Thus, the group $G_{\mathbb{P}^1\times\mathbb{P}^1}$ is contained in~\mbox{$\mathrm{Aut}^0(\mathbb{P}^1\times\mathbb{P}^1)\cong\mathrm{PGL}_{2}(\Bbbk)\times\mathrm{PGL}_{2}(\Bbbk)$}.

The morphism $\omega$ in \eqref{equation:diagram-scroll} is a standard conic bundle,
and its discriminant curve $\Delta$ is a curve of bidegree $(5,2)$ in $\mathbb{P}^1\times\mathbb{P}^1$.
The curve $\Delta$ is $G_{\mathbb{P}^1\times\mathbb{P}^1}$-invariant.
Moreover, it is reduced and has at most isolated ordinary double points as singularities, see Remark~\ref{remark:CB}.
Furthermore, if $C$ is an irreducible component of the curve $\Delta$, then the intersection number
$$
C\cdot (\Delta-C)
$$
must be even (see \cite[Corollary~2.1]{ChPrSh15}).
This implies, in particular, that no irreducible component of the curve $\Delta$ is a curve of bidegree $(0,1)$.

Let $C$ be an irreducible component of $\Delta$ of bidegree $(a,b)$ with $b\ge 1$, and let
$G_{\mathbb{P}^1\times\mathbb{P}^1}^0$ be the connected component of identity in
the group $G_{\mathbb{P}^1\times\mathbb{P}^1}$.
Then $C$ is $G_{\mathbb{P}^1\times\mathbb{P}^1}^0$-invariant. Moreover, the action of
$G_{\mathbb{P}^1\times\mathbb{P}^1}^0$ on $C$ is faithful; this follows from
Lemma~\ref{lemma:non-ruled}
for $b\ge 2$, and from Lemma~\ref{lemma:SO-stabilizer-hyperplane} for $b=1$.

Assume that there exists an irreducible component $C$ of $\Delta$ such that $C$ has bidegree~$(a,1)$. Then the intersection of $C$ with $\Delta- C$ consists of
$a+5-a=5$ points. Thus, the group $G_{\mathbb{P}^1\times\mathbb{P}^1}^0$ is trivial in this case.

This means that we may assume that there exists an irreducible component $C$ of $\Delta$ such that $C$ has bidegree
$(a,2)$. Suppose that $a\ge 4$. If the normalization of $C$ has positive genus, then
the group $G_{\mathbb{P}^1\times\mathbb{P}^1}^0$ is trivial by Lemma~\ref{lemma:non-ruled}.
Thus we may suppose that $C$ has at least $p_a(C)\ge 3$ singular points.
This again implies that $G_{\mathbb{P}^1\times\mathbb{P}^1}^0$ is trivial,
because the action of the group $G_{\mathbb{P}^1\times\mathbb{P}^1}^0$ lifts to the normalization of $C$
and preserves the preimage of the singular locus of $C$.

We are left with the case when $1\le a\le 3$. Then
the intersection of $C$ with $\Delta - C$ consists of
$2(5-a)\ge 6$ points. Thus, the group $G_{\mathbb{P}^1\times\mathbb{P}^1}^0$ is trivial in this case
as well.
\end{proof}

\begin{remark}
The commutative diagram \eqref{equation:diagram-scroll} is well known to experts.
For instance, it already appeared in the proof of \cite[Theorem~2.3]{Takeuchi},
in the proof of~\mbox{\cite[Proposition~3.8]{JaPeRa07}},
and in the proof of~\mbox{\cite[Lemma~8.2]{ChShUMN}}.
\end{remark}

\section{Remaining cases}
\label{section:remaining-cases}

In this section we consider smooth Fano threefolds $X$ with
$$
\gimel(X)\in\big\{2.1, 2.3, 2.5, 2.6, 2.8, 2.10, 2.11, 2.16, 2.19\big\}.
$$

Theorem~\ref{theorem:Prokhorov} immediately implies the following.

\begin{corollary}\label{corollary:KPS}
Let $X$ be a smooth Fano threefold
with
$$
\gimel(X)\in\big\{2.1, 2.3, 2.5, 2.10, 2.11, 2.16, 2.19\big\}.
$$
Then the group $\Aut(X)$ is finite.
\end{corollary}
\begin{proof}
These varieties are blow ups of smooth
Fano threefolds $Y$ with~\mbox{$\gimel(Y)\in\{1.11, 1.12 , 1.13, 1.14\}$}.
\end{proof}

We will need the following auxiliary fact.

\begin{lemma}\label{lemma:nodal-curve}
Let $\Delta\subset\P^2$ be a nodal curve of degree at least $4$.
Then the group $\Aut(\P^2;\Delta)$ is finite.
\end{lemma}

\begin{proof}
Left to the reader.
\end{proof}

\begin{lemma}\label{lemma:2-6}
Let $X$ be a smooth Fano threefold
with $\gimel(X)=2.6$.
Then the group~$\Aut(X)$ is finite.
\end{lemma}
\begin{proof}
The threefold $X$
is either a divisor of bidegree $(2,2)$ on $\P^2\times\P^2$, or
a double cover of the flag variety $W$ branched over a divisor $Z\sim -K_W$.

Suppose that $X$
is a divisor of bidegree $(2,2)$ on $\P^2\times\P^2$. Let $\phi_i\colon X\to\P^2$, $i=1,2$,
be the natural projections. Then $\phi_i$ is an $\Aut^0(X)$-equivariant
standard conic bundle whose discriminant curve $\Delta_i$ is a sextic.
By Remark~\ref{remark:CB}
the curve $\Delta_i$ is at worst nodal,
so that by Lemma~\ref{lemma:nodal-curve} the group $\Aut(\P^2;{\Delta_i})$ is finite.

We have two exact sequences of groups
$$
1\to\Theta_i\to\Aut^0(X)\to\Gamma_i,
$$
where the action of $\Theta_i$ is fiberwise with respect to $\phi_i$, and $\Gamma_i$ acts faithfully on $\PP^2$ pre\-ser\-ving~$\Delta_i$.
In particular, the groups $\Gamma_i\subset\Aut(\P^2;{\Delta_i})$ are finite. Since the group
$\Aut^0(X)$ is connected, the above sequences imply that
$$
\Theta_1=\Aut^0(X)=\Theta_2.
$$
On the other hand, the intersection $\Theta_1\cap\Theta_2$ acts trivially on $\PP^2\times\PP^2$,
and thus it is a trivial group. This means that the group
$\Aut^0(X)$ is trivial.

Now suppose that $X$ is a double cover of the flag variety $W$ branched over
a divisor~\mbox{$Z\sim -K_W$}.
The divisor class $-\frac{1}{2}K_W$ is very ample, so that by Lemma~\ref{lemma:non-ruled}
the group~$\Aut(W;Z)$ is finite.
On the other hand, both conic bundles $X\to \PP^2$ are~$\Aut^0(X)$-equivariant, so that the double cover $\theta\colon X\to W$ is
$\Aut^0(X)$-equivariant as well.
Thus,~$\Aut^0(X)$ is a subgroup in~$\Aut(W;Z)$, and the assertion follows.
\end{proof}

\begin{lemma}\label{lemma:2-8}
Let $X$ be a smooth Fano threefold
with $\gimel(X)=2.8$.
Then the group~$\Aut(X)$ is finite.
\end{lemma}

\begin{proof}
There exists a commutative diagram
$$
\xymatrix{
X\ar@{->}[d]_{\alpha}\ar@{->}[rr]^{\phi}&&V\ar@{->}[d]^{\beta}\\%
V_7\ar@{->}[rr]^{\pi}&&\P^3}
$$
where $\pi$ is a blow up of a point $O\in\P^3$,
the morphism $\beta$ is a double cover that is branched over an  irreducible quartic surface $S$
that has one isolated double point at $O$, the morphism~$\phi$ is a blow up of the (singular) threefold $V$ at the preimage of the point~$O$ via $\beta$,
and $\alpha$ is a double cover that is branched over the proper transform of the surface~$S$ via $\pi$.
The surface $S$ has singularity of type $\mathbb{A}_1$ or $\mathbb{A}_2$ at the point $O$.
(In the former case, the exceptional divisor of $\phi$ is a smooth quadric surface;
in the latter case, the exceptional divisor of $\phi$ is a quadric cone.)
In both cases, the morphism $\phi$ is a contraction of an extremal ray, so that $\phi$ is $\Aut^0(X)$-equivariant.
Furthermore, the morphism $\beta$ is given by the linear system $|-\frac{1}{2}K_V|$, and thus is also $\Aut^0(X)$-equivariant.
This means that~$\Aut^0(X)$ is isomorphic to a subgroup of~$\Aut(\PP^3;S)$.
Since $S$ is not uniruled, the group~$\mathrm{Aut}(\P^3;S)$ is finite by Lemma~\ref{lemma:non-ruled}.
Hence,
we see that the group $\Aut^0(X)$ is trivial.
\end{proof}

\appendix
\section{The Big Table}
\label{section:table}

In this section we provide
an explicit description of infinite automorphism groups arising
in Theorem~\ref{theorem:main}, and give more details about
the corresponding Fano varieties.
We refer the reader to the end of ~\S\ref{section:intro} for the notation
concerning some frequently appearing groups.
By $S_d$ we denote a smooth del Pezzo surface of degree $d$, except for the quadric surface.

In the first column of Table~\ref{table:big-table},
we give the identifier $\gimel(X)$
for a smooth Fano threefold~$X$. In the second column
we put the anticanonical degree $-K_X^3$.
In the third column we, mainly following~\cite{MM82},
\cite{IP99}, and~\cite{MM04},
give a brief description of the variety.
In the forth column we put
a dimension of the family of Fano threefolds of given type.
In columns $5$ and $6$ we present the groups $\Aut^0(X)$
if they are non-trivial,
and dimensions of families of varieties with the given
group $\Aut^0(X)$. Finally, in the last column we put the reference to
the statement in the text of our paper where the variety is discussed.

\setlength\LTleft{-0.15in}
\setlength\LTright{-0.15in}
\renewcommand{\arraystretch}{1.5}
\small{
\begin{longtable}{|c|c|p{7.5cm}|c|c|c|c|}
\caption{Automorphisms of smooth Fano threefolds}\label{table:big-table}\\
\hline $\gimel$ & $-K^3$ &  Brief description & $\delta$ &
$\Aut^0$ & $\delta^0$ & ref.\\ \cline{5-6}
\hline
\hline \multirow{3}{*}{$1.10$} & \multirow{3}{*}{$22$} & \multirow{3}{*}{
\begin{minipage}[c]{7.3cm} a zero locus of three sections of the rank
$3$ vector bundle $\bigwedge^2\mathcal{Q}$, where
$\mathcal{Q}$ is
the universal quotient bundle on $\mathrm{Gr}(3,7)$\end{minipage}}
& \multirow{3}{*}{$6$} & $\Gm$ & $1$ & \multirow{3}{*}{\ref{theorem:Prokhorov}} \\  \cline{5-6}
& & & & $\Ga$ & $0$ & \\ \cline{5-6}
& & & & $\PGL_2(\Bbbk)$ & $0$ & \\
\hline $1.15$ & $40$ & \begin{minipage}[c]{7.3cm}\vspace{.1cm}
$V_{5}$ that is a section of $\mathrm{Gr}(2,5)\subset\mathbb{P}^9$ by linear subspace of  codimension $3$ \vspace{.1cm}\end{minipage} & $0$ & $\PGL_2(\Bbbk)$ & $0$ & \ref{theorem:Prokhorov}\\
\hline $1.16$ & $54$ & $Q$ that is a hypersurface of degree $2$ in $\mathbb{P}^{4}$ & $0$ & $\PSO_5(\Bbbk)$ & $0$ & \ref{theorem:Prokhorov}\\
\hline $1.17$ & $64$ & $\mathbb{P}^{3}$ & $0$ & $\PGL_4(\Bbbk)$ & $0$ & \ref{theorem:Prokhorov}\\
\hline\hline $2.20$ & $26$ & the blow up of $V_5\subset\mathbb{P}^6$
along a
twisted cubic   & $3$ & $\Gm$ & $0$ & \ref{lemma:V5-cubic} \\
\hline \multirow{3}{*}{$2.21$} & \multirow{3}{*}{$28$} & \multirow{3}{*}{\begin{minipage}[c]{7.3cm} the blow up of $Q\subset\mathbb{P}^4$ along
a
twisted quartic\end{minipage}}   & \multirow{3}{*}{$2$} & $\Gm$ & $1$ & \multirow{3}{*}{\S\ref{section:2-21}}\\\cline{5-6}
& & & & $\PGL_2(\Bbbk)$ & $0$ & \\ \cline{5-6}
& & & & $\Ga$ & $0$ & \\
\hline $2.22$ & $30$ & the blow up of $V_5\subset\mathbb{P}^6$
along a
conic   & $1$ & $\Gm$ & $0$ & \ref{lemma:V5-conic}\\
\hline \multirow{2}{*}{$2.24$} & \multirow{2}{*}{$30$} & \multirow{2}{*}{\begin{minipage}[c]{7.3cm}a divisor on $\mathbb{P}^2\times\mathbb{P}^2$ of bidegree $(1, 2)$\end{minipage}}   &  \multirow{2}{*}{$1$}& $(\Gm)^2$ & $0$ & \multirow{2}{*}{\S\ref{section:2.24}}\\\cline{5-6} & & & & $\Gm$ & $0$ & \\
\hline \multirow{2}{*}{$2.26$} & \multirow{2}{*}{$34$} & \multirow{2}{*}{\begin{minipage}[c]{7.3cm}
the blow up of the threefold $V_5\subset\mathbb{P}^6$ along a line \end{minipage}}   &  \multirow{2}{*}{$0$}& $\Gm$ & $0$ & \multirow{2}{*}{\ref{lemma:V5-line}}\\\cline{5-6} & & & & $\BB$ & $0$ & \\
\hline $2.27$ & $38$ & the blow up of $\mathbb{P}^3$ along a twisted cubic   & $0$ & $\PGL_2(\Bbbk)$ & $0$ & \ref{lemma:2-27}\\
\hline $2.28$ & $40$ & the blow up of $\mathbb{P}^3$ along a plane cubic   & $1$ & $(\Ga)^3\rtimes \Gm$
& $1$ & \ref{lemma:2-28}\\
\hline $2.29$ & $40$ & the blow up of $Q\subset\mathbb{P}^4$ along a conic   & $0$ & $\Gm\times \PGL_2(\Bbbk)$ & $0$ & \ref{lemma:2-29}\\
\hline $2.30$ & $46$ & the blow up of $\mathbb{P}^3$ along a conic   & $0$ & $\PSO_{5;1}(\Bbbk)$ & $0$ & \ref{lemma:2-30-and-2-31}\\
\hline $2.31$ & $46$ & the blow up of $Q\subset\mathbb{P}^4$ along a line   & $0$ &$\PSO_{5;2}(\Bbbk)$ & $0$ & \ref{lemma:2-30-and-2-31}\\
\hline $2.32$ & $48$ & \begin{minipage}[c]{7.3cm}\vspace{.1cm} $W$ that is a divisor on $\mathbb{P}^2\times\mathbb{P}^2$ of bi\-deg\-ree~$(1, 1)$ \vspace{.1cm} \end{minipage} & $0$ & $\PGL_3(\Bbbk)$ & $0$ & \ref{lemma:2-32}\\
\hline $2.33$ & $54$ & the blow up of $\mathbb{P}^3$ along a line  & $0$ & $\PGL_{4;2}(\Bbbk)$ & $0$ & \ref{lemma:2-33-and-2-35}\\
\hline $2.34$ & $54$ & $\mathbb{P}^1\times\mathbb{P}^2$  & $0$ & $\PGL_2(\Bbbk)\times \PGL_3(\Bbbk)$ & $0$ & \ref{corollary:direct-product}\\
\hline $2.35$ & $56$ & $V_7$, the blow up of a point on $\PP^3$
& $0$ &
$\PGL_{4;1}(\Bbbk)$ & $0$ & \ref{lemma:2-33-and-2-35}\\
\hline $2.36$ & $62$ & $\mathbb{P}(\mathcal{O}_{\mathbb{P}^2}\oplus\mathcal{O}_{\mathbb{P}^2}(2))$  & $0$ & $\Aut\big(\P(1,1,1,2)\big)$ & $0$ & \ref{lemma:2-36}\\
\hline
\hline $3.5$ & $20$ & \begin{minipage}[c]{7.3cm} \vspace{.1cm} the blow up of $\mathbb{P}^1\times\mathbb{P}^2$ along a curve $C$ of bidegree $(5, 2)$ such that the com\-po\-si\-ti\-on~$C\hookrightarrow\mathbb{P}^1\times\mathbb{P}^2\to\mathbb{P}^2$ is an embedding \vspace{.1cm}  \end{minipage} & $5$ & $\Gm$ & $0$ & \ref{corollary:3-5-3-8-3-17-3-21}\\
\hline $3.8$ & $24$ & \begin{minipage}[c]{7.3cm}\vspace{.1cm} a divisor in the linear system $\mbox{$|(\alpha\circ\pi_1)^*(\cO_{\P^2}(1))\otimes\pi_2^*(\cO_{\P^2}(2))|$}$, where $\pi_{1}\colon\FF_1\times\P^2\to\FF_1$ and~$\mbox{$\pi_{2}\colon\FF_1\times\P^2\to\P^2$}$ are projections, and~$\alpha\colon\FF_1\to\P^2$ is a blow up of a point \vspace{.1cm} \end{minipage} & $3$ & $\Gm$ & $0$ & \ref{corollary:3-5-3-8-3-17-3-21}\\
\hline $3.9$ & $26$ & \begin{minipage}[c]{7.3cm}\vspace{.1cm} the blow up of a cone over the Veronese surface  $R_4\subset\mathbb{P}^5$ with center in a disjoint union of the vertex and a quartic on $R_4\cong\mathbb{P}^2$ \vspace{.1cm} \end{minipage}  & $6$ & $\Gm$ & $6$ & \ref{corollary:3-9-and-4-2}\\
\hline \multirow{2}{*}{$3.10$} & \multirow{2}{*}{$26$} & \multirow{2}{*}{\begin{minipage}[c]{7.3cm}the blow up of $Q\subset\mathbb{P}^4$ along a disjoint union of two conics\end{minipage}} & \multirow{2}{*}{2} & $\Gm$ & $1$ & \multirow{2}{*}{\ref{lemma:3-10}}\\\cline{5-6}
& & & & $(\Gm)^2$ & $0$ & \\ \cline{5-6}
\hline $3.12$ & $28$ & \begin{minipage}[c]{7.3cm} \vspace{.1cm} the blow up of $\mathbb{P}^3$ along a disjoint union of a line and a twisted cubic \vspace{.2cm} \end{minipage} & $1$ & $\Gm$ & $0$ & \ref{lemma:3-12}\\
\hline \multirow{3}{*}{$3.13$} & \multirow{3}{*}{$30$} & \multirow{3}{*}{\begin{minipage}[c]{7.3cm}   the blow up of
$W\subset\mathbb{P}^2\times\mathbb{P}^2$ along a curve of
bidegree $(2, 2)$ that is mapped by natural projections $\pi_{2}\colon W\to\mathbb{P}^{2}$
and $\pi_{1}\colon W\to\mathbb{P}^{2}$ to irreducible conics \vspace{.1cm}
  \end{minipage}}
& \multirow{3}{*}{$1$} & $\Gm$ & $1$ & \multirow{3}{*}{\ref{lemma:3-13}} \\ \cline{5-6}
& & & & $\Ga$ & $0$ & \\ \cline{5-6}
& & & & $\PGL_2(\Bbbk)$ & $0$ & \\
\hline $3.14$ & $32$ & \begin{minipage}[c]{7.3cm}\vspace{.1cm} the blow up of $\mathbb{P}^3$ along a
disjoint union of a plane cubic curve that is
contained in a  plane
$\Pi\subset\mathbb{P}^{3}$ and a point that is not contained in $\Pi$ \vspace{.2cm} \end{minipage} & $1$ & $\Gm$ & $1$ & \ref{lemma:3-14}\\
\hline $3.15$ & $32$ & \begin{minipage}[c]{7.3cm} \vspace{.1cm} the blow up of $Q\subset\mathbb{P}^4$ along a disjoint union of a line and a conic \vspace{.2cm} \end{minipage} & $0$ & $\Gm$ & $0$ & \ref{lemma:3-15} \\
\hline $3.16$ & $34$ & \begin{minipage}[c]{7.3cm} \vspace{.1cm} the blow up of $V_7$ along a proper
transform via the blow up $\alpha\colon
V_7\to\mathbb{P}^3$ of a twisted cubic
passing through the center of the blow up $\alpha$ \vspace{.1cm} \end{minipage} & $0$ & $\BB$ & $0$ & \ref{lemma:3-16}\\
\hline $3.17$ & $36$ & a divisor on $\mathbb{P}^1\times\mathbb{P}^1\times\mathbb{P}^2$ of tridegree $(1, 1, 1)$  & $0$ & $\PGL_2(\Bbbk)$ & $0$ & \ref{corollary:3-5-3-8-3-17-3-21}\\
\hline $3.18$ & $36$ & \begin{minipage}[c]{7.3cm} \vspace{.1cm} the blow up of $\mathbb{P}^3$ along a disjoint union of a line and a conic  \vspace{.2cm} \end{minipage}
& $0$ & $\BB\times \Gm$ & $0$ & \ref{lemma:3-18}\\
\hline $3.19$ & $38$ & \begin{minipage}[c]{7.3cm} \vspace{.1cm} the blow up of $Q\subset\mathbb{P}^4$ at two non-collinear points \vspace{.2cm} \end{minipage} & $0$ & $\Gm\times \PGL_2(\Bbbk)$ & $0$ & \ref{lemma:3-19}\\
\hline $3.20$ & $38$ & \begin{minipage}[c]{7.3cm} \vspace{.1cm} the blow up of $Q\subset\mathbb{P}^4$ along a disjoint union of two lines \vspace{.2cm} \end{minipage} & $0$ & $\Gm\times \PGL_2(\Bbbk)$ & $0$ & \ref{lemma:3-20}\\
\hline $3.21$ & $38$ & \begin{minipage}[c]{7.3cm} \vspace{.1cm} the blow up of $\mathbb{P}^1\times\mathbb{P}^2$ along a curve of bidegree  $(2, 1)$ \vspace{.1cm} \end{minipage} & $0$ & $(\Ga)^2\rtimes(\Gm)^2$ & $0$ & \ref{corollary:3-5-3-8-3-17-3-21}\\
\hline $3.22$ & $40$ & \begin{minipage}[c]{7.3cm} \vspace{.1cm} the blow up of $\mathbb{P}^1\times\mathbb{P}^2$ along a conic in a fiber of the projection $\mathbb{P}^{1}\times\mathbb{P}^2\to\mathbb{P}^1$ \vspace{.2cm} \end{minipage} & $0$ & $\BB\times\PGL_2(\Bbbk)$ & $0$ & \ref{lemma:3-22}\\
\hline $3.23$ & $42$ & \begin{minipage}[c]{7.3cm} \vspace{.1cm} the blow up of $V_7$ along a proper transform via the blow up $\alpha\colon V_7\to\mathbb{P}^3$ of an irreducible conic passing through the center of the blow up $\alpha$ \vspace{.1cm} \end{minipage}  & $0$ &
$(\Ga)^3\rtimes (\BB\times \Gm)$
& $0$ & \ref{lemma:3-23}\\
\hline $3.24$ & $42$ &
\begin{minipage}[c]{7.3cm} \vspace{.1cm} a blow up of $W$ along a
fiber of a projection~$\mbox{$W\to\P^2$}$
 \vspace{.1cm} \end{minipage}
& $0$ & $\PGL_{3;1}(\Bbbk)$ & $0$ & \ref{lemma:3-24}\\
\hline $3.25$ & $44$ & \begin{minipage}[c]{7.3cm} \vspace{.1cm} the blow up of $\mathbb{P}^3$ along a disjoint union of two lines \vspace{.2cm} \end{minipage} & $0$ &
$\PGL_{(2,2)}(\Bbbk)$
& $0$ & \ref{lemma:3-25}\\
\hline $3.26$ & $46$ &\begin{minipage}[c]{7.3cm} \vspace{.1cm}  the blow up of $\mathbb{P}^3$ with center in a disjoint union of a point and a line \vspace{.1cm} \end{minipage} & $0$ &
$(\Ga)^3\rtimes \left(\GL_2(\Bbbk)\times \Gm\right)$
& $0$ & \ref{lemma:3-26}\\
\hline $3.27$ & $48$ & $\mathbb{P}^1\times\mathbb{P}^1\times\mathbb{P}^1$  & $0$ & $(\PGL_2(\Bbbk))^3$ & $0$ & \ref{corollary:direct-product}\\
\hline $3.28$ & $48$ & $\mathbb{P}^1\times\mathbb{F}_1$  & $0$ & $\PGL_2(\Bbbk)\times \PGL_{3;1}(\Bbbk)$ & $0$ & \ref{corollary:direct-product}\\
\hline $3.29$ & $50$ & \begin{minipage}[c]{7.3cm} \vspace{.1cm} the blow up of the Fano threefold $V_7$
along a line in $E\cong\P^2$, where $E$ is the
exceptional divisor of the
blow up $V_7\to\mathbb{P}^3$ \vspace{.1cm} \end{minipage}  & $0$ &
$\PGL_{4;3,1}(\Bbbk)$
& $0$ & \ref{lemma:3-29}\\
\hline $3.30$ & $50$ & \begin{minipage}[c]{7.3cm} \vspace{.1cm} the blow up
of $V_7$ along a proper transform via the blow up $\alpha\colon V_7\to\mathbb{P}^3$ of a line that passes through the center of the blow up $\alpha$  \vspace{.1cm} \end{minipage}  & $0$ &
$\PGL_{4;2,1}(\Bbbk)$
& $0$ & \ref{lemma:3-30}\\
\hline $3.31$ & $52$ & \begin{minipage}[c]{7.3cm} \vspace{.1cm} the blow up of a cone
over a smooth quadric in $\mathbb{P}^3$ at the vertex \vspace{.2cm} \end{minipage} & $0$ &
$\PSO_{6;1}(\Bbbk)$
& $0$ & \ref{lemma:3-31}\\
\hline
\hline $4.2$ & $28$ & \begin{minipage}[c]{7.3cm} \vspace{.1cm} the blow up of the cone over a smooth quadric $S\subset\mathbb{P}^3$ along a disjoint union of the vertex and an elliptic curve on $S$ \vspace{.1cm} \end{minipage} & $1$ & $\Gm$ & $1$ & \ref{corollary:3-9-and-4-2}\\
\hline $4.3$ & $30$ & \begin{minipage}[c]{7.3cm} \vspace{.1cm} the blow up of $\mathbb{P}^1\times\mathbb{P}^1\times\mathbb{P}^1$ along a curve of tridegree $(1, 1, 2)$ \vspace{.1cm} \end{minipage} & $0$ & $\Gm$ & $0$ & \ref{corollary:4-3-4-13} \\
\hline $4.4$ & $32$ & \begin{minipage}[c]{7.3cm} \vspace{.1cm} the blow up of the smooth Fano threefold~$Y$ with $\gimel(Y)=3.19$ along the proper transform of a conic on the quadric $Q\subset\P^4$\hfill\break that passes through the both centers of the blow up $Y\to Q$ \vspace{.1cm} \end{minipage} & $0$ & $(\Gm)^2$ & $0$ & \ref{lemma:4-4}\\
\hline $4.5$ & $32$ & \begin{minipage}[c]{7.3cm} \vspace{.1cm} the blow up of $\mathbb{P}^1\times\mathbb{P}^2$ along a disjoint union of two irreducible curves of bidegree $(2, 1)$ and $(1, 0)$ \vspace{.1cm} \end{minipage} & $0$ & $(\Gm)^2$ & $0$ & \ref{lemma:4-5} \\
\hline $4.6$ & $34$ & \begin{minipage}[c]{7.3cm} \vspace{.1cm} the blow up of $\mathbb{P}^3$ along a disjoint union of three lines  \vspace{.2cm} \end{minipage} & $0$ & $\PGL_2(\Bbbk)$ & $0$ & \ref{lemma:4-6} \\
\hline $4.7$ & $36$ & \begin{minipage}[c]{7.3cm} \vspace{.1cm} the blow up of $W\subset\mathbb{P}^2\times\mathbb{P}^2$ along a disjoint union of two curves of bidegrees $(0, 1)$ and~$(1, 0)$ \vspace{.2cm} \end{minipage} & $0$ &
$\GL_2(\Bbbk)$ 
& $0$ & \ref{lemma:4-7}\\
\hline $4.8$ & $38$ & \begin{minipage}[c]{7.3cm} \vspace{.1cm} the blow up of $\mathbb{P}^1\times\mathbb{P}^1\times\mathbb{P}^1$ along a curve of tridegree $(0, 1, 1)$ \vspace{.1cm} \end{minipage} & $0$ & $\BB\times\PGL_2(\Bbbk)$ & $0$ & \ref{lemma:4-8}\\
\hline $4.9$ & $40$ & \begin{minipage}[c]{7.3cm} \vspace{.1cm} the blow up of the smooth Fano threefold~$Y$
with $\gimel(Y)=3.25$ along a curve that is contracted
by the blow up $Y \to\mathbb{P}^3$ \vspace{.1cm} \end{minipage}& $0$ & $\PGL_{(2,2);1}(\Bbbk)$ 
& $0$ & \ref{lemma:4-9}\\
\hline $4.10$ & $42$ & $\mathbb{P}^1\times S_7$  & $0$ & $\PGL_2(\Bbbk)\times \BB\times \BB$ & $0$ & \ref{corollary:direct-product}\\
\hline $4.11$ & $44$ & \begin{minipage}[c]{7.3cm} \vspace{.1cm} the blow up of
$\mathbb{P}^1\times\mathbb{F}_1$ along a curve $C\cong\P^1$ such
that $C$ is contained in a fiber
$F\cong\mathbb{F}_{1}$ of the projection
$\P^1\times\mathbb{F}_{1}\to\P^1$ and $C\cdot C=-1$ on $F$ \vspace{.1cm} \end{minipage}& $0$ & $\BB\times \PGL_{3;1}(\Bbbk)$ & $0$ & \ref{lemma:4-11}\\
\hline $4.12$ & $46$ & \begin{minipage}[c]{7.3cm} \vspace{.1cm} the blow up of the smooth Fano threefold
$Y$ with $\gimel(Y)=2.33$ along two curves that are
contracted by the blow up
$Y\to\mathbb{P}^3$ \vspace{.1cm} \end{minipage} & $0$ & $(\Ga)^4\rtimes \left(\GL_2(\Bbbk)\times \Gm\right)$ 
& $0$ & \ref{lemma:4-12}\\
\hline $4.13$ & $26$ & \begin{minipage}[c]{7.3cm} \vspace{.1cm} the blow up of $\mathbb{P}^1\times\mathbb{P}^1\times\mathbb{P}^1$ along a curve of tridegree $(1, 1, 3)$ \vspace{.1cm} \end{minipage} & $1$ & $\Gm$ & $0$ & \ref{corollary:4-3-4-13}\\
\hline
\hline $5.1$ & $28$ & \begin{minipage}[c]{7.3cm} \vspace{.1cm} the blow up of the smooth Fano threefold~$Y$
with $\gimel(Y)=2.29$ along three curves that are
contracted by the
blow up $Y\to Q$ \vspace{.1cm} \end{minipage} & $0$ & $\Gm$ & $0$ & \ref{lemma:5-1}\\
\hline $5.2$ & $36$ & \begin{minipage}[c]{7.3cm} \vspace{.1cm} the blow up of the smooth Fano threefold $Y$
with $\gimel(Y)=3.25$ along two cur\-ves~$C_{1}\ne
C_{2}$ that are contracted by the blow up~$\mbox{$\phi\colon
Y\to\mathbb{P}^3$}$ and that are contained in the
same exceptional divisor of the blow up $\phi$ \vspace{.1cm} \end{minipage} & $0$ & $\Gm\times \GL_2(\Bbbk)$ 
& $0$ & \ref{lemma:5-2} \\
\hline $5.3$ & $36$ & $\mathbb{P}^1\times S_6$  & $0$ & $\PGL_2(\Bbbk)\times(\Gm)^2$ & $0$ & \ref{corollary:direct-product}\\
\hline\hline $6.1$ & $30$ & $\mathbb{P}^1\times S_5$  & $0$ & $\PGL_2(\Bbbk)$ & $0$ & \ref{corollary:direct-product}\\
\hline\hline $7.1$ & $24$ & $\mathbb{P}^1\times S_4$  & $2$ & $\PGL_2(\Bbbk)$ & $2$ & \ref{corollary:direct-product}\\
\hline\hline $8.1$ & $18$ & $\mathbb{P}^1\times S_3$  & $4$ & $\PGL_2(\Bbbk)$ & $4$ & \ref{corollary:direct-product}\\
\hline\hline $9.1$ & $12$ & $\mathbb{P}^1\times S_2$  & $6$ & $\PGL_2(\Bbbk)$ & $6$ & \ref{corollary:direct-product}\\
\hline\hline $10.1$ & $6$ & $\mathbb{P}^1\times S_1$  & $8$ & $\PGL_2(\Bbbk)$ & $8$ & \ref{corollary:direct-product}\\
\hline
\end{longtable}
}

\end{document}